\documentclass[12pt, reqno]{amsart}
\usepackage{mathrsfs}
\usepackage{amssymb,amsthm,amsmath}
\usepackage[numbers,sort&compress]{natbib}
\usepackage{amssymb,amsmath}
\usepackage{amsfonts}
\usepackage{mathrsfs}
\usepackage{latexsym}
\usepackage{amssymb}
\usepackage{amsthm}
\usepackage{bm}
\usepackage{color}
\usepackage{pdfsync}
\usepackage{indentfirst}
\usepackage{appendix}
\date{today}

\usepackage{hyperref}
\hypersetup{hypertex=true,colorlinks=true,linkcolor=blue,anchorcolor=g,citecolor=red}

\usepackage{amsmath}
\usepackage{amsthm}
\usepackage{enumitem}
\allowdisplaybreaks
\newtheorem{remark}{Remark}[section]

\newtheorem{theorem}{Theorem}[section]
\newtheorem{proposition}{Proposition}[section]
\newtheorem{lemma}{Lemma}[section]
\newtheorem{corollary}{Corollary}[section]

\newcommand{\R}{\mathbb R}

\newcommand{\beq}{\begin{equation}}
	\newcommand{\eeq}{\end{equation}}
\newcommand{\ben}{\begin{eqnarray}}
	\newcommand{\een}{\end{eqnarray}}
\newcommand{\beno}{\begin{eqnarray*}}
	\newcommand{\eeno}{\end{eqnarray*}}

\numberwithin{equation}{section}

\def\n{\nabla}

\usepackage[
letterpaper,                 
textheight=8.35in,           
headsep=20pt,                
footskip=36pt,               
marginparwidth=0pt,          
marginparsep=0pt,            
left=0.75in,
right=0.75in
]{geometry}
\begin{document}
	\title[Boundary layer of the 2D chemotaxis-Navier-Stokes system]{Boundary layer analysis for the 2D chemotaxis-Navier-Stokes system with logarithmic sensitivity, Part I: Well-posedness}
	\author{Hui~Wang}
	\address[Hui~Wang]{School of Mathematical Sciences, Dalian University of Technology, Dalian, 116024,  China}
	\email{whd@mail.dlut.edu.cn}
	\author{Wendong~Wang}
	\address[Wendong~Wang]{School of Mathematical Sciences, Dalian University of Technology, Dalian, 116024,  China}
	\email{wendong@dlut.edu.cn}
	\author{Lingling~Zhao}
	\address[Lingling~Zhao]{College of Mathematics, Taiyuan University of Technology, Taiyuan, 030024, China}
	\email{zhaolingling@tyut.edu.cn}
	\date{\today}
	\maketitle

    \renewcommand{\theequation}{\arabic{section}.\arabic{equation}}
	\catcode`@=11 \@addtoreset{equation}{section} \catcode`@=12
	
	\begin{abstract}		
		This is the first part of a two-part work concerning the boundary layer convergence for chemotaxis-Navier-Stokes system in a two-dimensional half-space. In this paper, we investigate the chemotacxis-Navier-Stokes system with the logarithmic singularity under Navier-slip boundary conditions.
        More precisely, we perform an exact asymptotic expansion for the chemotaxis Navier-Stokes system with viscous coefficient $\varepsilon>0$, and obtain partial boundary layer profiles, establishing the well-posedness of the corresponding boundary layer profiles. Specially, we also establish the local well-posedness of solutions to the supercritical chemotaxis Euler equation (with $\varepsilon=0$) by overcoming the difficulty from the disappearance of  diffusion terms.
	\end{abstract}
    \section{Introduction}
    In this paper, we consider the following chemotaxis-Navier-Stokes system with logarithmic sensitivity:
	\begin{eqnarray}\label{eq:C-N-S}
		\left\{
		\begin{split}{}
			&\partial_tn-D\Delta n+u\cdot\nabla n+\chi\nabla\cdot(\frac{n}{c}\nabla c)=0,\;\;&\Omega\times(0,T),\\
			&\partial_tc-\varepsilon\Delta c+u\cdot\nabla c+nc=0,\;\;& \Omega\times(0,T),\\
			&\partial_tu-\varepsilon\Delta u+(u\cdot\nabla)u+\nabla p=n\nabla\phi,\;\;& \Omega\times(0,T),\\
			&\nabla\cdot u=0,\;\;& \Omega\times(0,T),\\
		\end{split}
		\right.
	\end{eqnarray}	
	where $n$, $c$, $u$ and $p$ denote cell density, oxygen concentration, the fluid velocity and the associated pressure respectively, $\Omega=\R^2_+=\{(x,y)\in\R^2|y>0\}$. $D>0$ and $\varepsilon\geq0$ are cell and oxygen diffusion coefficients respectively, and $\chi>0$ is referred to as the chemotactic coefficient measuring the strength of the chemotactic sensitivity, $\phi(x,y)$  is a sufficiently smooth function and independent of t. The system (\ref{eq:C-N-S}) is a deduced model derived from the chemotaxis system proposed by Tuval et al.\cite{TC}, which describes the mathematical mechanism for the oxygen-tactic behavior of aerobic bacteria in fluids. As seen in the experiments from \cite{TC}, aerobic bacteria accumulate from the bulk aqueous region toward the water-air interface. These bacteria impose no directional perturbation on the fluid flow or the oxygen rates into water, which results in the stratification of the adjacent flow field and oxygen transport velocity. Accordingly, the system is equipped with rigorous biological and physical background. For relevant investigations on bacterial dynamics, one may refer to \cite{DCC,WE}. Moreover, the case when $c=0$ means full depletion of the chemical attractant or dissolved oxygen. In this case, the "relative concentration variation rate" cannot be well defined, and model (\ref{eq:C-N-S}) breaks down at $c=0$.
    Some previous work without the singular term can be refered to \cite{BT,CK,JIN,LUI,WIN1,WIN2,WWX,ZHA,ZHE,ZZ,ZZH} and their references, in which the diffusion coefficient $\varepsilon$ was supposed to be negligible (or small). \textbf{A natural question is whether the dynamics of the system (\ref{eq:C-N-S}) change significantly when we consider both the singularity and the case $\varepsilon\ge 0$ simultaneously.} This motivates us to investigate the boundary layer problem for the system (\ref{eq:C-N-S}) with logarithmic sensitivity.

    However, most existing studies on singular equations focus on the case $u$ takes constant. Then the system (\ref{eq:C-N-S}) reduces to the Keller-Segel system with logarithmic sensitivity proposed in \cite{KS1,KS2}, which is widely adopted to simulate chemotactic motion of bacteria or cells \cite{LS1,NO}. A challenge in the corresponding investigations is how to handle the singularity at \(c=0\). Fortunately, this singularity can be eliminated via a Cole-Hopf-type transformation (see \cite{LS,LW}), from which one deduce that
	\begin{eqnarray}\label{v-c transform}
		\begin{split}{}
			v^\varepsilon=-\nabla\ln c=-\frac{\nabla c}{c}.
		\end{split}
	\end{eqnarray}
    The system becomes the singular Keller-Segel system in the absence of fluid as follows:
    \begin{eqnarray}\label{eq:nv varepsilon}
		\left\{
		\begin{split}{}
			&\partial_tn^\varepsilon-\Delta n^\varepsilon-\nabla\cdot(n^\varepsilon v^\varepsilon)=0,\\
			&\partial_tv^\varepsilon-\varepsilon\Delta v^\varepsilon+\nabla(\varepsilon|v^\varepsilon|^2-n^\varepsilon)=0.\\
		\end{split}
		\right.
	\end{eqnarray}
    One may refer to \cite{LL,LP,RWW,WXY,ZZ1} for well-posedness results of this system. The existence of boundary layer solutions to the system (\ref{eq:nv varepsilon}) was established in \cite{HWZ,LZ}. Subsequently, the convergence of solutions with \(\varepsilon>0\) to solutions with \(\varepsilon=0\) for the system (\ref{eq:nv varepsilon}) on a one-dimensional bounded domain was established in \cite{HL,PW}. In 2019, Hou-Wang \cite{HW} carried out a detailed derivation of the boundary layer profile equations for the system (\ref{eq:nv varepsilon}) on the two-dimensional half-plane and established the convergence result from \(\varepsilon>0\) to \(\varepsilon=0\). Further research results can be found in \cite{CHW,CLW,CLW1,LWY,MXW,MXW1}. Up to now, investigations on boundary layers of Keller-Segel systems with logarithmic sensitivity have been well developed.

    Compared with the Keller-Segel system, the singular chemotaxis-Navier-Stokes System (\ref{eq:C-N-S}) has been few investigated over the past decades. Regarding the existence of solutions, Liu \cite{LI} considered the initial-boundary value problem for the full chemotaxis-Navier-Stokes system (\ref{eq:C-N-S}) with constant viscosity. Boundary layer analysis on the system (\ref{eq:C-N-S}) is extremely limited. The only comparable results are established for the general chemotactic model with the term$-\nabla\cdot(n\nabla c)$, whose boundary layer investigations were provided in \cite{HOU,HOU1,LSW}.

    Therefore, in the present work, we investigate the well-posedness and boundary layer behaviors of the system (\ref{eq:C-N-S}). To this end, we first perform the Cole-Hopf transformation given in (\ref{v-c transform}) and set $D=\chi=1$ and $\nabla\phi=e_2=(0,1)$. Then (\ref{eq:C-N-S}) is transferred into the following system
	\begin{eqnarray}\label{eq:nvu varepsilon}
		\left\{
		\begin{split}{}
			&\partial_tn^\varepsilon-\Delta n^\varepsilon+u^\varepsilon\cdot\nabla n^\varepsilon-\nabla\cdot(n^\varepsilon v^\varepsilon)=0,\\
			&\partial_tv^\varepsilon-\varepsilon\Delta v^\varepsilon+\nabla(u^\varepsilon\cdot v^\varepsilon)+\nabla(\varepsilon|v^\varepsilon|^2-n^\varepsilon)=0,\\
			&\partial_tu^\varepsilon-\varepsilon\Delta u^\varepsilon+(u^\varepsilon\cdot\nabla)u^\varepsilon+\nabla p^\varepsilon=n^\varepsilon{e_2},\\
			&\nabla\cdot u^\varepsilon=0,\\
			&n^\varepsilon|_{t=0}=n_{in}, v^\varepsilon|_{t=0}=v_{in}, u^\varepsilon|_{t=0}=u_{in},
		\end{split}
		\right.
	\end{eqnarray}
	where $\varepsilon\geq0$.
	With the Cole-Hopf transformation (\ref{v-c transform}), the curl for $v^\varepsilon$ must be intrinsically free:
	\begin{eqnarray}\label{curl transform}
		\begin{split}{}
			\nabla\times v^\varepsilon=\partial_xv^\varepsilon_2-\partial_yv^\varepsilon_1=0.
		\end{split}
	\end{eqnarray}
	By taking the curl on both sides of the second equation of (\ref{eq:nvu varepsilon}), we can obtain $\partial_t(\nabla\times v^\varepsilon)=\varepsilon\Delta(\nabla\times v^\varepsilon)$. Therefore one can impose that $\nabla\times v_{in}=0$ and $\nabla\times v_{in}|_{\partial\Omega}=0$ for $\varepsilon>0$ to preserve the curl-free condition. Hence the slip boundary condition for (\ref{eq:nvu varepsilon}) is defined as
	\begin{eqnarray}\label{boundary condition}
		\left\{
		\begin{split}{}
			&({\partial_{y}} n^\varepsilon+n^\varepsilon v^\varepsilon_2)|_{y=0}=0,\;\;(v^\varepsilon_2,u^\varepsilon_2)|_{y=0}=0,\;\;(\partial_yv^\varepsilon_1,\partial_yu^\varepsilon_1)|_{y=0}=0,\;\;&\varepsilon>0,\\
            &({\partial_{y}} n^0+n^0 v^0_2)|_{y=0}=0,\;\;u^0_2|_{y=0}=0,&\varepsilon=0.\\
		\end{split}
		\right.
	\end{eqnarray}
    There are few results with respect to the system (\ref{eq:nvu varepsilon}) regarding boundary layer analysis. While the well-posed and the boundary layer results of the system $(\ref{eq:nvu varepsilon})$ remains unresolved, we will  answer  this question in the second part of the two-part work.

    In this article, we employ the matched asymptotic expansion method. Firstly, we introduce the boundary-layer coordinate $z=\frac{y}{\sqrt{\varepsilon}}$ and decompose the solution into an outer-layer solution characterizing the flow far from the boundary and an inner-layer solution capturing the viscous stratification near the boundary. Performing asymptotic expansions on the original equations, we derive the respective equations that the outer and inner solutions satisfy. Then we study the well-posedness of these equations. Accordingly, this paper establishes the boundary-layer asymptotic structure for the chemotaxis-fluid system with small viscosity and the well-posedness of approximate solution. We further demonstrate that the viscous fluid dynamics coupled with bacterial oxygen-taxis both determine the near-wall distributions of fluid velocity, bacterial density and oxygen concentration.
	
	The rest of the paper is organized as follows. In Sect.2, we state some notations, the chemotaxis Euler equation and the main well-posedness results. In Sect.3, we  derive  the detailed process of asymptotic expansion of the boundary layers. In Sect.4, the uniformly higher-order regularity of the Chemotaxis Euler equation equation is estimated. In Sect.5, the desired well-posedness of the solutions to the inner layer equations is presented. 	
	\section{Preliminaries and main results}
    \subsection{Notations}

        $N$ represents the set of positive integers. In the following sections, we use $L^p_{xy}$ to denote the Lebegue spaces $L^p(\R\times\R_+)$ with respect to $(x,y)$, and the corresponding norms are denoted by $\|\cdot\|_{L^p_{xy}}$,   $1\leq p\leq\infty$. Moreover, we use $H^q_{xy}$ to represent the Sobolev spaces $W^{q,2}(\R\times\R_+)$ with respect to $(x,y)$ for $q\in N$, and the corresponding norms be denoted by $\|\cdot\|_{H^q_{xy}}$. For simplicity, denote $\|\cdot\|_{L^q(0,T);Z}$ by $\|\cdot\|_{L^q_TZ}$ for the Space $Z$.

     Let $(\mathcal{H}^{k,m,l},\|\cdot\|_{k,m,l})$ be the anisotropic Sobolev spaces defined as follows: for $k,m,l\in N$,
	\begin{align}\label{H space}
			&\mathcal{H}^{k,m,l}:=\{g(x,z)\in L^2(\mathbb{R}\times\mathbb{R}_+):(1+z^{2k})^\frac{1}{2}\partial^\alpha_x\partial^\gamma_zg(x,z)\in L^2(\mathbb{R}\times\mathbb{R}_+),|\alpha|\leq m,\gamma\in N,|\gamma|\leq l\}
	\end{align}
	with the norm
	\begin{align}\label{normal defined}
		\|g\|^2_{k,m,l}=\sum_{|\alpha|\leq m,|\gamma|\leq l}\int^{+\infty}_{0}\int^{+\infty}_{-\infty}(1+z^{2k})|\partial^\alpha_x\partial^\gamma_zg(x,z)|^2dxdz.
	\end{align}
	In the rest of paper, $C$ is a constant, which may change from line to line, but is independent on the viscosity.
    \subsection{Boundary layer equations}
     This subsection is dedicated to deriving the outer and inner profile equations via formal asymptotic analysis of the solutions $(n^\varepsilon,v^\varepsilon,u^\varepsilon)$ to the systems (\ref{eq:nvu varepsilon}) and (\ref{boundary condition}) for small positive $\varepsilon>0$. Accordingly, the solution $(n^\varepsilon,v^\varepsilon,u^\varepsilon)$ admits the following asymptotic expansion in powers of $\varepsilon$ over the domain $\Omega$, where $j\in N$:
	\begin{align}\label{asymptotic expansions}
		\left\{
	\begin{aligned}{}
		&n^\varepsilon(t,x,y)=\sum\limits_{j\geq0}\varepsilon^{\frac{j}{2}}(n^{I,j}(t,x,y)+n^{B,j}(t,x,\frac{y}{\sqrt{\varepsilon}}))\\
        &v^\varepsilon(t,x,y)=\sum\limits_{j\geq0}\varepsilon^{\frac{j}{2}}(v^{I,j}(t,x,y)+v^{B,j}(t,x,\frac{y}{\sqrt{\varepsilon}}))\\
        &u^\varepsilon(t,x,y)=\sum\limits_{j\geq0}\varepsilon^{\frac{j}{2}}(u^{I,j}(t,x,y)+u^{B,j}(t,x,\frac{y}{\sqrt{\varepsilon}}))\\
		&p^\varepsilon(t,x,y)=\sum\limits_{j\geq0}\varepsilon^{\frac{j}{2}}(p^{I,j}(t,x,y)+p^{B,j}(t,x,\frac{y}{\sqrt{\varepsilon}}))\\
	\end{aligned}
	\right.
	\end{align}
	where $z=\frac{y}{\sqrt{\varepsilon}}$, $(n^{B,j},v^{B,j},u^{B,j},p^{B,j})$ is the boundary layer that is smooth and rapidly decreasing in the last variable.

    To derive the equations for the outer and inner profiles defined in (\ref{asymptotic expansions}), we divide into three main steps.
   Firstly, substituting expansion (\ref{asymptotic expansions}) into $(\ref{eq:nvu varepsilon})_5$ and $(\ref{boundary condition})$ yields the corresponding initial and boundary conditions.
   Secondly, inserting (\ref{asymptotic expansions}) into $(\ref{eq:nvu varepsilon})_4$ and (\ref{curl transform}) gives the expressions for the divergence and curl operators.
   Finally, successively substitution of (\ref{asymptotic expansions}) into the first three equations of $(\ref{eq:nvu varepsilon})$ allows us to derive explicit formulas for each layer profile.
   Combining the above arguments with the asymptotic matching method(see Section \ref{3} for details), we derive the leading-order equations that the outer profiles $(n^{I,0},v^{I,0},u^{I,0},p^{I,0})(t,x,y)$ satisfy
   \begin{eqnarray}\label{eq:n0v0u0}
        \left\{
	\begin{split}{}
		&\partial_tn^{I,0}-\Delta n^{I,0}+u^{I,0}\cdot\nabla n^{I,0}-\nabla\cdot(n^{I,0}v^{I,0})=0,\\
        &\partial_tv^{I,0}+\nabla(u^{I,0}\cdot v^{I,0})-\nabla n^{I,0}=0,\\
        &\nabla\times v^{I,0}=0,\\
        &\partial_tu^{I,0}+u^{I,0}\cdot\nabla u^{I,0}+\nabla p^{I,0}=n^{I,0}{e_2},\\
        &\nabla\cdot u^{I,0}=0,\\
        &\partial_yn^{I,0}+n^{I,0}v^{I,0}_2=0,\;u^{I,0}_2=0,\;\;\; on \;\; y=0,\\
		&(n^{I,0},v^{I,0},u^{I,0})|_{t=0}=(n_0,v_0,u_0).
	\end{split}
        \right.
\end{eqnarray}
   Note that $(\ref{eq:n0v0u0})$ is the system $(\ref{eq:nvu varepsilon})$-$(\ref{boundary condition})$ with $\varepsilon=0$, whose solution is represented by $(n^0,v^0,u^0,p^0)(t,x,y)$. Therefore, it can be inferred that
   \begin{eqnarray*}
	\begin{split}{}
		(n^0,v^0,u^0,p^0)(t,x,y)=(n^{I,0},u^{I,0},v^{I,0},p^{I,0})(t,x,y).
	\end{split}
\end{eqnarray*}
 We infer that the boundary layer $(n^{B,0},v^{B,0}_1,u^{B,0},u^{B,1}_2,p^{B,0},p^{B,1})(t,x,z)$ with $(t,x,z)\in(0,T)\times\R\times\R_+$ and $(n^{I,1},v^{I,1},u^{I,1})(t,x,y)$ with $(t,x,y)\in(0,T)\times\R\times\R_+$ satisfies
\begin{eqnarray*}
	\begin{split}{}
		n^{B,0}(t,x,z)&=0,\;\;v^{B,0}_1(t,x,z)=0,\;\;u^{B,0}(t,x,z)=0,\\
        u^{B,1}_2(t,x,z)&=0,\;\;p^{B,0}(t,x,z)=0,\;\;p^{B,1}(t,x,z)=0,\\
        n^{I,1}(t,x,y)&=0,\;\;v^{I,1}(t,x,y)=0,\;\;u^{I,1}(t,x,y)=0.
	\end{split}
\end{eqnarray*}
   For $v^{B,0}(t,x,z)$ with $(t,x,z)\in[0,T]\times\R\times\R_+$, $v^{B,0}_1=0$ and its second component satisfies
   \begin{eqnarray}\label{eq:vB02}
		\left\{
		\begin{split}{}			&\partial_tv^{B,0}_2+\overline{\partial_yu^{I,0}_2}v^{B,0}_2+\partial_zu^{B,1}_1\overline{v^{I,0}_1}+\partial_zv^{B,1}_1\overline{u^{I,0}_1}+\partial_zv^{B,0}_2z\overline{\partial_yu^{I,0}_2}-\partial_zn^{B,1}=\partial^2_zv^{B,0}_2\\
            &\partial_{x}v^{B,0}_2-\partial_zv^{B,1}_1=0,\\
        &{v^{B,0}_2|_{z=0}=-\overline{v^{I,0}_2}},\\
        &v^{B,0}_2(0,x,z)=0.
		\end{split}
		\right.
	\end{eqnarray}
    $n^{B,1}(t,x,z)$ is determined by $v^{B,0}_2(t,x,z)$ via
   \begin{eqnarray}\label{eq:nB1}
	\begin{split}{}
        &n^{B,1}(t,x,z)=\int^\infty_z\overline{n^{I,0}}v^{B,0}_2(t,x,\eta)d\eta,\\
	\end{split}
    \end{eqnarray}
   where we use $\overline{f}$ to represent $f(t,x,0)$, which applies to the entire article.

    For the first-order inner layer $v^{B,1}(t,x,z)$, its first component $v^{B,1}_1(t,x,z)$ with $(t,x,z)\in[0,T]\times\R\times\R_+$ satisfies
    \begin{align}\label{eq:vB11}
	\begin{cases}	     \partial_tv^{B,1}_1+\partial_x(z\overline{\partial_yu^{I,0}_2})v^{B,0}_2+\overline{\partial_xu^{I,0}_1}v^{B,1}_1  +\partial_xu^{B,1}_1\overline{v^{I,0}_1}\\    +\overline{\partial_xv^{I,0}_1}u^{B,1}_1+\partial_xv^{B,0}_2z\overline{\partial_yu^{I,0}_2}+\partial_xv^{B,1}_1\overline{u^{I,0}_1}-\partial_xn^{B,1}=\partial^2_zv^{B,1}_1,\\
    \partial_{x}v^{B,0}_2-\partial_zv^{B,1}_1=0\\
    \partial_zv^{B,1}_1|_{z=0}=-\overline{\partial_yv^{I,0}_1},\\
    v^{B,1}_1(0,x,z)=0,
	\end{cases}
\end{align}
   and its second component $v^{B,1}_2(t,x,z)$ with $(t,x,z)\in[0,T]\times\R\times\R_+$ satisfies
    \begin{align}\label{eq:vB12}
	\begin{cases}
\partial_tv^{B,1}_2+z\overline{\partial^2_yu^{I,0}_2}v^{B,0}_2+\overline{\partial_yu^{I,0}_1}v^{B,1}_1+\overline{\partial_yu^{I,0}_2}v^{B,1}_2+\partial_zu^{B,1}_1(z\overline{\partial_yv^{I,0}_1}+v^{B,1}_1)+\partial_zu^{B,2}_1\overline{v^{I,0}_1}\\
+\partial_zu^{B,2}_2(\overline{v^{I,0}_2}+v^{B,0}_2)+\overline{\partial_yv^{I,0}_1}u^{B,1}_1
+\partial_zv^{B,0}_2(\frac{1}{2}z^2\overline{\partial^2_yu^{I,0}_2}+\overline{u^{I,2}_2}+u^{B,2}_2)+\partial_zv^{B,1}_1(z\overline{\partial_yu^{I,0}_1}+u^{B,1}_1)\\
+\partial_zv^{B,1}_2z\overline{\partial_yu^{I,0}_2}+\partial_zv^{B,2}_1\overline{u^{I,0}_1}+2\partial_zv^{B,0}_2(\overline{v^{I,0}_2}+v^{B,0}_2)-\partial_zn^{B,2}=\partial^2_zv^{B,1}_2,\\
\partial_{x}v^{B,1}_2-\partial_zv^{B,2}_1=0,\\
v^{B,1}_2|_{z=0}=-\overline{v^{I,1}_2},\\
v^{B,1}_2(0,x,z)=0.
	\end{cases}
\end{align}
For the first-order inner layer $u^{B,1}(t,x,z)$ with $(t,x,z)\in[0,T]\times\R\times\R_+$, $u^{B,1}_1(t,x,z)$ satisfies
	\begin{align}\label{eq:uB11}
		\begin{cases}	\partial_tu^{B,1}_1+u^{B,1}_1\overline{\partial_xu^{I,0}_1}+\overline{u^{I,0}_1}\partial_xu^{B,1}_1+z\overline{\partial_yu^{I,0}_2}\partial_zu^{B,1}_1=\partial^2_zu^{B,1}_1\\
        \partial_xu^{B,1}_1+\partial_zu^{B,2}_2=0,\\
        \partial_zu^{B,1}_1|_{z=0}=-\overline{\partial_yu^{I,0}_1},\\
        u^{B,1}_1(0,x,z)=0.
		\end{cases}
\end{align}
  The second-order inner layer $n^{B,2}(t,x,z)$ satisfies
   \begin{align}\label{eq:nB2}
		n^{B,2}(t,x,z)&=\overline{n^{I,0}}\int^\infty_zv^{B,1}_2(t,x,\eta)d\eta+\overline{v^{I,0}_2}\int^\infty_zn^{B,1}(t,x,\eta)d\eta+\int^\infty_zv^{B,0}_2(t,x,\eta)n^{B,1}(t,x,\eta)d\eta\notag\\
        &\quad-\overline{\partial_yn^{I,0}}\int^\infty_z\int^\infty_\eta\big {(}\xi\partial_\xi v^{B,0}_2(t,x,\xi)+v^{B,0}_2(t,x,\xi)\big{)}d\xi d\eta.
\end{align}
Taking the first component of $v^{B,2}(t,x,z)$ with $(t,x,z)\in[0,T]\times\R\times\R_+$, $v_1^{B,2}(t,x,z)$ satisfies
\begin{align}\label{eq:vB21}
\left\{
\begin{aligned}
&\partial_tv^{B,2}_1+\overline{\partial_xu^{I,0}_1}v^{B,2}_1+z\overline{\partial_x\partial_yu^{I,0}_1}v^{B,1}_1+z\overline{\partial_x\partial_{y}u^{I,0}_2}v^{B,1}_2+\partial_x(\frac{1}{2}z^2\overline{\partial^2_yu^{I,0}_2}+\overline{u^{I,2}_2})v^{B,0}_2\\
    &+\partial_xu^{B,1}_1(z\overline{\partial_yv^{I,0}_1}+v^{B,1}_1)+\partial_xu^{B,2}_1\overline{v^{I,0}_1}+\partial_xu^{B,2}_2(\overline{v^{I,0}_2}+v^{B,0}_2)+\overline{\partial_xv^{I,0}_1}u^{B,2}_1+\overline{\partial_xv^{I,0}_2}u^{B,2}_2\\
    &+z\overline{\partial_x\partial_yv^{I,0}_1}u^{B,1}_1+\partial_xv^{B,0}_2(\frac{1}{2}z^2\overline{\partial^2_yu^{I,0}_2}+\overline{u^{I,2}_2}+u^{B,2}_2)+\partial_xv^{B,1}_1(z\overline{\partial_yu^{I,0}_1}+u^{B,1}_1)\\
    &+\partial_xv^{B,1}_2z\overline{\partial_{y}u^{I,0}_2}+\partial_xv^{B,2}_1\overline{u^{I,0}_1}+\overline{\partial_xv^{I,0}_2}v^{B,0}_2+2\partial_xv^{B,0}_2(\overline{v^{I,0}_2}+v^{B,0}_2)-\partial_xn^{B,2}=\partial^2_zv^{B,2}_1\\
    &\partial_zv^{B,2}_1|_{z=0}=0,\\
        &v^{B,2}_1(0,x,z)=0.
	\end{aligned}
    \right.
\end{align}
     The first component of $u^{B,2}(t,x,z)$ with $(t,x,z)\in[0,T]\times\R\times\R_+$ satisfies
     \begin{align}\label{eq:uB21}
     \left\{
	\begin{aligned}{}  &\partial_tu^{B,2}_1+u^{B,1}_1z\overline{\partial_x\partial_yu^{I,0}_1}+u^{B,2}_1\overline{\partial_xu^{I,0}_1}+u^{B,2}_2\overline{\partial_yu^{I,0}_1}+\overline{u^{I,0}_1}\partial_xu^{B,2}_1+z\overline{\partial_yu^{I,0}_1}\partial_xu^{B,1}_1\\
    &+u^{B,1}_1\partial_xu^{B,1}_1+z\overline{\partial_yu^{I,0}_2}\partial_zu^{B,2}_1+(\frac{1}{2}z^2\overline{\partial^2_yu^{I,0}_2}+\overline{u^{I,2}_2})\partial_zu^{B,1}_1+u^{B,2}_2\partial_zu^{B,1}_1+\partial_xp^{B,2}=\partial^2_zu^{B,2}_1,\\
    &\partial_xu^{B,2}_1+\partial_zu^{B,3}_2=0,\\
    &\partial_zu^{B,2}_1|_{z=0}=0,\\
        &u^{B,2}_1(0,x,z)=0.
	\end{aligned}
    \right.
\end{align}
From the curl and divergence conditions, we can deduce
    \begin{eqnarray}\label{eq:uB22}
		\begin{split}{}
			u^{B,2}_2(t,x,z)=\int^\infty_z\partial_xu^{B,1}_1(t,x,\eta)d\eta,
		\end{split}
	\end{eqnarray}
    and
    \begin{eqnarray}\label{eq:uB32}
		\begin{split}{}
        u^{B,3}_2(t,x,z)=\int^\infty_z\partial_xu^{B,2}_1(t,x,\eta)d\eta.
		\end{split}
	\end{eqnarray}
        The pressure boundary layer $p^{B,2}$ can be defined as
         \begin{eqnarray}\label{eq:pB2}
		\begin{split}{}
			&p^{B,2}(t,x,z)=-\int^\infty_zn^{B,1}(t,x,\eta) d\eta.
		\end{split}
	\end{eqnarray}
    The detailed derivation of $(\ref{eq:n0v0u0})$-$(\ref{eq:pB2})$ will be presented  in Section \ref{3}. Similarly, we can further derive the initial boundary value problems for high-order profiles $(n^{I,j},v^{I,j},u^{I,j},p^{I,j})$ and $(n^{B,j+1},v^{B,j+1}_1,v^{B,j}_2,u^{B,j+1}_1,u^{B,j+2}_2,p^{B,j+1})$ with $j\geq2$.

     \subsection{Main results}Our main results are stated in the following Theorems. Theorem \ref{th2.1} give the local well-posedness of the classical solution of $(\ref{eq:n0v0u0})$.

    To analyze the regularity of solutions to the system (\ref{eq:n0v0u0}), we impose the following compatibility conditions on the initial data for all $\theta\in N$
	\begin{eqnarray*}
	{(M_1)}\left\{
		\begin{split}{}
			&\partial_yn_{in}(x,0)=0,\\
                &\partial_y\partial^\theta_tn_{in}(x,0)=[\Delta\partial_y\partial^{\theta-1}_t n_{in}-\partial_y\partial^{\theta-1}_t(u_{in}\cdot\nabla n_{in})+\partial_y\partial^{\theta-1}_t(\nabla\cdot(n_{in}v_{in}))](x,0),\\
		\end{split}
		\right.
	\end{eqnarray*}
    \begin{eqnarray*}
	{(M_2)}\left\{
		\begin{split}{}
			&v_{in2}(x,0)=0,\;\;\partial_yv_{in1}(x,0)=0,\\
                &\partial_y\partial^{\theta}_tv_{in1}(x,0)=[-\partial_x\partial_y\partial^{\theta-1}_t(u_{in}\cdot v_{in})+\partial_x\partial_y\partial^{\theta-1}_tn_{in}](x,0),\\
                &\partial^{\theta}_tv_{in2}(x,0)=[-\partial_y\partial^{\theta-1}_t(u_{in}\cdot v_{in})+\partial_y\partial^{\theta-1}_tn_{in}](x,0),\\
		\end{split}
		\right.
	\end{eqnarray*}
     and
    \begin{eqnarray*}
	{(M_3)}\left\{
		\begin{split}{}
                &u_{in2}(x,0)=0,\;\;\partial_yu_{in1}(x,0)=0,\\
                &\partial^\theta_tu_{in2}(x,0)=[-\partial^{\theta-1}_t(u_{in}\cdot\nabla u_{02})-\partial_y\partial^{\theta-1}_tp_{in}+\partial^{\theta-1}_tn_{in}](x,0),\\
                &\partial_y\partial^\theta_tu_{in1}(x,0)=[-\partial_y\partial^{\theta-1}_t(u_{in}\cdot\nabla u_{in1})-\partial_{x}\partial_y\partial^{\theta-1}_tp_{in}](x,0).\\
		\end{split}
		\right.
	\end{eqnarray*}
    Where $p_{in}$ be imposed as follows
    \begin{eqnarray*}
	{(M_4)}\left\{
		\begin{split}{}
                &\Delta p_{in}(x,y)=-\nabla\cdot(u_{in}\cdot\nabla u_{in}-n_{in}e_2)(x,y),\\
                &\partial_yp_{in}(x,0)=n_{in}(x,0),\\
                &\Delta\partial^\theta_t p_{in}(x,y)=-\nabla\cdot(\partial^\theta_t(u_{in}\cdot\nabla u_{in})-\partial^\theta_tn_{in}e_2)(x,y),\\
                &\partial_y\partial^\theta_tp_{in}(x,0)=\partial^\theta_tn_{in}(x,0).\\
		\end{split}
		\right.
	\end{eqnarray*}

	\begin{theorem}\label{th2.1}
		{Assume that the initial data $(n_{in},v_{in},u_{in})$ satisfy
        $$(n_{in},v_{in},u_{in})\in H^{m}_{xy}\times H^{m}_{xy}\times H^{m+1}_{xy},\;\;n_{in}\geq0,\;\;\nabla\times v_{in}=0,\;\;m\geq2,$$
        with the compatibility conditions $(M_1)$-$(M_4)$, then there exists the time $T>0$ such that the system \eqref{eq:nvu varepsilon}-\eqref{boundary condition} with $\varepsilon=0$ \rm{(i.e. \eqref{eq:n0v0u0})} has a unique solution $(n^0,v^0,u^0)$ on $[0,T]$ satisfying
        \begin{align*}\label{n^0v^0u^0 Hm}
           \|n^0\|^2_{L^\infty(0,T;H^{m}_{xy})} +\|v^0\|^2_{L^\infty(0,T;H^{m}_{xy})}+\|u^0\|^2_{L^\infty(0,T;H^{m+1}_{xy})}+ \|n^0\|^2_{L^2(0,T;H^{m+1}_{xy})}\leq C,
        \end{align*}
        and
        \begin{align*}
          \|\partial^j_tn^0\|^2_{L^\infty(0,T;H^{m-2j}_{xy})}+\|\partial^j_tv^0\|^2_{L^\infty(0,T;H^{m-2j}_{xy})}&\leq C,\;0<2j\leq m,\notag\\
           \|\partial^j_tu^0\|^2_{L^\infty(0,T;H^{m+1-2j}_{xy})}&\leq C,\;0<2j\leq m+1.
        \end{align*}

          }
	\end{theorem}
       We will provide a detailed proof process for this conclusion in Section {\ref{4}}.

       \begin{remark}
       Compared to the results in \rm{\cite{HW,LSW}}, we study the system \eqref{eq:C-N-S}, which is a coupled Chemotaxis-Navier-Stokes system with logarithmic singularity. By the Cole-Hopf transformation $v^0=-\nabla\ln c^0=-\frac{\nabla c^0}{c^0}$, the system becomes \eqref{eq:nvu varepsilon}-\eqref{boundary condition} with $\varepsilon=0$, whose well-posedness results are established in Theorem \ref{th2.1}.
       During the proof, the mainly difficulty comes from
        the supercritical nonlinear terms. To overcome the difficult terms $\int^\infty_0\int^\infty_{-\infty}\partial^j_tv^0\cdot (\nabla\partial^j_tv^0)\cdot u^0dxdy(j=1,2)$, we apply the curl-free structure of \(v^0\). By the elliptic regularity theory and the mathematical induction method, each order regularity of the solutions is established, which will be adapted to the subsequent boundary-layer regularity estimates in Theorem \ref{th inner layer}.
        \end{remark}

        \begin{theorem}\label{th inner layer}
		{Let $(n_{in},v_{in},u_{in})$ satisfies the assumptions in Theorem \ref{th2.1} and $m\geq 11$. Then there exists the time $T>0$ and for every integer $m_1$ satisfying $7\leq m_1\leq m-4$ such that for any $k>1$\\
        \rm{(\rm{i})} (\ref{eq:vB02}) admits a unique solution $v^{B,0}_2$ satisfying
        \begin{align*}            &\partial^j_tv^{B,0}_2\in{L^\infty(0,T;\mathcal{H}^{k,m_1-j,0})\cap L^2(0,T;\mathcal{H}^{k,m_1-j,1})},\\
        &\partial_tv^{B,0}_2\in{L^\infty(0,T; \mathcal{H}^{k,m_1-l,l-1})},\;v^{B,0}_2\in{L^\infty(0,T;\mathcal{H}^{k,m_1-i,i})},
        \end{align*}
        then for (\ref{eq:nB1}) we obtain that
        \begin{align*}            &\partial^j_tn^{B,1}\in{L^\infty(0,T;\mathcal{H}^{k,m_1-j,1})\cap L^2(0,T;\mathcal{H}^{k,m_1-j,2})},\\
        &\partial_tn^{B,1}\in{L^\infty(0,T; \mathcal{H}^{k,m_1-l,l})},\;n^{B,1}\in{L^\infty(0,T;\mathcal{H}^{k,m_1-i,i+1})},
        \end{align*}
        further, for (\ref{eq:pB2})
        \begin{align*}            &\partial^j_tp^{B,2}\in{L^\infty(0,T;\mathcal{H}^{k,m_1-j,2})\cap L^2(0,T;\mathcal{H}^{k,m_1-j,3})},\\
        &\partial_tp^{B,2}\in{L^\infty(0,T; \mathcal{H}^{k,m_1-l,l+1})},\;p^{B,2}\in{L^\infty(0,T;\mathcal{H}^{k,m_1-i,i+2})},
        \end{align*}
        (ii) (\ref{eq:vB11}) admits a unique solution $v^{B,1}_1$ satisfying
        \begin{align*}            &\partial^j_tv^{B,1}_1\in{L^\infty(0,T;\mathcal{H}^{k,m_1-1-j,0})\cap L^2(0,T;\mathcal{H}^{k,m_1-1-j,1})},\\
        &\partial_tv^{B,1}_1\in{L^\infty(0,T; \mathcal{H}^{k,m_1-1-l,l-1})},\;v^{B,1}_1\in{L^\infty(0,T;\mathcal{H}^{k,m_1-1-i,i})},
        \end{align*}
        (iii) (\ref{eq:vB12}) admits a unique solution $v^{B,1}_2$ satisfying
        \begin{align*}            &\partial^j_tv^{B,1}_2\in{L^\infty(0,T;\mathcal{H}^{k,m_1-2-j,0})\cap L^2(0,T;\mathcal{H}^{k,m_1-2-j,1})},\\
        &\partial_tv^{B,1}_2\in{L^\infty(0,T; \mathcal{H}^{k,m_1-2-l,l-1})},\;v^{B,1}_2\in{L^\infty(0,T;\mathcal{H}^{k,m_1-2-i,i})},
        \end{align*}
        then for (\ref{eq:nB2}) we have
        \begin{align*}            &\partial^j_tn^{B,2}\in{L^\infty(0,T;\mathcal{H}^{k,m_1-2-j,1})\cap L^2(0,T;\mathcal{H}^{k,m_1-2-j,2})},\\
        &\partial_tn^{B,2}\in{L^\infty(0,T; \mathcal{H}^{k,m_1-2-l,l})},\;n^{B,2}\in{L^\infty(0,T;\mathcal{H}^{k,m_1-2-i,i+1})},
        \end{align*}
        (iv) (\ref{eq:uB11}) admits a unique solution $u^{B,1}_1$ satisfying
        \begin{align*}            &\partial^j_tu^{B,1}_1\in{L^\infty(0,T;\mathcal{H}^{k,m_1-j,0})\cap L^2(0,T;\mathcal{H}^{k,m_1-j,1})},\\
            &\partial_tu^{B,1}_1\in{L^\infty(0,T; \mathcal{H}^{k,m_1-l,l-1})},\;u^{B,1}_1\in{L^\infty(0,T;\mathcal{H}^{k,m_1-i,i})},
        \end{align*}
        then for (\ref{eq:uB22}) we obtain that
        \begin{align*}            &\partial^j_tu^{B,2}_2\in{L^\infty(0,T;\mathcal{H}^{k,m_1-1-j,1})\cap L^2(0,T;\mathcal{H}^{k,m_1-1-j,2})},\\
        &\partial_tu^{B,2}_2\in{L^\infty(0,T; \mathcal{H}^{k,m_1-1-l,l})},\;u^{B,2}_2\in{L^\infty(0,T;\mathcal{H}^{k,m_1-1-i,i+1})},
        \end{align*}
        (v) (\ref{eq:vB21}) admits a unique solution $v^{B,2}_1$ satisfying
        \begin{align*}            &\partial^j_tv^{B,2}_1\in{L^\infty(0,T;\mathcal{H}^{k,m_1-3-j,0})\cap L^2(0,T;\mathcal{H}^{k,m_1-3-j,1})},\\
        &\partial_tv^{B,2}_1\in{L^\infty(0,T; \mathcal{H}^{k,m_1-3-l,l-1})},\;v^{B,2}_1\in{L^\infty(0,T;\mathcal{H}^{k,m_1-3-i,i})},
        \end{align*}
        (vi) (\ref{eq:uB21}) admits a unique solution $u^{B,2}_1$ satisfying
        \begin{align*}            &\partial^j_tu^{B,2}_1\in{L^\infty(0,T;\mathcal{H}^{k,m_1-2-j,0})\cap L^2(0,T;\mathcal{H}^{k,m_1-2-j,1})},\\
        &\partial_tu^{B,2}_1\in{L^\infty(0,T; \mathcal{H}^{k,m_1-2-l,l-1})},\;u^{B,2}_1\in{L^\infty(0,T;\mathcal{H}^{k,m_1-2-i,i})},
        \end{align*}
        then for (\ref{eq:uB32}) we obtain that
        \begin{align*}            &\partial^j_tu^{B,3}_2\in{L^\infty(0,T;\mathcal{H}^{k,m_1-3-j,1})\cap L^2(0,T;\mathcal{H}^{k,m_1-3-j,2})},\\
        &\partial_tu^{B,3}_2\in{L^\infty(0,T; \mathcal{H}^{k,m_1-3-l,l})},\;u^{B,3}_2\in{L^\infty(0,T;\mathcal{H}^{k,m_1-3-i,i+1})},
        \end{align*}
      where $j=0,1,2$, $l=2,3$ and $i=1,2,3,4$.
		}
	\end{theorem}
    The detailed derivation is presented in Section \ref{5}.
    \begin{remark}
    	Theorem \ref{th inner layer} establishes the regularity for some inner profiles required in the asymptotic expansion up to the order considered in this paper.
        \begin{itemize}
        \item {\bf Difficulty. }
        The principal difficulty in the proof is that the boundary-layer equations contain unbounded normal transport term just like $z\,a(t,x)\partial_z f$ and nonhomogeneous boundary data. To overcome these difficulties, we apply the weighted polynomial anisotropic Sobolev estimates and the boundary homogenization by constructing an auxiliary function.
        \item {\bf Layered structure.} Since these profiles are interrelated and provide the source terms needed at the next level, for example the equation of \(v_1^{B,1}\) involves \(u_1^{B,1}\) and \(v_2^{B,0}\), the estimates of these profiles should be constructed in a certain order, beginning with the estimates of \(u_1^{B,1}\) and \(v_2^{B,0}\), which indicate ones of \(u_2^{B,2}\), \(n^{B,1}\) and \(p^{B,2}\), and then following them successively are \(v_1^{B,1}\), \(u_1^{B,2}\), \(v_2^{B,1}\) and \(v_1^{B,2}\), which indicate the regularity estimates for \(u_2^{B,3}\) and \(n^{B,2}\).
        \item {\bf Dependency on the leading outer solutions. }The restriction $7\leq m_1\leq m-4$ quantifies the finite regularity transferring from the leading outer solutions to the inner profiles. In particular, the theorem shows that sufficiently high tangential regularity supplies the normal regularity needed for the boundary-layer expansion. These weighted estimates provide the analytical foundation for the rigorous control of the approximate solution and the subsequent boundary-layer convergence analysis.
        \end{itemize}
    \end{remark}

     In the subsequent section, we perform an asymptotic expansion for system $(\ref{eq:nvu varepsilon})$ and derive the leading-order outer-layer system (\ref{eq:n0v0u0}), which provides the basis for the ensuing regularity estimates.
     \section{Asymptotic Expansion of Equations}\label{3}
    In this part, we employ the method of the matched asymptotic expansions (\ref{asymptotic expansions}) to derive the partial boundary layer profiles in detail with $j\leq2$ for the System $(\ref{eq:nvu varepsilon})$.

   \textbf{Step 1. Initial and boundary conditions.} Substituting $(\ref{asymptotic expansions})$ into the initial conditions in $(\ref{eq:nvu varepsilon})$, we can obtain
   \begin{equation}\label{initial asymptotic expansions}
   \left\{
       \begin{split}
        &n^\varepsilon(0,x,y)=\sum\limits_{j\geq0}\varepsilon^{\frac{j}{2}}(n^{I,j}(0,x,y)+n^{B,j}(0,x,\frac{y}{\sqrt{\varepsilon}})),\\       &v^\varepsilon(0,x,y)=\sum\limits_{j\geq0}\varepsilon^{\frac{j}{2}}(v^{I,j}(0,x,y)+v^{B,j}(0,x,\frac{y}{\sqrt{\varepsilon}})),\\        &u^\varepsilon(0,x,y)=\sum\limits_{j\geq0}\varepsilon^{\frac{j}{2}}(u^{I,j}(0,x,y)+u^{B,j}(0,x,\frac{y}{\sqrt{\varepsilon}})).
       \end{split}
       \right.
   \end{equation}

    Since $n^{B,j}$, $v^{B,j}$ and $u^{B,j}$ decay to zero exponentially with $z=\frac{y}{\sqrt{\varepsilon}}\rightarrow+\infty$, we can get
  \begin{equation}\label{initial zero layer}
	\begin{split}{}
		n^{I,0}(0,x,y)&=n_{in}(x,y),\;\;n^{B,0}(0,x,z)=0,\\
        v^{I,0}(0,x,y)&=v_{in}(x,y),\;\;v^{B,0}(0,x,z)=0,\\
        u^{I,0}(0,x,y)&=u_{in}(x,y),\;\;u^{B,0}(0,x,z)=0,\\
	\end{split}
\end{equation}
and for $j\geq1$
\begin{equation}\label{initial other layer}
	\begin{split}{}
		n^{I,j}(0,x,y)&=n^{B,j}(0,x,z)=0,\\
        v^{I,j}(0,x,y)&=v^{B,j}(0,x,z)=0,\\
        u^{I,j}(0,x,y)&=u^{B,j}(0,x,z)=0.\\
	\end{split}
\end{equation}
In order to obtain the boundary conditions of each order for $j\geq0$, we substituted $(\ref{asymptotic expansions})$ into $(\ref{boundary condition})$ to obtain
\begin{align*}
	\left\{
	\begin{aligned}{}        &\partial_yn^\varepsilon(t,x,0)=\sum\limits_{j\geq0}\varepsilon^{\frac{j}{2}}(\partial_yn^{I,j}(t,x,0)+\varepsilon^{-\frac{1}{2}}\partial_zn^{B,j}(t,x,0)=0,\\        &v^\varepsilon_2(t,x,0)=\sum\limits_{j\geq0}\varepsilon^{\frac{j}{2}}(v^{I,j}_2(t,x,0)+v^{B,j}_2(t,x,0)=0,\\		&u^\varepsilon_2(t,x,0)=\sum\limits_{j\geq0}\varepsilon^{\frac{j}{2}}(u^{I,j}_2(t,x,0)+u^{B,j}_2(t,x,0)=0,\\		&\partial_yv^\varepsilon_1(t,x,0)=\sum\limits_{j\geq0}\varepsilon^{\frac{j}{2}}(\partial_yv^{I,j}_1(t,x,0)+\varepsilon^{-\frac{1}{2}}\partial_zv^{B,j}_1(t,x,0)=0,\\        &\partial_yu^\varepsilon_1(t,x,0)=\sum\limits_{j\geq0}\varepsilon^{\frac{j}{2}}(\partial_yu^{I,j}_1(t,x,0)+\varepsilon^{-\frac{1}{2}}\partial_zu^{B,j}_1(t,x,0)=0,
	\end{aligned}
	\right.
\end{align*}
 and note that
\begin{align}\label{layer boundary}
&\partial_zn^{B,0}(t,x,0)=0,\;\;\sum\limits_{j\geq0}\varepsilon^{\frac{j}{2}}\partial_yn^{I,j}(t,x,0)+\sum\limits_{j\geq0}\varepsilon^{\frac{j}{2}}\partial_zn^{B,j+1}(t,x,0)=0,\notag\\
        &\;\sum\limits_{j\geq0}\varepsilon^{\frac{j}{2}}v^{I,j}_2(t,x,0)+\sum\limits_{j\geq0}\varepsilon^{\frac{j}{2}}v^{B,j}_2(t,x,0)=0,\notag\\
        &\sum\limits_{j\geq0}\varepsilon^{\frac{j}{2}}u^{I,j}_2(t,x,0)+\sum\limits_{j\geq0}\varepsilon^{\frac{j}{2}}u^{B,j}_2(t,x,0)=0,\\
        &\partial_zv^{B,0}_1(t,x,0)=0,\;\sum\limits_{j\geq0}\varepsilon^{\frac{j}{2}}\partial_yv^{I,j}_1(t,x,0)+\sum\limits_{j\geq0}\varepsilon^{\frac{{j}}{2}}\partial_zv^{B,j+1}_1(t,x,0)=0,\notag\\
        &\partial_zu^{B,0}_1(t,x,0)=0,\;\sum\limits_{j\geq0}\varepsilon^{\frac{j}{2}}\partial_yu^{I,j}_1(t,x,0)+\sum\limits_{j\geq0}\varepsilon^{\frac{{j}}{2}}\partial_zu^{B,j+1}_1(t,x,0)=0.\notag\notag
\end{align}
\textbf{Step 2. Equations for $(n^{I,j},v^{I,j},u^{I,j})$ and $(n^{B,j},v^{B,j},u^{B,j})$.} Substituting $(\ref{asymptotic expansions})$ into  $(\ref{boundary condition})$ and $(\ref{eq:nvu varepsilon})_4$, due to $v^{B,j}$ and $u^{B,j}$ are fast decay when $z\rightarrow+\infty$, then {the outflow $v^{I,j}$ and $u^{I,j}$ satisfy}
\begin{equation}\label{outer layer curl and div}
	\begin{split}{}
		\nabla\times v^{I,j}=0,\;\;\nabla\cdot u^{I,j}=0,\\
	\end{split}
\end{equation}
thus the internal flow $v^{B,j}$ and $u^{B,j}$ {satisfy}
\begin{equation}\label{inner layer curl and div}
	\begin{split}{}
        &\partial_z v^{B,0}_1=0,\;\;\;\sum\limits_{j\geq0}\varepsilon^{\frac{j}{2}}\partial_x v^{B,j}_2-\sum\limits_{j\geq0}\varepsilon^{\frac{j}{2}}\partial_z v^{B,j+1}_1=0,\\
		&\partial_z u^{B,0}_2=0,\;\;\;\sum\limits_{j\geq0}\varepsilon^{\frac{j}{2}}\partial_x u^{B,j}_1+\sum\limits_{j\geq0}\varepsilon^{\frac{j}{2}}\partial_z u^{B,j+1}_2=0.\\
	\end{split}
\end{equation}
Due to $v^{B,0}_1$ and $u^{B,0}_2$ is decay exponentially fast as $z\rightarrow+\infty$, {it then follows that}
\begin{equation}\label{vB01=0,uB02=0}
	\begin{split}{}
		 v^{B,0}_1\equiv0,\; u^{B,0}_2\equiv0.
	\end{split}
\end{equation}
Then we can deduce the boundary condition $u^{I,0}_2(t,x,0)=0$.\\

Substituting $(\ref{asymptotic expansions})$ into equation $(\ref{eq:nvu varepsilon})_1$, we can obtain the expansion equation of $n^\varepsilon$
\begin{align}\label{n varepsilon expansion}
		&\sum\limits_{j\geq0}\varepsilon^{\frac{j}{2}}\partial_t(n^{I,j}+n^{B,j})+\sum\limits_{j\geq0}\varepsilon^{\frac{j}{2}}\sum\limits^j_{k=0}(u^{I,k}+u^{B,k})\cdot\nabla n^{I,j-k}\notag\\
		&+\sum\limits_{j\geq0}\varepsilon^{\frac{j}{2}}\sum\limits^j_{k=0}(u^{I,k}_1+u^{B,k}_1)\partial_x n^{B,j-k}+\sum\limits_{j\geq0}\varepsilon^{\frac{j-1}{2}}\sum\limits^j_{k=0}(u^{I,k}_2+u^{B,k}_2)\partial_z n^{B,j-k}\notag\\
        &-\sum\limits_{j\geq0}\varepsilon^{\frac{j}{2}}\sum\limits^j_{k=0}(n^{I,k}+n^{B,k})\nabla\cdot v^{I,j-k}-\sum\limits_{j\geq0}\varepsilon^{\frac{j}{2}}\sum\limits^j_{k=0}(n^{I,k}+n^{B,k})\partial_x v^{B,j-k}_1\\
        &-\sum\limits_{j\geq0}\varepsilon^{\frac{j-1}{2}}\sum\limits^j_{k=0}(n^{I,k}+n^{B,k})\partial_z v^{B,j-k}_2-\sum\limits_{j\geq0}\varepsilon^{\frac{j}{2}}\sum\limits^j_{k=0}(v^{I,k}+v^{B,k})\cdot\nabla n^{I,j-k}\notag\\
        &-\sum\limits_{j\geq0}\varepsilon^{\frac{j}{2}}\sum\limits^j_{k=0}(v^{I,k}_1+v^{B,k}_1)\partial_x n^{B,j-k}-\sum\limits_{j\geq0}\varepsilon^{\frac{j-1}{2}}\sum\limits^j_{k=0}(v^{I,k}_2+v^{B,k}_2)\partial_z n^{B,j-k}\notag\\
        &=\sum\limits_{j\geq0}\varepsilon^{\frac{j}{2}}\Delta n^{I,j}+\sum\limits_{j\geq0}\varepsilon^{\frac{j}{2}}\partial^2_x n^{B,j}+\sum\limits_{j\geq0}\varepsilon^{\frac{j-2}{2}}\partial^2_z n^{B,j}.\notag\notag
\end{align}
Substituting $(\ref{asymptotic expansions})$ into equation $(\ref{eq:nvu varepsilon})_2$, it follows
	\begin{align}\label{v varepsilon expansion}
		&\sum\limits_{j\geq0}\varepsilon^{\frac{j}{2}}\partial_t(v^{I,j}+v^{B,j})  +   \sum\limits_{j\geq0}\varepsilon^{\frac{j}{2}}\sum\limits^j_{k=0}\nabla u^{I,k}\cdot(v^{I,j-k}+v^{B,j-k})\notag\\
        &+
        \begin{pmatrix}
	\sum\limits_{j\geq0}\varepsilon^{\frac{j}{2}}\sum\limits^j_{k=0}\partial_x u^{B,k}\cdot(v^{I,j-k}+v^{B,j-k})\notag\\
	\sum\limits_{j\geq0}\varepsilon^{\frac{j-1}{2}}\sum\limits^j_{k=0}\partial_z u^{B,k}\cdot(v^{I,j-k}+v^{B,j-k})\notag\\
        \end{pmatrix}
        +\sum\limits_{j\geq0}\varepsilon^{\frac{j}{2}}\sum\limits^j_{k=0}\nabla v^{I,k}\cdot(u^{I,j-k}+u^{B,j-k})\notag\\
        &+\begin{pmatrix}
	\sum\limits_{j\geq0}\varepsilon^{\frac{j}{2}}\sum\limits^j_{k=0}\partial_x v^{B,k}\cdot(u^{I,j-k}+u^{B,j-k})\notag\\
	\sum\limits_{j\geq0}\varepsilon^{\frac{j-1}{2}}\sum\limits^j_{k=0}\partial_z v^{B,k}\cdot(u^{I,j-k}+u^{B,j-k}) \notag\\
        \end{pmatrix}
         +2\sum\limits_{j\geq0}\varepsilon^{\frac{j+2}{2}}\sum\limits^j_{k=0}\nabla v^{I,k}\cdot(v^{I,j-k}+v^{B,j-k})\notag \\
        &+2\begin{pmatrix}
	\sum\limits_{j\geq0}\varepsilon^{\frac{j+2}{2}}\sum\limits^j_{k=0}\partial_x v^{B,k}\cdot(v^{I,j-k}+v^{B,j-k})\notag\\
	\sum\limits_{j\geq0}\varepsilon^{\frac{j+1}{2}}\sum\limits^j_{k=0}\partial_z v^{B,k}\cdot(v^{I,j-k}+v^{B,j-k})\notag \\
        \end{pmatrix}
        -\sum\limits_{j\geq0}\varepsilon^{\frac{j}{2}}\nabla n^{I,j}-\begin{pmatrix}\sum\limits_{j\geq0}\varepsilon^{\frac{j}{2}}\partial_x n^{B,j}\\\sum\limits_{j\geq0}\varepsilon^{\frac{j-1}{2}}\partial_z n^{B,j}\end{pmatrix}\notag\\
        &=\sum\limits_{j\geq0}\varepsilon^{\frac{j+2}{2}}\Delta v^{I,j}+\sum\limits_{j\geq0}\varepsilon^{\frac{j+2}{2}}\partial^2_x v^{B,j}+\sum\limits_{j\geq0}\varepsilon^{\frac{j}{2}}\partial^2_z v^{B,j}.
	\end{align}
Similarly, plugging $(\ref{asymptotic expansions})$ into the equations given in $(\ref{eq:nvu varepsilon})_3$, it follows
\begin{eqnarray}\label{u varepsilon expansion}
	\begin{split}{}
		&\sum\limits_{j\geq0}\varepsilon^{\frac{j}{2}}\partial_t(u^{I,j}+u^{B,j})+\sum\limits_{j\geq0}\varepsilon^{\frac{j}{2}}\sum\limits^j_{k=0}(u^{I,k}+u^{B,k})\cdot\nabla u^{I,j-k}\\
		&+\sum\limits_{j\geq0}\varepsilon^{\frac{j}{2}}\sum\limits^j_{k=0}(u^{I,k}_1+u^{B,k}_1)\partial_x u^{B,j-k}+\sum\limits_{j\geq0}\varepsilon^{\frac{j-1}{2}}\sum\limits^j_{k=0}(u^{I,k}_2+u^{B,k}_2)\partial_z u^{B,j-k}\\
        &+\sum\limits_{j\geq0}\varepsilon^{\frac{j}{2}}\nabla p^{I,j}+\sum\limits_{j\geq0}(\varepsilon^{\frac{j}{2}}\partial_x p^{B,j},\varepsilon^{\frac{j-1}{2}}\partial_z p^{B,j})\\
		&=\sum\limits_{j\geq0}\varepsilon^{\frac{j+2}{2}}\Delta u^{I,j}+\sum\limits_{j\geq0}\varepsilon^{\frac{j+2}{2}}\partial^2_x u^{B,j}+\sum\limits_{j\geq0}\varepsilon^{\frac{j}{2}}\partial^2_z u^{B,j}+\sum\limits_{j\geq0}\varepsilon^{\frac{j}{2}}(n^{I,j}+n^{B,j}){e_2}.
	\end{split}
\end{eqnarray}
For $(\ref{n varepsilon expansion})$-$(\ref{u varepsilon expansion})$, let $z$ to $+\infty$, then $(n^{B,j},v^{B,j},u^{B,j},p^{B,j})\rightarrow0$.  Furthermore for all $j\geq0$, this implies that
\begin{align}\label{nIj vIj uIj}
		&\partial_tn^{I,j}+\sum\limits^j_{k=0}u^{I,k}\cdot\nabla n^{I,j-k}-\sum\limits^j_{k=0}n^{I,k}\nabla\cdot v^{I,j-k}-\sum\limits^j_{k=0}v^{I,k}\cdot\nabla n^{I,j-k}=\Delta n^{I,j},\notag\\
        &\partial_tv^{I,j}+\sum\limits^j_{k=0}\nabla v^{I,k}\cdot u^{I,j-k}+\sum\limits^j_{k=0}\nabla u^{I,k}\cdot v^{I,j-k}+2\sum\limits^{j-2}_{k=0}\nabla v^{I,k}\cdot v^{I,j-2-k}\\
        &\quad\quad-\nabla n^{I,j}=\Delta v^{I,j-2},\notag\\
        &\partial_tu^{I,j}+\sum\limits^j_{k=0}u^{I,k}\cdot\nabla u^{I,j-k}+\nabla p^{I,j}=\Delta u^{I,j-2}+n^{I,j}{e_2},\notag\notag
\end{align}
 {where we denote by $v^{I,-1}=0$, $v^{I,-2}=0$ and $u^{I,-1}=0$, $u^{I,-2}=0$.}\\

 Next, we {estimate} the internal profile. Before that, we introduce the Taylor expansion
 \begin{equation*}
	\begin{split}{}
		f^{I,j}(t,x,y)=\sum\limits^j_{k=0}\frac{z^k}{k!}\overline{\partial^k_yf^{I,j-k}}(t,x),\\
	\end{split}
\end{equation*}
which expands at $y=0$ for the outer layer of $j\geq0$, and denoted by $\overline{f}(t,x)=f(t,x,0)$.\\
Applying the Taylor expansion, we begin to summarize the equations containing $\varepsilon^{\gamma}$ for $\gamma\geq-1$. First, combining the terms containing $O(\varepsilon^{-1})$, and using $\partial_zn^{B,0}(t,x,0)=0$ in $(\ref{layer boundary})_1$, we  obtain
\begin{equation}\label{nB0=0}
	\begin{split}{}
		\partial^2_zn^{B,0}=0,\;\;\Rightarrow n^{B,0}\equiv0.\\
	\end{split}
\end{equation}
Combining terms containing $O(\varepsilon^{-\frac{1}{2}})$, we can obtain
\begin{equation*}
	\begin{split}{}
		&(\overline{u^{I,0}_2}+u^{B,0}_2)\partial_zn^{B,0}-(\overline{n^{I,0}}+n^{B,0})\partial_zv^{B,0}_2-(\overline{v^{I,0}_2}+v^{B,0}_2)\partial_zn^{B,0}=\partial^2_zn^{B,1},\\
        &\begin{pmatrix}
	0\\
	\partial_zu^{B,0}\cdot(\overline{v^{I,0}}+v^{B,0})
        \end{pmatrix}+\begin{pmatrix}
	0\\
	\partial_zv^{B,0}\cdot(\overline{u^{I,0}}+u^{B,0}) \\
        \end{pmatrix} +
        \begin{pmatrix}
	0\\
	\partial_zn^{B,0} \\
        \end{pmatrix}=0\\
        &(\overline{u^{I,0}_2}+u^{B,0}_2)\partial_zu^{B,0}+\dbinom{0}{\partial_zp^{B,0}}=0.\\
	\end{split}
\end{equation*}
Using $(\ref{vB01=0,uB02=0})$, $(\ref{nB0=0})$ and $\overline{u^{I,0}_2}=0$, we  derive
\begin{eqnarray}\label{nB1 equation and pB0=0}
	\begin{split}{}
        &-\overline{n^{I,0}}\partial_zv^{B,0}_2=\partial^2_zn^{B,1},\\
        &\partial_zu^{B,0}_1\overline{v^{I,0}_1}=0,\\
        &p^{B,0}=0.
	\end{split}
\end{eqnarray}
{For} the term containing $O(\varepsilon^0)$ in $(\ref{n varepsilon expansion})$-$(\ref{u varepsilon expansion})$, we  obtain the following equations
	\begin{align}
        &\partial_t(\overline{n^{I,0}}+n^{B,0})+(\overline{u^{I,0}}+u^{B,0})\cdot\overline{\nabla n^{I,0}}+(\overline{u^{I,0}_1}+u^{B,0}_1)\partial_xn^{B,0}+(\overline{u^{I,0}_2}+u^{B,0}_2)\partial_zn^{B,1}\notag\\
		&+(z\overline{\partial_yu^{I,0}_2}+\overline{u^{I,1}_2}+u^{B,1}_2)\partial_zn^{B,0}-(\overline{n^{I,0}}+n^{B,0})\overline{\nabla\cdot v^{I,0}}-(\overline{n^{I,0}}+n^{B,0})\partial_xv^{B,0}_1\notag\\
        &-(\overline{n^{I,0}}+n^{B,0})\partial_zv^{B,1}_2-(z\overline{\partial_yn^{I,0}}+\overline{n^{I,1}}+n^{B,1})\partial_zv^{B,0}_2-(\overline{v^{I,0}}+v^{B,0})\cdot\overline{\nabla n^{I,0}}\notag\\
        &-(\overline{v^{I,0}_1}+v^{B,0}_1)\partial_xn^{B,0}-(\overline{v^{I,0}_2}+v^{B,0}_2)\partial_zn^{B,1}-(z\overline{\partial_yv^{I,0}_2}+\overline{v^{I,1}_2}+v^{B,1}_2)\partial_zn^{B,0}\notag\\
        &=\overline{\Delta n^{I,0}}+\partial^2_xn^{B,0}+\partial^2_zn^{B,2},\notag\\
        \notag\\
        &\partial_t(\overline{v^{I,0}}+v^{B,0})+\overline{\nabla u^{I,0}}\cdot(\overline{v^{I,0}}+v^{B,0})\notag\\
        &+\begin{pmatrix}
	\partial_xu^{B,0}\cdot(\overline{v^{I,0}}+v^{B,0})\notag\\
	\partial_zu^{B,0}\cdot(z\overline{\partial_yv^{I,0}}+\overline{v^{I,1}}+v^{B,1})+\partial_zu^{B,1}\cdot(\overline{v^{I,0}}+v^{B,0})
        \end{pmatrix}
        +\overline{\nabla v^{I,0}}\cdot(\overline{u^{I,0}}+u^{B,0})\notag\\
        &+\begin{pmatrix}
	\partial_xv^{B,0}\cdot(\overline{u^{I,0}}+u^{B,0})\notag\\
	\partial_zv^{B,0}\cdot(z\overline{\partial_yu^{I,0}}+\overline{u^{I,1}}+u^{B,1})+\partial_zv^{B,1}\cdot(\overline{u^{I,0}}+u^{B,0})
        \end{pmatrix}
        -\overline{\nabla n^{I,0}}-\begin{pmatrix}\partial_xn^{B,0}\\0\end{pmatrix}\notag\\
        &-\begin{pmatrix}0\\\partial_zn^{B,1}\end{pmatrix}=\partial^2_zv^{B,0},\label{varepsilon 0 expansion}\\
        \notag\\
		&\partial_t(\overline{u^{I,0}}+u^{B,0})+(\overline{u^{I,0}}+u^{B,0})\cdot\overline{\nabla u^{I,0}}+(\overline{u^{I,0}_1}+u^{B,0}_1)\partial_xu^{B,0}+z\overline{\partial_yu^{I,0}_2}\partial_zu^{B,0}\notag\\
		&+(\overline{u^{I,1}_2}+u^{B,1}_2)\partial_zu^{B,0}+(\overline{u^{I,0}_2}+u^{B,0}_2)\partial_zu^{B,1}+\overline{\nabla p^{I,0}}+\dbinom{\partial_xp^{B,0}}{0}+\dbinom{0}{\partial_zp^{B,1}}\notag\\
        &=\partial^2_zu^{B,0}+(\overline{n^{I,0}}+n^{B,0}){e_2}.\notag\\\notag
	\end{align}
    By substituting the outflow layer equation $(\ref{eq:n0v0u0})_1$ into $(\ref{varepsilon 0 expansion})_1$, using $(\ref{vB01=0,uB02=0})$, $(\ref{nB0=0})$ and $\overline{u^{I,0}_2}=0$, we can obtain
\begin{equation}\label{nB2 system}
	\begin{split}{}
		&u^{B,0}_1\overline{\partial_x n^{I,0}}-\overline{n^{I,0}}\partial_zv^{B,1}_2-(z\overline{\partial_yn^{I,0}}+\overline{n^{I,1}}+n^{B,1})\partial_zv^{B,0}_2-\overline{\partial_yn^{I,0}}v^{B,0}_2-(\overline{v^{I,0}_2}+v^{B,0}_2)\partial_zn^{B,1}=\partial^2_zn^{B,2}.
	\end{split}
\end{equation}
Since $v^{B,0}_1\equiv0$ is obtained in $(\ref{vB01=0,uB02=0})$, we only need to take the second component of the equation $(\ref{varepsilon 0 expansion})_2$, and at the same time substitute the second component $\overline{\partial_tv^{I,0}_2}+\overline{\partial_y(u^{I,0}\cdot v^{I,0})}-\overline{\partial_yn^{I,0}}=0$ of the outflow layer equation $(\ref{nIj vIj uIj})_2$ for $j=0$ into $(\ref{varepsilon 0 expansion})_2$ to obtain
\begin{equation}\label{vB02 outflow layer}
	\begin{split}{}		
&\partial_tv^{B,0}_2+\overline{\partial_yu^{I,0}_2}v^{B,0}_2+\partial_zu^{B,1}_1\overline{v^{I,0}_1}+\partial_zu^{B,1}_2(\overline{v^{I,0}_2}+v^{B,0}_2)+\partial_zu^{B,0}_1(z\overline{\partial_yv^{I,0}_1}+\overline{v^{I,1}_1}+v^{B,1}_1)\\        &+\overline{\partial_yv^{I,0}_1}u^{B,0}_1+\partial_zv^{B,1}_1(\overline{u^{I,0}_1}+u^{B,0}_1)+\partial_zv^{B,0}_2(z\overline{\partial_yu^{I,0}_2}+\overline{u^{I,1}_2}+u^{B,1}_2)-\partial_zn^{B,1}\\
&=\partial^2_zv^{B,0}_2.
	\end{split}
\end{equation}
We take the first component of equation $(\ref{varepsilon 0 expansion})_3$, {by the outflow layer equation $(\ref{eq:n0v0u0})_4$, $(\ref{vB01=0,uB02=0})$, $\overline{u^{I,0}_2}=0$ and $(\ref{nB1 equation and pB0=0})_3$, then we obtain
\begin{equation}\label{uB01 equation}
	\begin{split}{}
		&\partial_tu^{B,0}_1+\overline{u^{I,0}_1}\partial_xu^{B,0}_1+u^{B,0}_1\overline{\partial_x u^{I,0}_1}+u^{B,0}_1\partial_xu^{B,0}_1+z\overline{\partial_yu^{I,0}_2}\partial_zu^{B,0}_1\\
		&+(\overline{u^{I,1}_2}+u^{B,1}_2)\partial_zu^{B,0}_1=\partial^2_zu^{B,0}_1.
	\end{split}
\end{equation}
For the second component of $(\ref{varepsilon 0 expansion})_3$, this implies that
\begin{equation*}
	\begin{split}{}
		&\overline{\partial_y p^{I,0}}+\partial_zp^{B,1}=\overline{n^{I,0}}.\\
	\end{split}
\end{equation*}
Since the outflow layer equation $(\ref{nIj vIj uIj})_3$ can deduce $\overline{\partial_y p^{I,0}}=\overline{n^{I,0}}$, then
\begin{equation}\label{pB1=0}
	\begin{split}{}
		p^{B,1}\equiv0.\\
	\end{split}
\end{equation}
Combining $(\ref{layer boundary})_5$ and $u^{B,0}_1(0,x,z)=0$ in $(\ref{initial zero layer})$, we can obtain that $u^{B,0}_1$ satisfies the following system
\begin{equation}\label{uB01 system}
       \left\{
	\begin{split}{}		&\partial_tu^{B,0}_1+\overline{u^{I,0}_1}\partial_xu^{B,0}_1+u^{B,0}_1\overline{\partial_xu^{I,0}_1}+u^{B,0}_1\partial_xu^{B,0}_1+(\overline{u^{I,1}_2}+u^{B,1}_2+z\overline{\partial_yu^{I,0}_2})\partial_zu^{B,0}_1=\partial^2_zu^{B,0}_1,\\
    &\partial_zu^{B,0}_1|_{z=0}=0,\\
    &u^{B,0}_1|_{t=0}=0.
	\end{split}
    \right.
\end{equation}
Next, we will prove that this system has the unique solution $u^{B,0}_1=0$.
Therefore, we define energy functional
\begin{equation}\label{uB01 energy functional}
	\begin{split}{}
		E(t)=:\frac{1}{2}\int^\infty_{-\infty}\int^\infty_0(u^{B,0}_1)^2dzdx.
	\end{split}
\end{equation}
It has a time derivative
\begin{equation}\label{uB01 energy time derivative}
	\begin{split}{}
		\frac{d}{dt}E(t)&=\int^\infty_{-\infty}\int^\infty_0u^{B,0}_1\partial_tu^{B,0}_1dzdx\\
        &=\int^\infty_{-\infty}\int^\infty_0u^{B,0}_1[-\overline{u^{I,0}_1}\partial_xu^{B,0}_1-u^{B,0}_1\overline{\partial_xu^{I,0}_1}-u^{B,0}_1\partial_xu^{B,0}_1\\
        &\quad-(\overline{u^{I,1}_2}+u^{B,1}_2+z\overline{\partial_yu^{I,0}_2})\partial_zu^{B,0}_1+\partial^2_zu^{B,0}_1]dzdx.\\
	\end{split}
\end{equation}
{Using integration by parts, $(\ref{outer layer curl and div})$ and $(\ref{inner layer curl and div})_2$,}
we derive
\begin{equation*}
	\begin{split}{}
		\frac{d}{dt}E(t)&=-\int^\infty_{-\infty}\int^\infty_0\overline{\partial_xu^{I,0}_1}(u^{B,0}_1)^2dzdx+\int^\infty_{-\infty}\int^\infty_0u^{B,0}_1\partial^2_zu^{B,0}_1dzdx\\
        &=-\int^\infty_{-\infty}\int^\infty_0\overline{\partial_xu^{I,0}_1}(u^{B,0}_1)^2dzdx-\int^\infty_{-\infty}\int^\infty_0(\partial_zu^{B,0}_1)^2dzdx,\\
	\end{split}
\end{equation*}
{by the regularity of $\overline{u^{I,0}}$ and the Gronwall's inequality,
then we have $E(t)=0$.} Therefore, $u^{B,0}_1=0$ is the unique solution of equation $(\ref{uB01 system})$, and we can get
\begin{equation}\label{uB01=0 uB12=0}
	\begin{split}{}
		u^{B,0}_1\equiv0\Rightarrow u^{B,1}_2\equiv0.
	\end{split}
\end{equation}
Substituting $(\ref{uB01=0 uB12=0})$ into $(\ref{layer boundary})_3$, we can obtain $\overline{u^{I,1}_2}=0$. Thus, the equations $(\ref{nB2 system})$ and $(\ref{vB02 outflow layer})$ can be reduced to
\begin{equation}\label{nB2 equation}
	\begin{split}{}
		&-\overline{n^{I,0}}\partial_zv^{B,1}_2-(z\overline{\partial_yn^{I,0}}+\overline{n^{I,1}}+n^{B,1})\partial_zv^{B,0}_2-\overline{\partial_yn^{I,0}}v^{B,0}_2-(\overline{v^{I,0}_2}+v^{B,0}_2)\partial_zn^{B,1}=\partial^2_zn^{B,2},\\
	\end{split}
\end{equation}
and
\begin{equation}\label{vB02 equation}
	\begin{split}{}		
&\partial_tv^{B,0}_2+\overline{\partial_yu^{I,0}_2}v^{B,0}_2+\partial_zu^{B,1}_1\overline{v^{I,0}_1}+\partial_zv^{B,1}_1\overline{u^{I,0}_1}+\partial_zv^{B,0}_2z\overline{\partial_yu^{I,0}_2}-\partial_zn^{B,1}=\partial^2_zv^{B,0}_2.\\
	\end{split}
\end{equation}
Moreover, we can obtain the first-order outer layer $(n^{I,1},v^{I,1},u^{I,1})(t,x,y)$ is the solution of
         \begin{align}\label{eq:nI1vI1uI1}
		\begin{cases}
		\partial_tn^{I,1}+u^{I,0}\cdot\nabla n^{I,1}+u^{I,1}\cdot\nabla n^{I,0}-\nabla\cdot(n^{I,0} v^{I,1})-\nabla\cdot(n^{I,1} v^{I,0})=\Delta n^{I,1}\\
        \partial_tv^{I,1}+\nabla (v^{I,0}\cdot u^{I,1})+\nabla (v^{I,1}\cdot u^{I,0})-\nabla n^{I,1}=0\\
        \nabla\times v^{I,1}=0,\\
        \partial_tu^{I,1}+u^{I,0}\cdot\nabla u^{I,1}+u^{I,1}\cdot\nabla u^{I,0}+\nabla p^{I,1}=n^{I,1}{e_2}\\
        \nabla\cdot u^{I,1}=0,\\
        \partial_yn^{I,1}|_{y=0}=\overline{n^{I,0}}\overline{v^{I,1}_2}+\overline{n^{I,1}}\overline{v^{I,0}_2}+\overline{\partial_yn^{I,0}}\int_0^{+\infty}\left(z\partial_zv_2^{B,0}+v_2^{B,0}\right)\,dz,\;u^{I,1}_2|_{y=0}=0,\\
        (n^{I,1},v^{I,1},u^{I,1})(0,x,y)=0.
		\end{cases}
	\end{align}

    Expanding the no-flux boundary condition $(\ref{layer boundary})$ at order \(\varepsilon^{1/2}\), we have
\begin{align*}
	\left.\partial_y n^{I,1}\right|_{y=0}=-\left.\partial_z n^{B,2}\right|_{z=0}.
\end{align*}
We integrate the equation $(\ref{nB2 equation})$ with respect to \(z\) and use \(v_2^{I,j}+v_2^{B,j}=0\) on \(z=0\) for \(j=0,1\) to get
\begin{align*}
	v_2^{B,0}|_{z=0}=-v_2^{I,0},\qquad v_2^{B,1}|_{z=0}=-v_2^{I,1}.
\end{align*}
Then we obtain the boundary condition for \(n^{I,1}\) in $(\ref{eq:nI1vI1uI1})$.
Furthermore, owing to the rapid decay of \(v_2^{B,0}\) as \(z\to+\infty\), we have
\begin{align*}	\int_0^{+\infty}\left(z\partial_zv_2^{B,0}+v_2^{B,0}\right)\,dz=\left(zv_2^{B,0}\right)\Big{|}_{0}^{+\infty}=0.
\end{align*}
Then, the boundary condition reduces to
\begin{align*}
	\partial_y n^{I,1}+n^{I,0}v_2^{I,1}+n^{I,1}v_2^{I,0}=0\qquad\text{on }y=0.
\end{align*}
Therefore, by the well-posedness theory for solutions of linear systems, we conclude that the trivial solution
\begin{align}\label{nvu I1=0}
  n^{I,1}=v^{I,1}=u^{I,1}=0,
\end{align}
is the unique solution to the system (\ref{eq:nI1vI1uI1}).
Then equation (\ref{nB2 equation}) reduces to
\begin{equation}\label{eq:nB2 equation}
	\begin{split}{}
		&-\overline{n^{I,0}}\partial_zv^{B,1}_2-(z\overline{\partial_yn^{I,0}}+n^{B,1})\partial_zv^{B,0}_2-\overline{\partial_yn^{I,0}}v^{B,0}_2-(\overline{v^{I,0}_2}+v^{B,0}_2)\partial_zn^{B,1}=\partial^2_zn^{B,2}.\\
	\end{split}
\end{equation}
For the inner profile $n^{B,j}$, we have obtained $(\ref{nB1 equation and pB0=0})_1$ and $(\ref{nB2 equation})$, so we will not do any its higher-order decomposition later.\\
Similar to $(\ref{varepsilon 0 expansion})_{2,3}$, {for the term containing $O(\varepsilon^\frac{1}{2})$ in $(\ref{v varepsilon expansion})$-$(\ref{u varepsilon expansion})$, we can obtain the following equations}
\begin{align}
&\partial_t(z\overline{\partial_yv^{I,0}}+\overline{v^{I,1}}+v^{B,1})\notag\\
&+
\begin{pmatrix}
	\overline{\partial_xu^{I,0}}\cdot(z\overline{\partial_yv^{I,0}}+\overline{v^{I,1}}+v^{B,1})+\partial_x(z\overline{\partial_yu^{I,0}}+\overline{u^{I,1}})\cdot(\overline{v^{I,0}}+v^{B,0})\\
	\overline{\partial_yu^{I,0}}\cdot(z\overline{\partial_yv^{I,0}}+\overline{v^{I,1}}+v^{B,1})+(z\overline{\partial^2_yu^{I,0}}+\overline{\partial_yu^{I,1}})\cdot(\overline{v^{I,0}}+v^{B,0})\\
        \end{pmatrix}\notag\\
        &+\begin{pmatrix}
	\partial_xu^{B,0}\cdot(z\overline{\partial_yv^{I,0}}+\overline{v^{I,1}}+v^{B,1})+\partial_xu^{B,1}\cdot(\overline{v^{I,0}}+v^{B,0})\\
	\partial_zu^{B,0}\cdot(\frac{1}{2}z^2\overline{\partial^2_yv^{I,0}}+z\overline{\partial_yv^{I,1}}+\overline{v^{I,2}}+v^{B,2})
        \end{pmatrix}\notag\\
      &+\begin{pmatrix}
	0\\
	\partial_zu^{B,1}\cdot(z\overline{\partial_yv^{I,0}}+\overline{v^{I,1}}+v^{B,1})+\partial_zu^{B,2}\cdot(\overline{v^{I,0}}+v^{B,0})
        \end{pmatrix}\label{vB1 expansion}\\
        &+\begin{pmatrix}
	\overline{\partial_xv^{I,0}}\cdot(z\overline{\partial_yu^{I,0}}+\overline{u^{I,1}}+u^{B,1})+\partial_x(z\overline{\partial_yv^{I,0}}+\overline{v^{I,1}})\cdot(\overline{u^{I,0}}+u^{B,0})\\
	\overline{\partial_yv^{I,0}}\cdot(z\overline{\partial_yu^{I,0}}+\overline{u^{I,1}}+u^{B,1})+(z\overline{\partial^2_yv^{I,0}}+\overline{\partial_yv^{I,1}})\cdot(\overline{u^{I,0}}+u^{B,0})\\
        \end{pmatrix}\notag\\
       &+\begin{pmatrix}
	\partial_xv^{B,0}\cdot(z\overline{\partial_yu^{I,0}}+\overline{u^{I,1}}+u^{B,1})+\partial_xv^{B,1}\cdot(\overline{u^{I,0}}+u^{B,0})\\
	\partial_zv^{B,0}\cdot(\frac{1}{2}z^2\overline{\partial^2_yu^{I,0}}+z\overline{\partial_yu^{I,1}}+\overline{u^{I,2}}+u^{B,2})
        \end{pmatrix}\notag\\
        &+\begin{pmatrix}
	0\\
	\partial_zv^{B,1}\cdot(z\overline{\partial_yu^{I,0}}+\overline{u^{I,1}}+u^{B,1})+\partial_zv^{B,2}\cdot(\overline{u^{I,0}}+u^{B,0})
        \end{pmatrix}\notag\\
        &+2\begin{pmatrix}
	0\\
	\partial_zv^{B,0}\cdot(\overline{v^{I,0}}+v^{B,0})
        \end{pmatrix}
        -\overline{\nabla n^{I,1}}-z\overline{\nabla\partial_y n^{I,0}}-\begin{pmatrix}
	\partial_xn^{B,1}\\
	\partial_zn^{B,2}
        \end{pmatrix}
        =\partial^2_zv^{B,1}, \notag\\\notag
        \end{align}
and
\begin{align}\label{uB1 expansion}	&\partial_t(z\overline{\partial_yu^{I,0}}+\overline{u^{I,1}}+u^{B,1})+\overline{u^{I,0}_1}\partial_x(z\overline{\partial_yu^{I,0}}+\overline{u^{I,1}})+(z\overline{\partial_yu^{I,0}_1}+\overline{u^{I,1}_1}+u^{B,1}_1)\overline{\partial_xu^{I,0}}\notag\\
        &+z\overline{\partial_yu^{I,0}_2}\overline{\partial_yu^{I,0}}+\overline{u^{I,0}_1}\partial_xu^{B,1}+z\overline{\partial_yu^{I,0}_2}\partial_zu^{B,1}+\overline{\nabla p^{I,1}}+z\overline{\nabla \partial_yp^{I,0}}+\dbinom{0}{\partial_zp^{B,2}}\\
        &=\partial^2_zu^{B,1}+(z\overline{\partial_yn^{I,0}}+\overline{n^{I,1}}+n^{B,1}){e_2}.\notag\notag
\end{align}
Taking the first component of $(\ref{vB1 expansion})$, by substituting the outflow equation $(\ref{nIj vIj uIj})_2$ with $j=1$ into $(\ref{vB1 expansion})$, and using $(\ref{vB01=0,uB02=0})$, $\overline{u^{I,0}_2}=0$, it can be concluded that
\begin{align}\label{vB11 equation}
&\partial_tv^{B,1}_1+\overline{\partial_xu^{I,0}_1}v^{B,1}_1+\partial_x(z\overline{\partial_yu^{I,0}_2})v^{B,0}_2+\partial_xu^{B,1}_1\overline{v^{I,0}_1}\notag\\
&+\overline{\partial_xv^{I,0}_1}u^{B,1}_1+\partial_xv^{B,0}_2z\overline{\partial_yu^{I,0}_2}+\partial_xv^{B,1}_1\overline{u^{I,0}_1}-\partial_xn^{B,1}=\partial^2_zv^{B,1}_1.
\end{align}
Similar to $(\ref{vB11 equation})$, taking the second component of $(\ref{vB1 expansion})$ and using (\ref{nvu I1=0}), we get
\begin{align}\label{vB12 equation}
&\partial_tv^{B,1}_2+\overline{\partial_yu^{I,0}_1}v^{B,1}_1+\overline{\partial_yu^{I,0}_2}v^{B,1}_2+z\overline{\partial^2_yu^{I,0}_2}v^{B,0}_2+\partial_zu^{B,1}_1(z\overline{\partial_yv^{I,0}_1}+v^{B,1}_1)\notag\\
&+\partial_zu^{B,2}_1\overline{v^{I,0}_1}+\partial_zu^{B,2}_2(\overline{v^{I,0}_2}+v^{B,0}_2)+\overline{\partial_yv^{I,0}_1}u^{B,1}_1+\partial_zv^{B,0}_2(\frac{1}{2}z^2\overline{\partial^2_yu^{I,0}_2}+\overline{u^{I,2}_2}+u^{B,2}_2)\\
&+\partial_zv^{B,1}_1(z\overline{\partial_yu^{I,0}_1}+u^{B,1}_1)+\partial_zv^{B,1}_2z\overline{\partial_yu^{I,0}_2}+\partial_zv^{B,2}_1\overline{u^{I,0}_1}+2\partial_zv^{B,0}_2(\overline{v^{I,0}_2}+v^{B,0}_2)-\partial_zn^{B,2}=\partial^2_zv^{B,1}_2.\notag\notag
\end{align}
By substituting the first component of $(\ref{nIj vIj uIj})_3$ into $(\ref{uB1 expansion})$, and using $(\ref{pB1=0})$, then the first component of $(\ref{uB1 expansion})$ is reduced to
\begin{equation}\label{uB11 equation}
	\begin{split}{}  &\partial_tu^{B,1}_1+u^{B,1}_1\overline{\partial_xu^{I,0}_1}+\overline{u^{I,0}_1}\partial_xu^{B,1}_1+z\overline{\partial_yu^{I,0}_2}\partial_zu^{B,1}_1=\partial^2_zu^{B,1}_1.\\
	\end{split}
\end{equation}
Substituting the second component of $(\ref{nIj vIj uIj})_3$ into $(\ref{uB1 expansion})$, and owing to $u^{B,1}_2=0$ in $(\ref{uB01=0 uB12=0})$, we obtain
\begin{equation}\label{pB2 equation}
	\begin{split}{}
		&\partial_zp^{B,2}=n^{B,1}.\\
	\end{split}
\end{equation}
Combining the term in $(\ref{v varepsilon expansion})$ containing $O(\varepsilon)$, and taking the first component, by substituting the outflow equation $(\ref{nIj vIj uIj})_2$ with $j=2$, using $(\ref{vB01=0,uB02=0})$ $(\ref{uB01=0 uB12=0})$ and (\ref{nvu I1=0}), we can get
\begin{align}\label{vB21 equation}
&\partial_tv^{B,2}_1+\overline{\partial_xu^{I,0}_1}v^{B,2}_1+z\overline{\partial_x\partial_yu^{I,0}_1}v^{B,1}_1+z\overline{\partial_x\partial_{y}u^{I,0}_2}v^{B,1}_2+\partial_x(\frac{1}{2}z^2\overline{\partial^2_yu^{I,0}_2}+\overline{u^{I,2}_2})v^{B,0}_2\notag\\
    &+\partial_xu^{B,1}_1(z\overline{\partial_yv^{I,0}_1}+v^{B,1}_1)+\partial_xu^{B,2}_1\overline{v^{I,0}_1}+\partial_xu^{B,2}_2(\overline{v^{I,0}_2}+v^{B,0}_2)+\overline{\partial_xv^{I,0}_1}u^{B,2}_1\notag\\
    &+\overline{\partial_xv^{I,0}_2}u^{B,2}_2+z\overline{\partial_x\partial_yv^{I,0}_1}u^{B,1}_1+\partial_xv^{B,0}_2(\frac{1}{2}z^2\overline{\partial^2_yu^{I,0}_2}+\overline{u^{I,2}_2}+u^{B,2}_2)\\
    &+\partial_xv^{B,1}_1(z\overline{\partial_yu^{I,0}_1}+u^{B,1}_1)+\partial_xv^{B,1}_2z\overline{\partial_{y}u^{I,0}_2}+\partial_xv^{B,2}_1\overline{u^{I,0}_1}+\overline{\partial_xv^{I,0}_2}v^{B,0}_2\notag\\
    &+2\partial_xv^{B,0}_2(\overline{v^{I,0}_2}+v^{B,0}_2)-\partial_xn^{B,2}=\partial^2_zv^{B,2}_1.\notag\notag
\end{align}
Similarly, combining the term in $(\ref{u varepsilon expansion})$ containing $O(\varepsilon)$, and taking the first component, by substituting the outflow equation $(\ref{nIj vIj uIj})_3$ with $j=2$, { $(\ref{vB01=0,uB02=0})$, $(\ref{uB01=0 uB12=0})$ and (\ref{nvu I1=0}), we obtain
\begin{equation}\label{uB21 equation}
	\begin{split}{}  &\partial_tu^{B,2}_1+u^{B,1}_1z\overline{\partial_x\partial_yu^{I,0}_1}+u^{B,2}_1\overline{\partial_xu^{I,0}_1}+u^{B,2}_2\overline{\partial_yu^{I,0}_1}+\overline{u^{I,0}_1}\partial_xu^{B,2}_1+z\overline{\partial_yu^{I,0}_1}\partial_xu^{B,1}_1\\
    &+u^{B,1}_1\partial_xu^{B,1}_1+z\overline{\partial_yu^{I,0}_2}\partial_zu^{B,2}_1+(\frac{1}{2}z^2\overline{\partial^2_yu^{I,0}_2}+\overline{u^{I,2}_2})\partial_zu^{B,1}_1+u^{B,2}_2\partial_zu^{B,1}_1+\partial_xp^{B,2}=\partial^2_zu^{B,2}_1.\\
	\end{split}
\end{equation}
{By} the divergence conditions, we can obtain the expressions for $u^{B,2}_2$ and $u^{B,3}_2$ as follows
\begin{eqnarray}\label{uB22 and uB32 equation}
		\begin{split}{}
			u^{B,2}_2(t,x,z)=\int^\infty_z\partial_xu^{B,1}_1(t,x,\eta)d\eta,\;\;
            u^{B,3}_2(t,x,z)=\int^\infty_z\partial_xu^{B,2}_1(t,x,\eta)d\eta.
		\end{split}
	\end{eqnarray}
Similar to the previous derivation process, we can further derive the initial boundary value problems for high-order profiles $(n^{I,j},v^{I,j},u^{I,j},p^{I,j})$ and $(n^{B,j+1},v^{B,j+1}_1,v^{B,j}_2,u^{B,j+1}_1,u^{B,j+2}_2,p^{B,j+1})$ with $j\geq2$, but they are not related to our results, hence these derivations are omitted.

\section{leading-order outer layer estimate}\label{4}
        In this section, we prove that the chemotaxis Euler system (\ref{eq:n0v0u0}) subject to the boundary condition $(\ref{boundary condition})_2$ possesses the regularity result stated in Theorem \ref{th2.1} in the general Sobolev spaces.

        First, we introduce  the following useful Lemmas, which are frequently used in the proof of Theorem \ref{th2.1}.

        \begin{lemma}[See \cite{GT,TX}]\label{lem nabla omega}
        {Let $u = (u_1, u_2)$ be a vector field in the half plane $\mathbb{R}_+^2 = \{ (x,y) \in \mathbb{R}^2 : y>0 \}$, satisfying
        $\mathrm{div}\,u = 0$
        with the boundary condition $u_2|_{y=0}=0$. Denote $\omega =\nabla\times\,u = \partial_x u_2 - \partial_y u_1$, then for the integer $k\geq 0$, there exists a constant $C$ depending only on $k$, such that
        \begin{equation*}
				\begin{split}{}
        \left\| \nabla^{k+1} u \right\|_{L^2(\mathbb{R}_+^2)} \leq C \left\| \nabla^k \omega \right\|_{L^2(\mathbb{R}_+^2)}.
        \end{split}
			\end{equation*}	
            }
        \end{lemma}

\begin{lemma}[See \cite{AF}]\label{lem trace theorem}
		{{In the half plane $\R^2_+$}, the following inequality holds true:
			\begin{equation*}
				\begin{split}{}
					\|f(x,y)\|_{L^2_{x}L^\infty_{y}}\leq C\|f\|^{\frac{1}{2}}_{L^2_{xy}}\|\partial_yf\|^{\frac{1}{2}}_{L^2_{xy}},\;\; \forall f\in W^{1,2}_{xy}.
				\end{split}
			\end{equation*}	
		}
	\end{lemma}
    \begin{lemma}\label{lem higher regurality}
		{For any $f,g\in H^{s}(\R^2_+)\cap L^\infty(\R^2_+)$ with $s>0$, it holds  that
			\begin{equation*}
				\begin{split}{}					
                    \|fg\|_{H^s_{xy}}\leq C(\|f\|_{L^\infty_{xy}}\|g\|_{H^s_{xy}}+\|f\|_{H^s_{xy}}\|g\|_{L^\infty_{xy}}).
				\end{split}
			\end{equation*}	
		}
	\end{lemma}

       To prove the local well-posedness of the system (\ref{eq:n0v0u0}), it suffices to obtain some necessary prior estimates. We assume that the initial data $(n_{in},v_{in},u_{in})$ in $\R^2_+$ and the solution $(n^0,v^0,u^0)$ in $\R^2_+\times [0,T^*)$ are suitably smooth without loss of generality.

        \begin{lemma}\label{n0 v0 u0 L2}
         {Assume that the initial data $(n_{in},v_{in},u_{in})$ satisfy the assumptions of Theorem \ref{th2.1} and $(n^0,v^0,u^0)$ is smooth in $\R^2_+\times [0,T^*)$. Then there exists $0<T<T^*$ such that the uniform energy estimates hold:
		\begin{equation}\label{n0 v0 u0 H2}
        \begin{split}{}
        &\|(n^0, v^0, u^0, \nabla u^0)\|^2_{L^\infty_TH^2_{xy}}+\|(\partial_tn^0, \partial_tv^0, \partial_tu^0, \partial_t\nabla u^0) \|^2_{L^\infty_TL^2_{xy}}+\|(\nabla\partial_tn^0, \n\partial^2_x n^0)\|^2_{L^2_TL^2_{xy}}\\
        &\leq C(\|n_{in}\|_{H^2_{xy}},\|v_{in}\|_{H^2_{xy}},\|u_{in}\|_{H^3_{xy}},\|\partial_tn_{in}\|_{L^2_{xy}},\|\partial_tv_{in}\|_{L^2_{xy}},\|\partial_tu_{in}\|_{H^1_{xy}}),\\
        \end{split}
		\end{equation}
        and
        \begin{equation}\label{nabla3 n0 L2T}
        \begin{split}{}
        \|(\nabla\partial_tv^0, \nabla^2\partial_t u^0) \|^2_{L^\infty_TL^2_{xy}}+\|\n^3 n^0\|_{L^2_TL^2_{xy}}&\leq C(\|n_{in}\|_{H^2_{xy}},\|v_{in}\|_{H^2_{xy}},\|u_{in}\|_{H^3_{xy}}).\\
        \end{split}
		\end{equation}

		}
	\end{lemma}
       \noindent{\bf{Proof.}} Consider the evolution of the following energy functional:
       \begin{align}\label{energy functional}
           E(t)&:=\| (n^0, v^0,  \omega^0)\|^2_{H^2_{xy}}+\|(u^0, \partial_t\omega^0)\|^2_{L^2_{xy}}+\|(\partial_tn^0, \partial_tv^0, \partial_tu^0)\|^2_{L^2_{xy}},
       \end{align}
       where $\omega^0=\partial_xu^0_2-\partial_yu^0_1$, which is divided  into four main steps.

       \textbf{Step 1. The $L^2_{xy}$ energy estimate.}
       Taking the $L^2$ inner product of $(\ref{eq:n0v0u0})_{1}$ with $n^0$ and applying integration by parts, the divergence condition
$(\ref{eq:n0v0u0})_{5}$ and
 the boundary condition $(\ref{eq:n0v0u0})_{6}$, we deduce that
       \begin{align*}
		\begin{split}{}
        \frac{1}{2}\frac{d}{dt}\int^\infty_0\int^\infty_{-\infty}|n^0|^2dxdy&=\int^\infty_0\int^\infty_{-\infty}\nabla\cdot(\nabla n^0+n^0v^0)n^0dxdy-\int^\infty_0\int^\infty_{-\infty}u^0\cdot\nabla n^0n^0dxdy\\
        &=-\int^\infty_0\int^\infty_{-\infty}|\nabla n^0|^2dxdy-\int^\infty_0\int^\infty_{-\infty}\nabla n^0\cdot (n^0v^0)dxdy,
		\end{split}
	\end{align*}
   which  implies
    \begin{equation}\label{n0 L2}
		\begin{split}{}
        \frac{1}{2}\frac{d}{dt}\|n^0\|^2_{L^2_{xy}}+\|\nabla n^0\|^2_{L^2_{xy}}&\leq\|\nabla n^0\|_{L^2_{xy}}\|n^0\|_{L^4_{xy}}\|v^0\|_{L^4_{xy}}\\
        &\leq C\|\nabla n^0\|^2_{L^2_{xy}}\|n^0\|^2_{H^1_{xy}}+C\|v^0\|^2_{H^1_{xy}},
		\end{split}
	\end{equation}
    due to the Gagliardo¨CNirenberg interpolation inequality.

    Similarly, multiplying $(\ref{eq:n0v0u0})_{2}$ by $v^0$, we obtain
    \begin{equation}\label{v0 L2}
		\begin{split}{}
        \frac{1}{2}\frac{d}{dt}\|v^0\|^2_{L^2_{xy}}&=-\int^\infty_0\int^\infty_{-\infty}\nabla(u^0\cdot v^0)\cdot v^0dxdy+\int^\infty_0\int^\infty_{-\infty}\nabla n^0\cdot v^0dxdy\\
        &\leq C\left(\|\nabla u^0\|_{L^2_{xy}}\|v^0\|^2_{H^1_{xy}}+\|\nabla v^0\|_{L^2_{xy}}\|v^0\|_{H^1_{xy}}\|u^0\|_{H^1_{xy}}+\|\nabla n^0\|_{L^2_{xy}}\|v^0\|_{L^2_{xy}}\right).
		\end{split}
	\end{equation}
    Furthermore, we take the $L^2$ inner product of $(\ref{eq:n0v0u0})_{4}$ with $u^0$. Using the boundary condition $u^0_2|_{y=0}=0$ and the divergence-free $\nabla\cdot u^0=0$, we get
    \begin{align}\label{u0 L2}
        \frac{1}{2}\frac{d}{dt}\|u^0\|^2_{L^2_{xy}}&=-\int^\infty_0\int^\infty_{-\infty}u^0\cdot\nabla u^0\cdot u^0dxdy-\int^\infty_0\int^\infty_{-\infty}\nabla p^0\cdot u^0dxdy+\int^\infty_0\int^\infty_{-\infty}n^0e_2\cdot u^0dxdy\notag\\
       &=\int^\infty_0\int^\infty_{-\infty}n^0 u^0_2dxdy\leq \|n^0\|_{L^2_{xy}}\|u^0\|_{L^2_{xy}}.
	\end{align}
     Combining the estimates (\ref{n0 L2})-(\ref{u0 L2}), (\ref{energy functional}) and Lemma \ref{lem nabla omega}, we conclude
    \begin{equation}\label{n0 v0 u0 L2 estimate}
		\begin{split}{}
        &\frac{d}{dt}(\|n^0\|^2_{L^2_{xy}}+\|v^0\|^2_{L^2_{xy}}+\|u^0\|^2_{L^2_{xy}})\leq C(E(t)+1)^2.
		\end{split}
	\end{equation}
    \textbf{Step 2. The gradient estimates.} In this step, we  estimate $\n n^0$, $\n v^0$, $\n u^0$ and $\nabla \omega^0$.
    Multiplying $(\ref{eq:n0v0u0})_{1}$ by $\partial_t n^0$, by $(\ref{eq:n0v0u0})_{6}$ it can be shown that
    \begin{align*}
		\begin{split}{}
        \int^\infty_0\int^\infty_{-\infty}|\partial_tn^0|^2dxdy&=\int^\infty_0\int^\infty_{-\infty}\nabla\cdot(\nabla n^0+n^0v^0)\partial_tn^0dxdy-\int^\infty_0\int^\infty_{-\infty}u^0\cdot\nabla n^0\partial_tn^0dxdy\\
        &\leq-\int^\infty_0\int^\infty_{-\infty}\nabla \partial_tn^0\cdot\nabla n^0 dxdy-\int^\infty_0\int^\infty_{-\infty}\nabla\partial_t n^0\cdot (n^0v^0)dxdy\\
        &\quad+\int^\infty_0\int^\infty_{-\infty}|u^0||\nabla n^0||\partial_tn^0|dxdy,
		\end{split}
	\end{align*}
    which implies that
    \begin{equation}\label{nabla n0}
		\begin{split}{}
        &\frac{1}{2}\frac{d}{dt}\|\n n^0\|^2_{L^2_{xy}}+\|\partial_t n^0\|^2_{L^2_{xy}}\\
        &\leq\|\nabla\partial_t n^0\|_{L^2_{xy}}\|n^0\|_{L^4_{xy}}\|v^0\|_{L^4_{xy}}+\|u^0\|_{L^\infty_{xy}}\|\nabla n^0\|_{L^2_{xy}}\|\partial_t n^0\|_{L^2_{xy}}\\
        &\leq \frac{1}{16}\|\nabla\partial_t n^0\|^2_{L^2_{xy}}+\frac{1}{8}\|\partial_t n^0\|^2_{L^2_{xy}}+C(\| n^0\|^2_{H^1_{xy}}\|v^0\|^2_{H^1_{xy}}+\|u^0\|^2_{H^2_{xy}}\|\nabla n^0\|^2_{L^2_{xy}}).
		\end{split}
	\end{equation}
    Taking $\nabla$ on $(\ref{eq:n0v0u0})_{2}$, multiplying by $\nabla v^0$, and applying Lemma \ref{lem higher regurality}, we obtain
    \begin{align}\label{nabla v0}
        \frac{1}{2}\frac{d}{dt}\|\n v^0\|^2_{L^2_{xy}}&=-\int^\infty_0\int^\infty_{-\infty}\nabla^2(u^0\cdot v^0):\nabla v^0+\int^\infty_0\int^\infty_{-\infty}\nabla^2n^0:\nabla v^0\notag\\
        &\leq C\left(\|u^0\|_{H^2_{xy}}\|v^0\|_{H^2_{xy}}\|\nabla v^0\|_{L^2_{xy}}+\|\nabla^2 n^0\|_{L^2_{xy}}\|\nabla v^0\|_{L^2_{xy}}\right).
	\end{align}
    In order to estimate $\nabla u^0$,  taking $\nabla\times$ on the equation $(\ref{eq:n0v0u0})_{4}$, we get
    \begin{equation}\label{eq:omega}
		\begin{split}{}
        \partial_t\omega^0+u^0\cdot\nabla\omega^0=\partial_xn^0,
		\end{split}
	\end{equation}
    and it follows that
    \begin{equation}\label{omega L2}
		\begin{split}{}
        \frac{1}{2}\frac{d}{dt}\|\omega^0\|^2_{L^2_{xy}}&=-\int^\infty_0\int^\infty_{-\infty}u^0\cdot\nabla\omega^0\omega^0dxdy+\int^\infty_0\int^\infty_{-\infty}\partial_xn^0\omega^0dxdy\\
        &\leq \|\partial_xn^0\|_{L^2_{xy}}\|\omega^0\|_{L^2_{xy}}.
		\end{split}
	\end{equation}
    Applying $\nabla$ to (\ref{eq:omega}), we conclude that
    \begin{equation}\label{nabla omega}
		\begin{split}{}
        \frac{1}{2}\frac{d}{dt}\|\nabla\omega^0\|^2_{L^2_{xy}}&\leq C\|\nabla u^0\|_{H^2_{xy}}\|\nabla\omega^0\|^2_{L^2_{xy}}+C\|\nabla\partial_xn^0\|_{L^2_{xy}}\|\nabla\omega^0\|_{L^2_{xy}}.\\
		\end{split}
	\end{equation}
     Combining $(\ref{nabla n0})$-$(\ref{nabla v0})$ and $(\ref{omega L2})$-$(\ref{nabla omega})$,  by Lemma \ref{lem nabla omega} we have
    \begin{equation}\label{nabla n0 v0 omega}
		\begin{split}{}
        \frac{d}{dt}\left(\|\n n^0\|^2_{L^2_{xy}}+\|\n v^0\|^2_{L^2_{xy}}+\|\omega^0\|^2_{L^2_{xy}}+\|\n \omega^0\|^2_{L^2_{xy}}\right)\leq C(E(t)+1)^2+\frac{1}{8}\|\nabla\partial_t n^0\|^2_{L^2_{xy}}.
		\end{split}
	\end{equation}
    \textbf{Step 3. The derivative estimates with respect to $t$.} Recall that  we also need to estimate $\|\nabla\partial_t n^0\|^2_{L^2_{xy}}$ in $Step\; 2$. Taking $\partial_t$ on $(\ref{eq:n0v0u0})_{1}$, we get
    \begin{equation}\label{eq:n0tt}
		\begin{split}{}
        \partial_{tt}n^0+\partial_tu^0\cdot\nabla n^0+u^0\cdot\nabla\partial_t n^0=\nabla\cdot\left(\partial_t\nabla n^0+\partial_t(n^0v^0)\right).
		\end{split}
	\end{equation}
    Then, multiplying $(\ref{eq:n0tt})$ by $\partial_tn^0$ and using the conditions $(\ref{eq:n0v0u0})_{6}$, we  show that
    \begin{align}\label{Equ(nt)}
        &\frac{1}{2}\frac{d}{dt}\int^\infty_0\int^\infty_{-\infty}|\partial_t n^0|^2dxdy\notag\\
        &=\int^\infty_0\int^\infty_{-\infty}\nabla\cdot(\nabla\partial_t n^0+\partial_t(n^0v^0))\partial_tn^0dxdy-\int^\infty_0\int^\infty_{-\infty}(\partial_tu^0\cdot\nabla n^0+u^0\cdot\nabla\partial_t n^0)\partial_tn^0dxdy\notag\\
        &\leq-\int^\infty_0\int^\infty_{-\infty}|\nabla\partial_t n^0|^2dxdy+C\left(\| \partial_tn^0\|^2_{L^2_{xy}}\|v^0\|^2_{H^2_{xy}}+\| n^0\|^2_{H^2_{xy}}\|\partial_tv^0\|^2_{L^2_{xy}}\right)\\
        &\quad+C\left(\|\partial_tu^0\|^2_{H^1_{xy}}\|\nabla n^0\|^2_{H^1_{xy }}+\|\partial_t n^0\|^2_{L^2_{xy}}\right)+\frac{1}{2}\|\nabla\partial_t n^0\|^2_{L^2_{xy}}.\notag\notag
	\end{align}
    Further, differentiating $(\ref{eq:n0v0u0})_{2}$ with respect to $t$ and multiplying by $\partial_tv^0$, we find that
    \begin{align}\label{Equ(vt)}
        &\frac{1}{2}\frac{d}{dt}\|\partial_tv^0\|^2_{L^2_{xy}}\notag\\
        &=-\int^\infty_0\int^\infty_{-\infty}\nabla\partial_t(u^0\cdot v^0)\cdot\partial_tv^0dxdy+\int^\infty_0\int^\infty_{-\infty}\nabla\partial_tn^0\cdot\partial_tv^0dxdy\notag\\
        &=-\int^\infty_0\int^\infty_{-\infty}\left[\left(\nabla\partial_tu^0\cdot v^0+\nabla u^0\cdot\partial_t v^0\right)\cdot \partial_tv^0+\partial_tv^0\cdot (\nabla v^0)\cdot\partial_tu^0+\partial_tv^0\cdot (\nabla\partial_tv^0)\cdot u^0\right]dxdy\notag\\
        &\quad+\int^\infty_0\int^\infty_{-\infty}\nabla\partial_tn^0\cdot\partial_tv^0dxdy\\
        &\leq C\|\nabla \partial_tu^0\|^2_{L^2_{xy}}\|v^0\|^2_{H^2_{xy}}+C\|\nabla u^0\|_{H^2_{xy}}\|\partial_tv^0\|^2_{L^2_{xy}}+C\| \partial_tu^0\|^2_{H^1_{xy}}\|\nabla v^0\|^2_{H^1_{xy}}\notag\\
        &\quad+C\|\partial_t v^0\|^2_{L^2_{xy}}+\frac{1}{2}\|\nabla\partial_t n^0\|^2_{L^2_{xy}},\notag\notag
	\end{align}
   where we used that
   \begin{align}\label{curl transformation term}
      &\int^\infty_0\int^\infty_{-\infty}\partial_tv^0\cdot (\nabla\partial_tv^0)\cdot u^0dxdy\notag\\
      &=\int^\infty_0\int^\infty_{-\infty}\left(\partial_tv^0_1\partial_x\partial_tv^0_1u^0_1+\partial_tv^0_1\partial_x\partial_tv^0_2u^0_2+\partial_tv^0_2\partial_y\partial_tv^0_1u^0_1+\partial_tv^0_2\partial_y\partial_tv^0_2u^0_2\right)dxdy\notag\\
      &=\frac{1}{2}\int^\infty_0\int^\infty_{-\infty}\left(u^0_1\partial_x|\partial_tv^0_1|^2+u^0_2\partial_y|\partial_tv^0_1|^2+u^0_1\partial_x|\partial_tv^0_2|^2+u^0_2\partial_y|\partial_tv^0_2|^2\right)dxdy\\
   &=\frac{1}{2}\int^\infty_0\int^\infty_{-\infty}u^0 \cdot \nabla|\partial_tv^0|^2 dxdy\notag\\
   &=0,\notag\notag
   \end{align}
   since $\nabla\cdot u^0=0$ and $\nabla\times v^0=\partial_xv^0_2-\partial_yv^0_1=0$ in $(\ref{eq:n0v0u0})$. Next we estimate $\|\partial_tu^0\|_{L^2_{xy}}$ and $\|\nabla\partial_tu^0\|_{L^2_{xy}}$ as follows
    \begin{equation}\label{Equ(ut)}
		\begin{split}{}
        \frac{1}{2}\frac{d}{dt}\|\partial_tu^0\|^2_{L^2_{xy}}&\leq
        \| \partial_tu^0\|^2_{L^4_{xy}}\| \nabla u^0\|_{L^2_{xy}}+\|\partial_tn^0\|_{L^2_{xy}}\| \partial_tu^0\|_{L^2_{xy}}\\
        &\leq C\left(\| \partial_tu^0\|^2_{H^1_{xy}}\| \nabla u^0\|_{L^2_{xy}}+\|\partial_tu^0\|^2_{L^2_{xy}}+\|\partial_tn^0\|^2_{L^2_{xy}}\right),\\
        \frac{1}{2}\frac{d}{dt}\|\partial_t\omega^0\|^2_{L^2_{xy}}&\leq C\| \partial_tu^0\|^2_{H^1_{xy}}\|\nabla \omega^0\|^2_{H^1_{xy}}+\frac{1}{2}\|\partial_x\partial_tn^0\|^2_{L^2_{xy}}+C\|\partial_t\omega^0\|^2_{L^2_{xy}}.
		\end{split}
	\end{equation}
    Combining $(\ref{Equ(nt)})$-$(\ref{Equ(ut)})$, we obtain
    \begin{equation}\label{tn}
		\begin{split}{}
        \frac{d}{dt}(18\|\partial_tn^0\|^2_{L^2_{xy}}+\|\partial_tv^0\|^2_{L^2_{xy}}+\|\partial_tu^0\|^2_{L^2_{xy}}+\|\partial_t\omega^0\|^2_{L^2_{xy}})+16\|\nabla\partial_t n^0\|^2_{L^2_{xy}}\leq C(E(t)+1)^2.
		\end{split}
	\end{equation}

    \textbf{Step 4. The higher order derivative estimates of $\|\nabla^2 n^0\|^2_{L^2_{xy}}$, $\|\nabla^2 v^0\|^2_{L^2_{xy}}$ and $\|\nabla^2 \omega^0\|^2_{L^2_{xy}}$.}

    First, taking $\partial_y$ on $(\ref{eq:n0v0u0})_{1}$, multiplying by $\partial_y\partial_tn^0$ and applying integration by parts, we find that
    \begin{align*}
        &\frac{1}{2}\frac{d}{dt}\|\n\partial_y n^0\|^2_{L^2_{xy}}+\|\partial_t\partial_y n^0\|^2_{L^2_{xy}}\notag\\
        &=-\int^\infty_{-\infty}\partial^2_y n^0(t,x,0) \partial_t\partial_y n^0(t,x,0)dx-\int^\infty_0\int^\infty_{-\infty}(\partial_yu^0\cdot\nabla n^0+u^0\cdot\nabla\partial_y n^0)\partial_t\partial_y n^0dxdy\notag\\
        &\quad+\int^\infty_0\int^\infty_{-\infty}\nabla\cdot(\partial_yn^0v^0+n^0\partial_y v^0)\partial_t\partial_y n^0dxdy\\
        &=:I_1+I_2+I_3.
	\end{align*}
   Substituting the boundary conditions $(\partial_yn^{0}+n^{0}v^{0}_2)|_{y=0}=0$ into $I_1$, and applying Lemma \ref{lem trace theorem} together with Young¡¯s inequality, we have
   \begin{align*}
        I_1&=\int^\infty_{-\infty}\partial^2_y n^0(t,x,0) \partial_t(n^{0}(t,x,0)v^{0}_2(t,x,0))dx\\
		&\leq C\|\partial^2_y n^0\|^\frac{1}{2}_{L^2_{xy}}\|\partial^3_y n^0\|^\frac{1}{2}_{L^2_{xy}}\|\partial_t(n^0v^0_2)\|^\frac{1}{2}_{L^2_{xy}}\|\partial_t\partial_y(n^0v^0_2)\|^\frac{1}{2}_{L^2_{xy}}\notag\\
        &\leq C\|\partial^2_y n^0\|^2_{L^2_{xy}}+\frac{1}{16}\|\partial^3_y n^0\|^2_{L^2_{xy}}+(\|\partial_tn^0\|_{L^2_{xy}}\|v^0_2\|_{L^\infty_{xy}}+\|n^0\|_{L^\infty_{xy}}\|\partial_tv^0_2\|_{L^2_{xy}})(\|\partial_t\partial_yn^0\|_{L^2_{xy}}\|v^0_2\|_{L^\infty_{xy}}\\
        &\quad+\|\partial_yn^0\|_{L^4_{xy}}\|\partial_tv^0_2\|_{L^4_{xy}}+\|\partial_tn^0\|_{L^4_{xy}}\|\partial_yv^0_2\|_{L^4_{xy}}+\|n^0\|_{L^\infty_{xy}}\|\partial_t\partial_yv^0_2\|_{L^2_{xy}})\\
        &\leq C\|\nabla^2 n^0\|^2_{L^2_{xy}}+\frac{1}{16}\|\nabla^3n^0\|^2_{L^2_{xy}}+\frac{1}{16}\|\nabla\partial_tn^0\|^2_{L^2_{xy}}+C(E(t)+1)^3\notag\\
        &\quad+  C(E(t)+1)^{\frac32}+\|\nabla\partial_tv^0\|_{L^2_{xy}})
        \\
        &\leq C\|\nabla^2 n^0\|^2_{L^2_{xy}}+\frac{1}{16}\|\nabla^3n^0\|^2_{L^2_{xy}}+\frac{1}{16}\|\nabla\partial_tn^0\|^2_{L^2_{xy}}+C(E(t)+1)^3,
	\end{align*}
    where we  estimate $\|\nabla\partial_tv^0\|_{L^2_{T}L^2_{xy}}$ using $(\ref{eq:n0v0u0})_{2}$ and
    \begin{align}\label{(yv)}
       \|\nabla\partial_t v^0\|_{L^2_{xy}}&\leq C\|u^0\|_{H^2_{xy}}\| v^0\|_{H^2_{xy}}+\|\nabla^2n^0\|_{L^2_{xy}}.\notag\notag
    \end{align}
    By Young's inequality and the Sobolev embedding inequality, we obtain
    \begin{align*}
        I_2&\leq\left(\|\partial_yu^0\|_{L^\infty_{xy}}\|\nabla n^0\|_{L^2_{xy}}+\|u^0\|_{L^\infty_{xy}}\|\nabla\partial_y n^0\|_{L^2_{xy}}\right)\|\partial_t\partial_y n^0\|_{L^2_{xy}}\\
        &\leq C(E(t)+1)^2+\frac{1}{2}\|\partial_t\partial_y n^0\|^2_{L^2_{xy}}.
	\end{align*}
    Similarly, we estimate $I_3$ as follows
    \begin{align*}
        I_3&\leq \left(\|\nabla\partial_yn^0\|_{L^2_{xy}}\|v^0\|_{L^\infty_{xy}}+\|\partial_yn^0\|_{L^4_{xy}}\|\nabla v^0\|_{L^4_{xy}}+\|\nabla n^0\|_{L^4_{xy}}\|\partial_y v^0\|_{L^4_{xy}}+\|n^0\|_{L^\infty_{xy}}\|\nabla\partial_y v^0\|_{L^2_{xy}}\right)\\
        &\quad\times\|\partial_t\partial_y n^0\|_{L^2_{xy}},\notag\\
        &\leq C(E(t)+1)^2+\frac{1}{2}\|\partial_t\partial_y n^0\|^2_{L^2_{xy}}.
	\end{align*}
    Combining the terms $I_1$-$I_3$ gives
    \begin{equation}\label{tyn}
\begin{split}{}
        \frac{1}{2}\frac{d}{dt}\|\n\partial_y n^0\|^2_{L^2_{xy}}\leq \frac{1}{16}\|\nabla^3n^0\|^2_{L^2_{xy}}+\frac{1}{16}\|\nabla\partial_tn^0\|^2_{L^2_{xy}}+C(E(t)+1)^3.
        \end{split}{}
	\end{equation}
    For $\nabla^2n^0$, we also need to estimate $\partial^2_xn^0$. Applying $\partial^2_x$ to $(\ref{eq:n0v0u0})_{1}$ and multiplying by $\partial^2_xn^0$, we deduce that
    \begin{equation}\label{xxn}
		\begin{split}{}
        &\frac{1}{2}\frac{d}{dt}\|\partial^2_x n^0\|^2_{L^2_{xy}}+\|\n\partial^2_x n^0\|^2_{L^2_{xy}}\\
        &=-\int^\infty_0\int^\infty_{-\infty}\partial^2_x(n^0v^0)\nabla\partial^2_x n^0dxdy-\int^\infty_0\int^\infty_{-\infty}\partial^2_x(u^0\cdot\nabla n^0)\partial^2_x n^0dxdy\\
        &\leq \frac{1}{8}\|\nabla\partial^2_xn^0\|^2_{L^2_{xy}}+C\|n^0\|^2_{H^2_{xy}}\|v^0\|^2_{H^2_{xy}}+C\|u^0\|^2_{H^3_{xy}}\| n^0\|^2_{H^2_{xy}}+C\|\partial^2_x n^0\|^2_{L^2_{xy}}\\
        &\leq C(E(t)+1)^2+\frac{1}{8}\|\nabla\partial^2_xn^0\|^2_{L^2_{xy}}.
		\end{split}
	\end{equation}
    Combining the estimates for $(\ref{tyn})$-$(\ref{xxn})$, we arrive at
    \begin{align}\label{n2}
        &\frac{1}{2}\frac{d}{dt}\|\nabla^2 n^0\|^2_{L^2_{xy}}+\|\n\partial^2_x n^0\|^2_{L^2_{xy}}\notag\\
        &\leq \frac{1}{4}\|\n\partial^2_x n^0\|^2_{L^2_{xy}}+\frac{1}{16}\|\nabla^3 n^0\|^2_{L^2_{xy}}+\frac{1}{16}\|\nabla\partial_tn^0\|^2_{L^2_{xy}}+C(E(t)+1)^3,
	\end{align}
    where we need estimate $\|\nabla^3 n^0\|^2_{L^2_TL^2_{xy}}$. It follows  from $(\ref{eq:n0v0u0})_{1}$ that $$\partial^2_yn^0=\partial_tn^{0}+u^{0}\cdot\nabla n^{0}-\nabla\cdot(n^{0}v^{0})-\partial^2_xn^0,$$ and hence
    \begin{align}\label{Equ(n3)}
        \|\nabla^3 n^0\|_{L^2_{xy}}
        &\leq \|\partial^3_xn^0\|_{L^2_{xy}}+\|\partial^2_x\partial_yn^0\|_{L^2_{xy}}+\|\partial_x\partial^2_yn^0\|_{L^2_{xy}}+\|\partial^3_yn^0\|_{L^2_{xy}}\notag\\
        &\leq \|\partial^3_xn^0\|_{L^2_{xy}}+\|\partial^2_x\partial_yn^0\|_{L^2_{xy}}+\|\partial_x(\partial_tn^{0}+u^{0}\cdot\nabla n^{0}-\nabla\cdot(n^{0}v^{0})-\partial^2_xn^0)\|_{L^2_{xy}}\notag\\
        &\quad+\|\partial_y(\partial_tn^{0}+u^{0}\cdot\nabla n^{0}-\nabla\cdot(n^{0}v^{0})-\partial^2_xn^0)\|_{L^2_{xy}}\notag\\
        &\leq2\|\partial^3_xn^0\|_{L^2_{xy}}+2\|\partial^2_x\partial_yn^0\|_{L^2_{xy}}+\|\partial_x\partial_t n^0\|_{L^2_{xy}}+\|\partial_y\partial_t n^0\|_{L^2_{xy}}+C\|u^0\|_{H^3_{xy}}\| n^0\|_{H^2_{xy}}\notag\\
     &\quad+C\|n^0\|_{H^2_{xy}}\|v^0\|_{H^2_{xy}}\\
        &\leq 4\|\nabla\partial^2_xn^0\|_{L^2_{xy}}+2\|\nabla\partial_t n^0\|_{L^2_{xy}}+C\|u^0\|_{H^3_{xy}}\| n^0\|_{H^2_{xy}}+C\|n^0\|_{H^2_{xy}}\|v^0\|_{H^2_{xy}}.\notag\notag
	\end{align}
    Next, we estimate $\|\nabla^2 v^0\|^2_{L^2_{xy}}$ and $\|\nabla^2 \omega^0\|^2_{L^2_{xy}}$. Applying $\nabla^2$ to $(\ref{eq:n0v0u0})_{2}$ and multiplying by $\nabla^2 v^0$, we deduce that
    \begin{equation}\label{v2}
		\begin{split}{}
        \frac{1}{2}\frac{d}{dt}\|\nabla^2v^0\|^2_{L^2_{xy}}
        &\leq \|\nabla^3u^0\|^2_{L^2_{xy}}\|v^0\|^2_{L^\infty_{xy}}+\|\nabla^2 u^0\|^2_{L^4_{xy}}\|\nabla v^0\|^2_{L^4_{xy}}+\|\nabla u^0\|^2_{L^\infty_{xy}}\|\nabla^2v^0\|^2_{L^2_{xy}}\\
        &\quad+\|\nabla^2 v^0\|^2_{L^2_{xy}}+\|\nabla^3 n^0\|^2_{L^2_{xy}}\\
        &\leq 4\|\nabla\partial^2_xn^0\|^2_{L^2_{xy}}+2\|\nabla\partial_t n^0\|^2_{L^2_{xy}}+C(E(t)+1)^2.\\
		\end{split}
	\end{equation}
    Similarly, applying $\nabla^2$ to $(\ref{eq:omega})$ and multiplying by $\nabla^2\omega^0$, one can show that
    \begin{equation}\label{u3}
		\begin{split}{}
        \frac{1}{2}\frac{d}{dt}\|\nabla^2\omega^0\|^2_{L^2_{xy}}
        &\leq \|\nabla^2u^0\|^2_{L^4_{xy}}\|\nabla \omega^0\|^2_{L^4_{xy}}+\|\nabla u^0\|^2_{L^\infty_{xy}}\|\nabla^2 \omega^0\|^2_{L^2_{xy}}+\|\nabla^2\omega^0\|^2_{L^2_{xy}}+\|\nabla^3 n^0\|^2_{L^2_{xy}}\\
        &\leq 4\|\nabla\partial^2_xn^0\|^2_{L^2_{xy}}+2\|\nabla\partial_t n^0\|^2_{L^2_{xy}}+C(E(t)+1)^2.
		\end{split}
	\end{equation}
    Combining $(\ref{n2})$, $(\ref{v2})$ and  $(\ref{u3})$, we have
    \begin{equation}\label{nvu}
		\begin{split}{}
        \frac{d}{dt}(18\|\nabla^2n^0\|^2_{L^2_{xy}}+\|\nabla^2v^0\|^2_{L^2_{xy}}+\|\nabla^2\omega^0\|^2_{L^2_{xy}})+2\|\nabla\partial^2_xn^0\|^2_{L^2_{xy}}
        &\leq C(E(t)+1)^3+15\|\nabla\partial_t n^0\|^2_{L^2_{xy}}.
		\end{split}
	\end{equation}
     Therefore, combining $(\ref{n0 v0 u0 L2 estimate})$, $(\ref{nabla n0 v0 omega})$, $(\ref{tn})$, $(\ref{nvu})$ and (\ref{energy functional}), we have
     \begin{equation*}
		\begin{split}{}
        \frac{d}{dt}E(t)+\|\nabla\partial_tn^0\|^2_{L^2_{xy}}+\|\n\partial^2_x n^0\|^2_{L^2_{xy}}\leq C(E(t)+1)^3,
		\end{split}
	\end{equation*}
    for any $t\in[0,T]$, where
    $$
    T<\left[{2C\left(\| (n_{in}, v_{in})\|^2_{H^2_{xy}}+\|u_{in}\|^2_{H^3_{xy}}+\|(\partial_tn_{in}, \partial_tv_{in})\|^2_{L^2_{xy}}+\|\partial_tu_{in}\|^2_{H^1_{xy}}+1\right)^2}\right]^{-1}.
    $$
     Using Gronwall¡¯s inequality together with Lemma \ref{lem nabla omega}, we show that $(\ref{n0 v0 u0 H2})$ holds.
     Using (\ref{eq:n0v0u0}), (\ref{eq:omega}) and $(\ref{n0 v0 u0 H2})$, we obtain
    \begin{equation}\label{nabla v0t nabla omega0t}
		\begin{split}{}
        \|\nabla\partial_t v^0\|_{L^\infty_TL^2_{xy}}&\leq\|u^0\|_{L^\infty_TH^2_{xy}}\| v^0\|_{L^\infty_TH^2_{xy}}+\|\nabla^2n^0\|_{L^\infty_TL^2_{xy}}\leq C,
        \\
        \|\nabla\partial_t \omega^0\|_{L^\infty_TL^2_{xy}}&\leq\|u^0\|_{L^\infty_TH^3_{xy}}\| \nabla\omega^0\|_{L^\infty_TL^2_{xy}}+\|u^0\|_{L^\infty_TH^2_{xy}}\| \nabla^2\omega^0\|_{L^\infty_TL^2_{xy}}+\|\nabla^2n^0\|_{L^\infty_TL^2_{xy}}\leq C.\\
    \end{split}
	\end{equation}
     Furthermore, integrating both sides of (\ref{Equ(n3)}) over the time, then it yields
     $$
     \|\n^3 n^0\|_{L^2_TL^2_{xy}}\leq C.
     $$
    Clearly, this yields (\ref{nabla3 n0 L2T}). The proof is complete.

    Next we aim to obtain estimates for higher-order derivatives in Lemma \ref{n0 v0 u0 L2}.
    \begin{lemma}\label{nabla3 n0}
         {Suppose the initial data $(n_{in},v_{in},u_{in})$ and the solution $(n^0,v^0,u^0)$ satisfy the assumptions of  Lemma \ref{n0 v0 u0 L2}. Then  the following properties holds:
		\begin{align}\label{nabla3 tt n0 v0 omega}
        &\|\nabla\partial_tn^0\|^2_{L^\infty_TL^2_{xy}}+\|(\partial^2_tn^0,\partial^2_t v^0,\partial^2_t u^0,\nabla \partial^2_t u^0)\|^2_{L^\infty_TL^2_{xy}}\notag\\
        &+\|(\nabla\partial^2_x n^0,\partial^2_x \partial_tn^0,\nabla^3 v^0,\nabla^4 u^0)\|^2_{L^\infty_TL^2_{xy}}+\|\nabla\partial^3_xn^0\|^2_{L^2_TL^2_{xy}}\leq C,
		\end{align}
        and
        \begin{align}\label{nabla3 n0 L infty nabla4 n0 L2}
        \|(\nabla^2\partial_t v^0,\nabla^3\partial_t u^0,\nabla^3 n^0,\nabla^2\partial_t n^0)\|_{L^\infty_TL^2_{xy}}+\|\nabla^4 n^0\|_{L^2_TL^2_{xy}}\leq C,
    \end{align}
    where $C>0$ is a constant depending only on the initial data
    $$    (\|n_{in}\|_{H^3_{xy}},\|v_{in}\|_{H^3_{xy}},\|u_{in}\|_{H^4_{xy}},\|\partial_tn_{in}\|_{H^2_{xy}},\|\partial^2_x\partial_tn_{in}\|_{L^2_{xy}},\|(\partial^2_tn_{in},\partial^2_tv_{in},\partial^2_tu_{in})\|_{L^2_{xy}},\|\partial^2_tu_{in}\|_{H^1_{xy}}).
    $$
		}
	\end{lemma}
        \noindent{\bf{Proof.}}  Deriving the estimates in (\ref{nabla3 tt n0 v0 omega})-(\ref{nabla3 n0 L infty nabla4 n0 L2}) requires us to decompose the proof into three steps.
Next, we will eatimate each term step by step.

       \textbf{Step 1. I. The estimates of $\|\nabla\partial_t n^0\|_{L^\infty_TL^2_{xy}}$.} We write the two equations for $n^0$ as follows:
       \begin{equation}\label{eqt2n}
		\begin{split}{}
        \partial^2_{t}n^0+\partial_t(u^0\cdot\nabla n^0)=\nabla\cdot(\partial_t\nabla n^0+\partial_t(n^0v^0)).
		\end{split}
	\end{equation}
    \begin{equation}\label{eqt3n}
		\begin{split}{}
        \partial^3_{t}n^0+\partial^2_{t}(u^0\cdot\nabla n^0)=\nabla\cdot(\partial^2_{t}\nabla n^0+\partial^2_{t}(n^0v^0)).
		\end{split}
	\end{equation}
    Taking the inner product of $(\ref{eqt2n})$ with $\partial^2_t n^0$,
    \begin{equation*}
		\begin{split}{}
        \|\partial^2_t n^0\|^2_{L^2_{xy}}=\int^\infty_0\int^\infty_{-\infty}\nabla\cdot\partial_t(\nabla n^0+n^0v^0)\partial^2_t n^0dxdy-\int^\infty_0\int^\infty_{-\infty}\partial_t(u^0\cdot\nabla n^0)\partial^2_t n^0dxdy,
		\end{split}
	\end{equation*}
    integrating by parts and by $(\ref{n0 v0 u0 H2})$, we obtain
    \begin{align}\label{nabla n0t}
        &\frac{1}{2}\frac{d}{dt}\|\n\partial_t n^0\|^2_{L^2_{xy}}+\|\partial^2_t n^0\|^2_{L^2_{xy}}\notag\\
    &\leq(\|\partial_tn^0\|_{L^2_{xy}}\|v^0\|_{L^\infty_{xy}}+\|n^0\|_{L^\infty_{xy}}\|\partial_tv^0\|_{L^2_{xy}})\|\nabla\partial^2_tn^0\|_{L^2_{xy}}+\|\partial_t u^0\|_{L^4_{xy}}\|\nabla n^0\|_{L^4_{xy}}\|\partial^2_tn^0\|_{L^2_{xy}}\\
    &\quad+\|u^0\|_{L^\infty_{xy}}\|\nabla\partial_t n^0\|_{L^2_{xy}}\|\partial^2_tn^0\|_{L^2_{xy}}\notag\\
    &\leq\frac{1}{8}\|\nabla\partial^2_t n^0\|^2_{L^2_{xy}}+\frac{1}{2}\|\partial^2_t n^0\|^2_{L^2_{xy}}+C\|\nabla\partial_t n^0\|^2_{L^2_{xy}}+C,\notag\notag
	\end{align}
    where we need to estimate $\|\nabla\partial^2_tn^0\|_{L^2_{T}L^2_{xy}}$.\\
     \textbf{II. The second temporal derivatives of $n^0$, $v^0$ and $u^0$.} We take the inner product of $(\ref{eqt3n})$ with $\partial^2_t n^0$ and apply the Sobolev embedding inequality to derive
    \begin{equation*}
		\begin{split}{}
        &\frac{1}{2}\frac{d}{dt}\|\partial^2_t n^0\|^2_{L^2_{xy}}+\|\n\partial^2_t n^0\|^2_{L^2_{xy}}\\
    &\leq\left(\|\partial^2_tn^0\|_{L^2_{xy}}\|v^0\|_{L^\infty_{xy}}+2\|\partial_tn^0\|_{L^4_{xy}}\|\partial_tv^0\|_{L^4_{xy}}+\|n^0\|_{L^\infty_{xy}}\|\partial^2_tv^0\|_{L^2_{xy}}\right)\|\nabla\partial^2_tn^0\|_{L^2_{xy}}\\
    &\quad+\|\partial^2_t u^0\|_{L^4_{xy}}
    \|\nabla n^0\|_{L^4_{xy}}\|\partial^2_tn^0\|_{L^2_{xy}}+2\|\partial_t u^0\|_{L^\infty_{xy}}\|\nabla\partial_t n^0\|_{L^2_{xy}}\|\partial^2_tn^0\|_{L^2_{xy}}\\
    &\leq\left(\|\partial^2_tn^0\|_{L^2_{xy}}\|v^0\|_{H^2_{xy}}+2\|\partial_tn^0\|^\frac{1}{2}_{L^2_{xy}}\|\nabla\partial_tn^0\|^\frac{1}{2}_{L^2_{xy}}\|\partial_tv^0\|_{H^1_{xy}}+\|n^0\|_{H^2_{xy}}\|\partial^2_tv^0\|_{L^2_{xy}}\right)\|\nabla\partial^2_tn^0\|_{L^2_{xy}}\\
    &\quad+\|\partial^2_t u^0\|_{H^1_{xy}}
    \| n^0\|_{H^2_{xy}}\|\partial^2_tn^0\|_{L^2_{xy}}+2\|\partial_t u^0\|_{H^2_{xy}}\|\nabla\partial_t n^0\|_{L^2_{xy}}\|\partial^2_tn^0\|_{L^2_{xy}}\\
    &\leq\frac{1}{8}\|\nabla\partial^2_t n^0\|^2_{L^2_{xy}}+C(\|\partial^2_t n^0\|^2_{L^2_{xy}}+\|\nabla\partial_t n^0\|^2_{L^2_{xy}}+\|\partial^2_t v^0\|^2_{L^2_{xy}}+\|\partial^2_t u^0\|^2_{H^1_{xy}}).
		\end{split}
	\end{equation*}
    which implies
    \begin{align}\label{t2n}
        &\frac{1}{2}\frac{d}{dt}\|\partial^2_t n^0\|^2_{L^2_{xy}}+\frac{3}{8}\|\n\partial^2_t n^0\|^2_{L^2_{xy}}\notag\\
        &\leq C(\|\nabla\partial_tn^0\|^2_{L^2_{xy}}+\|\partial^2_tn^0\|^2_{L^2_{xy}}+\|\partial^2_t v^0\|^2_{L^2_{xy}}+\|\partial^2_t u^0\|^2_{H^1_{xy}}+1).
	\end{align}
    Moreover, for the term $\partial^2_t v^0$, applying $\partial^2_t$ to $(\ref{eq:n0v0u0})_2$, we get
    \begin{equation}\label{Equ(vt3)}
		\begin{split}{}
        \partial^3_tv^0+\partial^2_t\nabla(u^0\cdot v^0)=\partial^2_t\nabla n^0,
		\end{split}
	\end{equation}
     and multiplying $(\ref{Equ(vt3)})$ by $\partial^2_t v^0$, we find that
    \begin{align}\label{t2v}
        \frac{1}{2}\frac{d}{dt}\|\partial^2_t v^0\|^2_{L^2_{xy}}&\leq\|\nabla\partial^2_tu^0\|^2_{L^2_{xy}}\| v^0\|^2_{L^\infty_{xy}}+\|\partial^2_tu^0\|^2_{L^4_{xy}}\|\nabla v^0\|^2_{L^4_{xy}}+2\|\nabla\partial_tu^0\|^2_{L^4_{xy}}\|\partial_tv^0\|^2_{L^4_{xy}}\notag\\
        &\quad+2\|\partial_tu^0\|^2_{L^\infty_{xy}}\|\nabla\partial_tv^0\|^2_{L^2_{xy}}+\|\nabla u^0\|^2_{L^\infty_{xy}}\|\partial^2_tv^0\|^2_{L^2_{xy}}+\frac{1}{8}\|\nabla\partial^2_tn^0\|^2_{L^2_{xy}}+C\|\partial^2_tv^0\|^2_{L^2_{xy}}\\
        &\leq \frac{1}{8}\|\nabla\partial^2_t n^0\|^2_{L^2_{xy}}+C\left(\|\partial^2_tv^0\|^2_{L^2_{xy}}+\|\partial^2_t u^0\|^2_{H^1_{xy}}+1\right),\notag\notag
	\end{align}
    where we perform a derivation similar to (\ref{curl transformation term}), which gives
    \beno
   \int^\infty_0\int^\infty_{-\infty}\partial^2_tv^0\cdot (\nabla\partial^2_tv^0)\cdot u^0dxdy=\frac{1}{2}\int^\infty_0\int^\infty_{-\infty}u^0 \cdot \nabla|\partial^2_tv^0|^2 dxdy=0.
   \eeno
     Similarly, the estimate for $\partial^2_t u^0$ yields
     \begin{equation}\label{t2u}
		\begin{split}{}
        &\frac{1}{2}\frac{d}{dt}\|\partial^2_t u^0\|^2_{L^2_{xy}}\\
        &\leq\|\partial^2_tu^0\|^2_{L^2_{xy}}\| \nabla u^0\|_{L^\infty_{xy}}+2\|\partial_tu^0\|_{L^4_{xy}}\|\nabla \partial_tu^0\|_{L^2_{xy}}\|\partial^2_tu^0\|_{L^4_{xy}}+\|\partial^2_tn^0\|_{L^2_{xy}}\|\partial^2_tu^0\|_{L^2_{xy}}\\
        &\leq C\left(\|\partial^2_tn^0\|^2_{L^2_{xy}}+\|\partial^2_tu^0\|^2_{H^1_{xy}}\right).
		\end{split}
	\end{equation}
     For the estimates of $\nabla\partial^2_t u^0$, we begin with the equation of $\omega^0$
    \begin{align}\label{eq:omega t3}    \partial^3_t\omega^0+\partial^2_t(u^0\cdot\nabla\omega^0)=\partial^2_t\partial_xn^0.
	\end{align}
    Multiplying $(\ref{eq:omega t3})$ by $\partial^2_t \omega^0$, we obtain
    \begin{equation}\label{t2w}
		\begin{split}{}
        &\frac{1}{2}\frac{d}{dt}\|\partial^2_t \omega^0\|^2_{L^2_{xy}}\\
        &\leq\|\partial^2_tu^0\|^2_{L^4_{xy}}\|\nabla \omega^0\|^2_{L^4_{xy}}+2\|\partial_tu^0\|^2_{L^\infty_{xy}}\|\nabla\partial_t\omega^0\|^2_{L^2_{xy}}+C\|\partial^2_t\omega^0\|^2_{L^2_{xy}}+\frac{1}{8}\|\nabla\partial^2_t n^0\|^2_{L^2_{xy}}\\
        &\leq \frac{1}{8}\|\nabla\partial^2_t n^0\|^2_{L^2_{xy}}+C\left(\|\partial^2_t u^0\|^2_{H^1_{xy}}+\|\partial^2_t\omega^0\|^2_{L^2_{xy}}+1\right).
		\end{split}
	\end{equation}
    Combining $(\ref{nabla n0t})$-$(\ref{t2n})$, $(\ref{t2v})$-$(\ref{t2u})$ and $(\ref{t2w})$, we have
    \begin{equation*}
		\begin{split}{}
        &\frac{d}{dt}(\|\nabla\partial_tn^0\|^2_{L^2_{xy}}+\|\partial^2_tn^0\|^2_{L^2_{xy}}+\|\partial^2_t v^0\|^2_{L^2_{xy}}+\|\partial^2_t u^0\|^2_{L^2_{xy}}+\|\partial^2_t\omega^0\|^2_{L^2_{xy}})+\|\nabla\partial^2_t n^0\|^2_{L^2_{xy}}\\
        &\leq C(\|\nabla\partial_tn^0\|^2_{L^2_{xy}}+\|\partial^2_tn^0\|^2_{L^2_{xy}}+\|\partial^2_t v^0\|^2_{L^2_{xy}}+\|\partial^2_t u^0\|^2_{L^2_{xy}}+\|\partial^2_t\omega^0\|^2_{L^2_{xy}})+C.
		\end{split}
	\end{equation*}
Thus, by Gronwall's inequality and Lemma \ref{lem nabla omega}, we obtain the estimate for $t\in[0,T]$.
    \begin{equation*}
		\begin{split}{}
        &\|\nabla\partial_tn^0\|^2_{L^2_{xy}}+\|\partial^2_tn^0\|^2_{L^2_{xy}}+\|\partial^2_t v^0\|^2_{L^2_{xy}}+\|\partial^2_t u^0\|^2_{L^\infty_TL^2_{xy}}+\|\partial^2_t\nabla u^0\|^2_{L^\infty_TL^2_{xy}}+\int^t_0\|\nabla\partial^2_t n^0\|^2_{L^2_{xy}}\\
        &\leq C(\|n_{in}\|_{H^2_{xy}},\|v_{in}\|_{H^2_{xy}},\|u_{in}\|_{H^2_{xy}},\|\partial_tn_{in}\|^2_{H^1_{xy}},\|\partial_tu_{in}\|^2_{H^1_{xy}},\|\partial^2_tn_{in}\|^2_{L^2_{xy}},\|\partial^2_tv_{in}\|^2_{L^2_{xy}},\|\partial^2_tu_{in}\|^2_{H^1_{xy}}).
		\end{split}
	\end{equation*}
    This completes the estimates of Step 1.\\
    \textbf{Step 2. I. The estimates of  $\|\nabla\partial^2_xn^0\|_{L^\infty_TL^2_{xy}}$ and $\|\partial^2_x\partial_t n^0\|_{L^\infty_TL^2_{xy}}$.}
    Applying $\partial^2_x$ to $(\ref{eq:n0v0u0})_{1}$ and multiplying by $\partial^2_x\partial_tn^0$, we deduce that
    \begin{align}\label{dx2n}
        &\frac{1}{2}\frac{d}{dt}\|\n\partial^2_x n^0\|^2_{L^2_{xy}}+\|\partial_t\partial^2_x n^0\|^2_{L^2_{xy}}\notag\\
        &=-\int^\infty_0\int^\infty_{-\infty}\partial^2_x(n^0v^0)\cdot\nabla\partial^2_x\partial_t n^0dxdy-\int^\infty_0\int^\infty_{-\infty}\partial^2_x(u^0\cdot\nabla n^0)\partial_t\partial^2_x n^0dxdy\notag\\
        &\leq\frac{1}{2}\|\nabla\partial^2_x\partial_t n^0\|^2_{L^2_{xy}}+C\|n^0\|^2_{H^2_{xy}}\|v^0\|^2_{H^2_{xy}}+\frac{1}{2}\|\partial_t\partial^2_x n^0\|^2_{L^2_{xy}}+C\|u^0\|^2_{H^3_{xy}}\|n^0\|^2_{H^2_{xy}}+C\|u^0\|^2_{H^2_{xy}}\|\nabla\partial^2_xn^0\|^2_{L^2_{xy}}\notag\\
        &\leq\frac{1}{2}\|\nabla\partial^2_x\partial_t n^0\|^2_{L^2_{xy}}+\frac{1}{2}\|\partial_t\partial^2_x n^0\|^2_{L^2_{xy}}+C\|\nabla\partial^2_xn^0\|^2_{L^2_{xy}}+C.
	\end{align}
    where the key estimate is $\|\nabla\partial^2_x\partial_t n^0\|_{L^2_TL^2_{xy}}$.
    Differentiating $(\ref{eqt2n})$ twice in $x$ and multiplying the resulting equation by $\partial^2_x\partial_tn^0$, we obtain
    \begin{align}\label{n0t x2}
        &\frac{1}{2}\frac{d}{dt}\|\partial^2_x\partial_t n^0\|^2_{L^2_{xy}}+\|\n\partial^2_x\partial_t n^0\|^2_{L^2_{xy}}\notag\\
        &=-\int^\infty_0\int^\infty_{-\infty}\partial^2_x\partial_t(n^0v^0)\cdot\nabla\partial^2_x\partial_t n^0dxdy-\int^\infty_0\int^\infty_{-\infty}\partial^2_x\partial_t(u^0\cdot\nabla n^0)\partial^2_x\partial_t n^0dxdy\\
        &\leq \frac{1}{2}\|\nabla\partial^2_x\partial_tn^0\|^2_{L^2_{xy}}+C\|\partial^2_x\partial_t n^0\|^2_{L^2_{xy}}+C\|\nabla\partial^2_x n^0\|^2_{L^2_{xy}}+C\|\nabla^2\partial_t v^0\|^2_{L^2_{xy}}+C\|\nabla^3 v^0\|^2_{L^2_{xy}}+C\notag\\
        &\leq \frac{1}{2}\|\nabla\partial^2_x\partial_tn^0\|^2_{L^2_{xy}}+C\|\partial^2_x\partial_t n^0\|^2_{L^2_{xy}}+C\|\nabla\partial^2_x n^0\|^2_{L^2_{xy}}+C\|\nabla^3 v^0\|^2_{L^2_{xy}}+C,\notag\notag
	\end{align}
    due to
    \begin{align}\label{nabla2 vt}
        \|\nabla^2\partial_t v^0\|_{L^2_{xy}}&\leq\|u^0\|_{H^3_{xy}}\| v^0\|_{H^3_{xy}}+\|\nabla^3n^0\|_{L^2_{xy}}\leq C(\| \nabla^3v^0\|_{L^2_{xy}}+\|\nabla\partial^2_xn^0\|_{L^2_{xy}}+1).
	\end{align}
    \textbf{II. The estimates of  $\|\nabla^3v^0\|_{L^\infty_TL^2_{xy}}$ and $\|\nabla^3\omega^0\|_{L^\infty_TL^2_{xy}}$.} We now proceed to bound  $\| \nabla^3v^0\|_{L^2_{xy}}$.
    For the $v^0$ equation $(\ref{eq:n0v0u0})_2$, we apply $\nabla^3$ and take the inner product with $\nabla^3 v$.
    \begin{align}\label{nabla3 v0}
        \frac{1}{2}\frac{d}{dt}\|\nabla^3 v^0\|^2_{L^2_{xy}}&\leq\|u^0\|^2_{H^4_{xy}}\|v^0\|^2_{H^2_{xy}}+C\|\nabla^3 u^0\|^2_{H^1_{xy}}\|\nabla v^0\|^2_{H^1_{xy}}+C\|\nabla^2 u^0\|^2_{H^2_{xy}}\|\nabla^2 v^0\|^2_{L^2_{xy}}\notag\\
        &\quad+C\|\nabla u^0\|^2_{H^2_{xy}}\|\nabla^3 v^0\|^2_{L^2_{xy}}+C\|\nabla^3 v^0\|^2_{L^2_{xy}}+\|\nabla^4 n^0\|^2_{L^2_{xy}}\\
        &\leq C(\|\nabla^4u^0\|^2_{L^2_{xy}}+\|\nabla^3 v^0\|^2_{L^2_{xy}}+1)+\|\nabla^4 n^0\|^2_{L^2_{xy}},\notag\notag
	\end{align}
    where we need estimate $\|\nabla^4u^0\|^2_{L^2_{xy}}$ and $\|\nabla^4 n^0\|^2_{L^2_{xy}}$.
    To estimate $\|\nabla^4u^0\|^2_{L^2_{xy}}$, we apply $\nabla^3$ to $(\ref{eq:omega})$ and multiply by $\nabla^3\omega^0$, which yields
    \begin{align}\label{nabla3 omega}
        &\frac{1}{2}\frac{d}{dt}\|\nabla^3\omega^0\|^2_{L^2_{xy}}\notag\\
        &\leq \|\nabla^3u^0\|^2_{L^2_{xy}}\|\nabla \omega^0\|^2_{L^\infty_{xy}}+C\|\nabla^2 u^0\|^2_{L^4_{xy}}\|\nabla^2 \omega^0\|^2_{L^4_{xy}}+C\|\nabla u^0\|^2_{L^\infty_{xy}}\|\nabla^3 \omega^0\|^2_{L^2_{xy}}\\
        &\quad+C\|\nabla^3\omega^0\|^2_{L^2_{xy}}+\|\nabla^4n^0\|^2_{L^2_{xy}}\notag\\
        &\leq C\|\nabla^3\omega^0\|^2_{L^2_{xy}}+\|\nabla^4n^0\|^2_{L^2_{xy}}.\notag\notag
	\end{align}
    We now proceed to estimate $\|\nabla^4 n^0\|^2_{L^2_{xy}}$.
    \begin{align}\label{d4n}
        \|\nabla^4 n^0\|_{L^2_{xy}}
        &\leq \|\partial^4_xn^0\|_{L^2_{xy}}+\|\partial^3_x\partial_yn^0\|_{L^2_{xy}}+\|\partial^2_x\partial^2_yn^0\|_{L^2_{xy}}+\|\partial_x\partial^3_yn^0\|_{L^2_{xy}}+\|\partial^4_yn^0\|_{L^2_{xy}}\notag\\
        &\leq \|\nabla\partial^3_xn^0\|_{L^2_{xy}}+\|\partial^2_x(\partial_tn^{0}+u^{0}\cdot\nabla n^{0}-\nabla\cdot(n^{0}v^{0})-\partial^2_xn^0)\|_{L^2_{xy}}\notag\\
        &\quad+\|\partial_x\partial_y(\partial_tn^{0}+u^{0}\cdot\nabla n^{0}-\nabla\cdot(n^{0}v^{0})-\partial^2_xn^0)\|_{L^2_{xy}}\notag\\
        &\quad+\|\partial^2_y(\partial_tn^{0}+u^{0}\cdot\nabla n^{0}-\nabla\cdot(n^{0}v^{0})-\partial^2_xn^0)\|_{L^2_{xy}}\notag\\
        &\leq \|\nabla\partial^3_xn^0\|_{L^2_{xy}}+\|\partial^2_x\partial_tn^{0}+\partial^2_x(u^{0}\cdot\nabla n^{0})-\nabla\cdot\partial^2_x(n^{0}v^{0})-\partial^4_xn^0\|_{L^2_{xy}}\\
        &\quad+\|\partial_x\partial_y\partial_tn^{0}+\partial_x\partial_y(u^{0}\cdot\nabla n^{0})-\nabla\cdot\partial_x\partial_y(n^{0}v^{0})-\partial^3_x\partial_yn^0)\|_{L^2_{xy}}\notag\\
        &\quad+\|\partial_t(\partial_tn^{0}+u^{0}\cdot\nabla n^{0}-\nabla\cdot(n^{0}v^{0})-\partial^2_xn^0)+\partial^2_y(u^{0}\cdot\nabla n^{0})-\nabla\cdot\partial^2_y(n^{0}v^{0})\notag\\
        &\quad-\partial^2_x(\partial_tn^{0}+u^{0}\cdot\nabla n^{0}-\nabla\cdot(n^{0}v^{0})-\partial^2_xn^0)\|_{L^2_{xy}}\notag\\
        &\leq 4\|\nabla\partial^3_xn^0\|_{L^2_{xy}}+C\|\nabla\partial^2_xn^0\|_{L^2_{xy}}+C\|\partial^2_x\partial_tn^0\|_{L^2_{xy}}+C\|\partial_x\partial_y\partial_tn^0\|_{L^2_{xy}}+C\|\nabla^3v^0\|_{L^2_{xy}}+C).\notag\notag
	\end{align}
    In order to bound $\|\nabla^4 n^0\|^2_{L^2_{xy}}$, we shall estimate $\|\partial_x\partial_y\partial_tn^0\|_{L^2_TL^2_{xy}}$ and $\|\nabla\partial^3_xn^0\|_{L^2_TL^2_{xy}}$. By applying $\partial_x$ to (\ref{eqt2n}) and multiplying by $\partial_x\partial^2_tn^0$, we deduce the estimate for $\|\partial_x\partial_y\partial_tn^0\|_{L^2_TL^2_{xy}}$ is given as follows
    \begin{align}\label{nabla xt n}
        &\frac{1}{2}\frac{d}{dt}\|\n\partial_x\partial_t n^0\|^2_{L^2_{xy}}+\|\partial_x\partial^2_t n^0\|^2_{L^2_{xy}}\notag\\
        &=-\int^\infty_0\int^\infty_{-\infty}\partial_x\partial_t(n^0v^0)\cdot\nabla\partial_x\partial^2_t n^0dxdy-\int^\infty_0\int^\infty_{-\infty}\partial_x\partial_t(u^0\cdot\nabla n^0)\partial_x\partial^2_t n^0dxdy\\
        &\leq \frac{1}{2}\|\nabla\partial_x\partial^2_tn^0\|^2_{L^2_{xy}}+\frac{1}{2}\|\partial_x\partial^2_t n^0\|^2_{L^2_{xy}}+C(\|\nabla\partial_x\partial_tn^0\|^2_{L^2_{xy}}+1).\notag\notag
	\end{align}
    Next, we employ (\ref{eqt3n}) to bound $\|\nabla\partial_x\partial^2_tn^0\|^2_{L^2_{xy}}$, which is just as the following
    \begin{align}\label{xtt n}
        &\frac{1}{2}\frac{d}{dt}\|\partial_x\partial^2_t n^0\|^2_{L^2_{xy}}+\|\n\partial_x\partial^2_t n^0\|^2_{L^2_{xy}}\notag\\
        &=-\int^\infty_0\int^\infty_{-\infty}\partial_x\partial^2_t(n^0v^0)\cdot\nabla\partial_x\partial^2_t n^0dxdy-\int^\infty_0\int^\infty_{-\infty}\partial_x\partial^2_t(u^0\cdot\nabla n^0)\partial_x\partial^2_t n^0dxdy\notag\\
        &\leq \left(\|\partial_x\partial^2_t n^0\|_{L^2_{xy}}\| v^0\|_{L^\infty_{xy}}+\|\partial^2_t n^0\|_{L^2_{xy}}\|\partial_x v^0\|_{L^\infty_{xy}}\right)\|\nabla\partial_x\partial^2_t n^0\|_{L^2_{xy}}\notag\\
        &\quad+2\left(\|\partial_x\partial_t n^0\|_{L^4_{xy}}\|\partial_t v^0\|_{L^4_{xy}}+\|\partial_t n^0\|_{L^4_{xy}}\|\partial_x\partial_tv^0\|_{L^4_{xy}}\right)\|\nabla\partial_x\partial^2_t n^0\|_{L^2_{xy}}\notag\\
        &\quad+\left(\|\partial_x n^0\|_{L^\infty_{xy}}\|\partial^2_t v^0\|_{L^2_{xy}}+\|n^0\|_{L^\infty_{xy}}\|\partial_x\partial^2_t v^0\|_{L^2_{xy}}\right)\|\nabla\partial_x\partial^2_t n^0\|_{L^2_{xy}}\\
        &\quad+\left(\|\partial_x\partial^2_t u^0\|_{L^2_{xy}}\|\nabla n^0\|_{L^\infty_{xy}}+\|\partial^2_t u^0\|_{L^4_{xy}}\|\nabla\partial_x n^0\|_{L^4_{xy}}\right)\|\partial_x\partial^2_t n^0\|_{L^2_{xy}}\notag\\
        &\quad+2\|\partial_x\partial_t u^0\|_{L^4_{xy}}\|\nabla\partial_t n^0\|_{L^2_{xy}}\|\partial_x\partial^2_t n^0\|_{L^4_{xy}}\notag\\
        &\quad+\left(2\|\partial_t u^0\|_{L^\infty_{xy}}\|\nabla \partial_x\partial_tn^0\|_{L^2_{xy}}+\|\partial_x u^0\|_{L^\infty_{xy}}\|\nabla\partial^2_t n^0\|_{L^2_{xy}}\right)\|\partial_x\partial^2_t n^0\|_{L^2_{xy}}\notag\\
        &\leq \frac{1}{2}\|\nabla\partial_x\partial^2_tn^0\|^2_{L^2_{xy}}+C\left(\|\nabla^3 n^0\|^2_{L^2_{xy}}+\|\nabla^3 v^0\|^2_{L^2_{xy}}+\|\nabla\partial^2_xn^0\|_{L^2_{xy}}\right)\notag\\
        &\quad+C\left(\|\nabla\partial^2_tn^0\|^2_{L^2_{xy}}+\|\partial_x\partial^2_t n^0\|^2_{L^2_{xy}}+\|\nabla\partial_x\partial_tn^0\|^2_{L^2_{xy}}+1\right),\notag\notag
	\end{align}
    due to
    \begin{align*}
        \|\partial_x\partial^2_t v^0\|_{L^2_{xy}}&\leq\|\partial_tu^0\|_{H^2_{xy}}\| v^0\|_{H^2_{xy}}+\|u^0\|_{H^3_{xy}}\|\partial_tv^0\|_{H^1_{xy}}+\|u^0\|_{H^2_{xy}}\|\nabla\partial_x\partial_tv^0\|_{L^2_{xy}}\\
        &\leq C(\|\nabla^2\partial_tv^0\|_{L^2_{xy}}+1)\\
        &\leq C(\| \nabla^3v^0\|_{L^2_{xy}}+\|\nabla\partial^2_xn^0\|_{L^2_{xy}}+1),
	\end{align*}
    and (\ref{nabla2 vt}).\\
    Next, we estimate $\|\n\partial^3_x n^0\|^2_{L^2_TL^2_{xy}}$.
    Applying $\partial^3_x$ to $(\ref{eq:n0v0u0})_{1}$ and multiplying by $\partial^3_xn^0$, we then deduce
    \begin{align}\label{x3n}
        &\frac{1}{2}\frac{d}{dt}\|\partial^3_x n^0\|^2_{L^2_{xy}}+\|\n\partial^3_x n^0\|^2_{L^2_{xy}}\notag\\
        &=-\int^\infty_0\int^\infty_{-\infty}\partial^3_x(n^0v^0)\nabla\partial^3_x n^0dxdy-\int^\infty_0\int^\infty_{-\infty}\partial^3_x(u^0\cdot\nabla n^0)\partial^3_x n^0dxdy\\
        &\leq \frac{1}{2}\|\nabla\partial^3_xn^0\|^2_{L^2_{xy}}+C\|\partial^3_x n^0\|^2_{L^2_{xy}}+C\|\nabla^3v^0\|^2_{L^2_{xy}}+C\|\nabla\partial^2_xn^0\|^2_{L^2_{xy}}+C\|\nabla^4u^0\|^2_{L^2_{xy}}.\notag\notag
	\end{align}
  Combining $(\ref{dx2n})$-$(\ref{n0t x2})$, $(\ref{nabla3 v0})$-$(\ref{nabla3 omega})$, $(\ref{nabla xt n})$-$(\ref{xtt n})$ and substituting $(\ref{d4n})$, 
    by Gronwall's inequality and Lemma \ref{lem nabla omega}, we arrive at
    \begin{equation}\label{nabla3 n0 v0 omega}
		\begin{split}{}
        &\|(\n\partial^2_x n^0,\partial^2_x\partial_t n^0,\nabla^3 v^0,\nabla^4 u^0,\n\partial_x\partial_t n^0,\partial_x\partial^2_t n^0)\|^2_{L^2_{xy}}+\int^t_0\|\nabla\partial^3_xn^0\|^2_{L^2_{xy}}d\tau\\
        &\leq C(\|n_{in}\|_{H^3_{xy}},\|v_{in}\|_{H^3_{xy}},\|u_{in}\|_{H^4_{xy}},\|\partial^2_x\partial_tn_{in}\|_{L^2_{xy}},\|\nabla\partial_x\partial_tn_{in}\|_{L^2_{xy}},\|\partial_x\partial^2_tn_{in}\|_{L^2_{xy}}).
	\end{split}
	\end{equation}
    With these two steps, we verify that estimate (\ref{nabla3 tt n0 v0 omega}) holds.\\
    \textbf{Step 3. The estimates of  $\|\nabla^3 n^0\|_{L^\infty_TL^2_{xy}}$, $\|\nabla^2\partial_t v^0\|_{L^\infty_TL^2_{xy}}$ and $\|\nabla^2\partial_t \omega^0\|_{L^\infty_TL^2_{xy}}$.} Using the result in $(\ref{nabla3 tt n0 v0 omega})$, $(\ref{Equ(n3)})$ and the results of Lemma \ref{n0 v0 u0 L2}, we derive the estimate for $\|\nabla^3 n^0\|_{L^\infty_TL^2_{xy}}$ as follows,
    \begin{align}\label{d3n}
        &\|\nabla^3 n^0\|_{L^\infty_TL^2_{xy}}\leq C(\|\nabla\partial_t n^0\|_{L^\infty_TL^2_{xy}}+\|\nabla\partial^2_xn^0\|_{L^\infty_TL^2_{xy}}+1)\leq C.
	\end{align}
    By a similar argument (\ref{nabla v0t nabla omega0t}), we deduce that
    \begin{equation}\label{nabla2 v0t nabla2 omega0t}
		\begin{split}{}
        \|\nabla^2\partial_t v^0\|_{L^\infty_TL^2_{xy}}&\leq C\|u^0\|_{L^\infty_TH^3_{xy}}\| v^0\|_{L^\infty_TH^3_{xy}}+\|\nabla^3n^0\|_{L^\infty_TL^2_{xy}}\leq C,
        \\
        \|\nabla^2\partial_t \omega^0\|_{L^\infty_TL^2_{xy}}&\leq C\|u^0\|_{L^\infty_TH^3_{xy}}\| \nabla\omega^0\|_{L^\infty_TH^1_{xy}}+C\|u^0\|_{L^\infty_TH^2_{xy}}\| \nabla^3\omega^0\|_{L^\infty_TL^2_{xy}}+\|\nabla^3n^0\|_{L^\infty_TL^2_{xy}}\\
        &\leq C.
    \end{split}
	\end{equation}
    Further, using Lemma \ref{n0 v0 u0 L2}, (\ref{nabla3 n0 v0 omega}), (\ref{nabla2 v0t nabla2 omega0t}) and $\partial^2_yn^0=\partial_tn^{0}+u^{0}\cdot\nabla n^{0}-\nabla\cdot(n^{0}v^{0})-\partial^2_xn^0$, we derive the estimate for $\|\nabla^2\partial_t n^0\|_{L^\infty_TL^2_{xy}}$ as follows:
    \begin{align}
        \|\nabla^2\partial_t n^0\|_{L^\infty_TL^2_{xy}}&\leq \|\partial^2_x\partial_tn^0\|_{L^\infty_TL^2_{xy}}+\|\partial_x\partial_y\partial_tn^0\|_{L^\infty_TL^2_{xy}}+\|\partial^2_y\partial_tn^0\|_{L^\infty_TL^2_{xy}}\notag\\
        &\leq C+\|\partial_t(\partial_tn^{0}+u^{0}\cdot\nabla n^{0}-\nabla\cdot(n^{0}v^{0})-\partial^2_xn^0)\|_{L^\infty_TL^2_{xy}}\notag\\
        &\leq C+\|\partial^2_tn^{0}\|_{L^\infty_TL^2_{xy}}+\|\partial_tu^{0}\|_{L^\infty_TH^1_{xy}}\|n^{0}\|_{L^\infty_TH^2_{xy}}+\|u^{0}\|_{L^\infty_TH^2_{xy}}\|\nabla\partial_tn^{0}\|_{L^\infty_TL^2_{xy}}\\
        &\quad+\|\partial_tn^{0}\|_{L^\infty_TH^1_{xy}}\|v^{0}\|_{L^\infty_TH^2_{xy}}+\|n^{0}\|_{L^\infty_TH^2_{xy}}\|\partial_tv^{0}\|_{L^\infty_TH^1_{xy}}+\|\partial^2_x\partial_tn^0\|_{L^\infty_TL^2_{xy}}\notag\\
        &\leq C.\notag\notag
    \end{align}

    Moreover, integrating both sides of $(\ref{d4n})$ over the time, we have
    \begin{align*}
        \int^t_0\|\nabla^4 n^0\|^2_{L^2_{xy}}d\tau\leq C(\|n_{in}\|_{H^3_{xy}},\|v_{in}\|_{H^3_{xy}},\|u_{in}\|_{H^4_{xy}},\|\partial^2_x\partial_tn_{in}\|_{L^2_{xy}},\|\partial_x\partial_tn_{in}\|_{L^2_{xy}}).
    \end{align*}
    The proof is complete.

       Since the argument of Lemma $\ref{nabla3 n0}$ can be extended by induction to arbitrary higher-order derivatives, we arrive at the following regularity results.

        \begin{corollary}\label{cor5.1}
        Suppose the initial data $(n_{in},v_{in},u_{in})$ satisfies the hypotheses of Theorem \ref{th2.1} and $(n^0,v^0,u^0)$ satisfies the properties from Lemma \ref{n0 v0 u0 L2} and Lemma \ref{nabla3 n0}. This yields the following estimates:
        \begin{align}\label{nvu Hm}
           n^0\in {L^\infty(0,T;H^{m}_{xy})}\cap {L^2(0,T;H^{m+1}_{xy})},\;v^0\in {L^\infty(0,T;H^{m}_{xy})},\;u^0\in {L^\infty(0,T;H^{m+1}_{xy})},\; m\geq4,
        \end{align}
        and
        \begin{align}\label{nvu tj}
           \partial^j_tn^0,\partial^j_tv^0\in L^\infty(0,T;H^{m-2j}_{xy}),\;0<2j\leq m,\\
           \partial^j_tu^0\in L^\infty(0,T;H^{m+1-2j}_{xy}),\;0<2j\leq m+1,\notag\notag
        \end{align}
        where $C>0$ is a constant depending only on $m$ and the initial data $(n_{in},v_{in},u_{in})$, and $m, j$ be integer.
        \end{corollary}
        \noindent{\bf{Proof.}} For any integer $m\geq4$, we may suppose the inductive statement $\mathcal{P}_m$ be defined as:
        \begin{align}\label{Pm}
        \begin{cases}
         \displaystyle\sup\limits_{0\leq t\leq T}\left(\|n^0\|^2_{{H}^{m}}+\|v^0\|^2_{H^m_{xy}}+\|u^0\|^2_{H^{m+1}_{xy}}\right)+\int^T_0\|n^0\|^2_{H^{m+1}_{xy}}dt\leq C_m,\\
          \displaystyle\sup\limits_{0\leq t\leq T}\sum\limits_{0<2j\leq m}\left(\|\partial^j_tn^0\|^2_{H^{m-2j}_{xy}}+\|\partial^j_tv^0\|^2_{H^{m-2j}_{xy}}\right)\leq C_m,\\
           \displaystyle\sup\limits_{0\leq t\leq T}\sum\limits_{0<2j\leq m+1}\|\partial^j_tu^0\|^2_{H^{m+1-2j}_{xy}}\leq C_m.
        \end{cases}
        \end{align}
        By Lemma \ref{n0 v0 u0 L2} and Lemma \ref{nabla3 n0}, the case $m\leq3$ satisfies that
        \begin{align*}
            \sup_{0\leq t\leq T}\left(\|n^0\|^2_{{H}^{m}_{xy}}+\|v^0\|^2_{H^m_{xy}}+\|u^0\|^2_{H^{m+1}_{xy}}\right)+\int^T_0\|n^0\|^2_{H^{m+1}_{xy}}dt\leq C.
        \end{align*}
        We proceed to prove
        \begin{align*}
            \mathcal{P}_m\Longrightarrow\mathcal{P}_{m+1}.
        \end{align*}
        For this purpose, we define the energy functional
        \begin{align}
            \mathcal{E}(t):=\|\nabla\partial^m_x n^0\|^2_{L^2_{xy}}+\|\partial^m_x\partial_tn^0\|^2_{L^2_{xy}}+\|v^0\|^2_{H^{m+1}_{xy}}+\|\omega^0\|^2_{H^{m+1}_{xy}},
        \end{align}
        and we divide the argument into several steps as below.\\
        \textbf{Step 1. Estimate $n^0$.}
         Applying the operator $\partial^\alpha_x$ to the equation for $n^0$ in $(\ref{eq:n0v0u0})_1$, we rewrite the resulting equation in two equivalent forms:
        \begin{align}\label{eq n0 alpha}
          \partial^\alpha_x\partial_{t}n^0+\partial^\alpha_x(u^0\cdot\nabla n^0)=\nabla\cdot(\partial^\alpha_x\nabla n^0+\partial^\alpha_x(n^0v^0)),
        \end{align}
        \begin{align}\label{eq n0 t alpha}
           \partial^\alpha_x\partial^2_{t}n^0+\partial^\alpha_x\partial_t(u^0\cdot\nabla n^0)=\nabla\cdot(\partial^\alpha_x\partial_t\nabla n^0+\partial^\alpha_x\partial_t(n^0v^0)).
        \end{align}
        We take $\alpha=m$ for the equation (\ref{eq n0 alpha}) and multiply by  $\partial^m_x\partial_tn^0$, then from (\ref{Pm}), which yields
        \begin{align}\label{nabla n^0 Hmx}
            &\frac{1}{2}\frac{d}{dt}\|\nabla\partial^m_x n^0\|^2_{L^2_{xy}}+\|\partial^m_x\partial_tn^0\|^2_{L^2_{xy}}\notag\\
            &=-\left(\int^\infty_0\int^\infty_{-\infty}\partial^m_x(u^0\cdot\nabla n^0)\partial^m_x\partial_tn^0dxdy+\int^\infty_0\int^\infty_{-\infty}\partial^m_x(n^0 v^0)\cdot\nabla\partial^m_x\partial_tn^0dxdy\right)\\
            &\leq C\left(\|u^0\|_{H^{m+1}_{xy}}\|n^0\|_{{H}^{m}_{xy}}+\|u^0\|_{H^{m+1}_{xy}}\|\nabla\partial^m_x n^0\|_{L^2_{xy}}\right)\|\partial^m_x\partial_tn^0\|_{L^2_{xy}}+\|n^0\|_{H^{m}_{xy}}\|v^0\|_{{H}^{m}_{xy}}\|\nabla\partial^m_x\partial_tn^0\|_{L^2_{xy}}\notag\\
            &\leq \frac{1}{2}\|\partial^m_x\partial_tn^0\|^2_{L^2_{xy}}+\frac{1}{4}\|\nabla\partial^m_x\partial_tn^0\|^2_{L^2_{xy}}+C_m\left(\|\nabla\partial^m_x n^0\|^2_{L^2_{xy}}+1\right),\notag\notag
        \end{align}
        where we have used the boundary condition $(\partial_yn^{0}+n^{0}v^{0}_2)|_{y=0}=0$ such that
        $$\int^\infty_{-\infty}\left(\partial^m_x\partial_y n^0(t,x,0)+\partial^m_x(n^0(t,x,0)v^0_2(t,x,0))\right)\partial^m_x\partial_tn^0(t,x,0)dx=0.$$
        Furthermore, multiplying the equation (\ref{eq n0 t alpha}) by $\partial^m_x\partial_tn^0$  when $\alpha=m$ , substituting the equation of $\partial_tv^0=-\nabla(u^0\cdot v^0)+\nabla n^0$, using the Sobolev embedding inequality and $\nabla\cdot u^0=0$, we deduce that
        \begin{align}\label{n^0 t Hmx}
            &\frac{1}{2}\frac{d}{dt}\|\partial^m_x\partial_tn^0\|^2_{L^2_{xy}}+\|\nabla\partial^m_x\partial_tn^0\|^2_{L^2_{xy}}\notag\\
            &=-\left(\int^\infty_0\int^\infty_{-\infty}\left(\partial^m_x\left(\partial_tu^0\cdot\nabla n^0)+\partial^m_x(u^0\cdot\nabla \partial_tn^0\right)\right)\partial^m_x\partial_tn^0dxdy\right)\notag\\
            &\quad-\left(\int^\infty_0\int^\infty_{-\infty}\left(\partial^m_x\left(\partial_tn^0 v^0)+\partial^m_x(n^0 (-\nabla(u^0\cdot v^0)+\nabla n^0\right)\right)\cdot\nabla\partial^m_x\partial_tn^0dxdy\right)\\
            &\leq C\sum_{\alpha\leq m}\left(\|\partial_tu^0\|_{H^{m}_{xy}}\|n^0\|_{{H}^{m+1}_{xy}}+\|\partial_tu^0\|_{H^{m}_{xy}}\|\nabla\partial^m_x n^0\|_{L^2_{xy}}+\|u^0\|_{H^{m+1}_{xy}}\|\nabla\partial^{\alpha-1}_x\partial_tn^0\|_{L^2_{xy}}\right)\|\partial^m_x\partial_tn^0\|_{L^2_{xy}}\notag\\
            &\quad+C\sum_{\alpha\leq m}\left(\|\partial^m_x\partial_tn^0\|_{L^2_{xy}}\|v^0\|_{H^{2}_{xy}}+\|\nabla\partial^{\alpha-1}_x\partial_tn^0\|_{L^2_{xy}}\|v^0\|_{H^{m+1}_{xy}}+\|\partial_tn^0\|_{H^2_{xy}}\|v^0\|_{H^{m}_{xy}}\right)\|\nabla\partial^{m}_x\partial_tn^0\|_{L^2_{xy}}\notag\\
            &\quad+C\left(\|n^0\|_{{H}^{m}_{xy}}\|u^0\|_{H^{m+1}_{xy}}\|v^0\|_{H^{m+1}_{xy}}+\|n^0\|_{{H}^{m}_{xy}}\|n^0\|_{H^{m+1}_{xy}}\right)\|\nabla\partial^{m}_x\partial_tn^0\|_{L^2_{xy}}\notag\\
            &\leq \frac{1}{4}\|\nabla\partial^{m}_x\partial_tn^0\|^2_{L^2_{xy}}+C_m\left(\|\nabla\partial^m_x n^0\|^2_{L^2_{xy}}+\|\partial^m_x\partial_tn^0\|^2_{L^2_{xy}}+\|n^0\|^2_{{H}^{m+1}_{xy}}+\|v^0\|^2_{H^{m+1}_{xy}}+1\right),\notag\notag
        \end{align}
        for the norm $\|\partial_tu^0\|_{H^{m}_{xy}}$, we obtain the following bound via identity (\ref{eq:omega}):
        \begin{align*}           \|\partial_tu^0\|_{H^{m}_{xy}}&=\|\partial_tu^0\|_{L^2_{xy}}+\|\partial_t\omega^0\|_{H^{m-1}_{xy}}\\
            &\leq \|\partial_tu^0\|_{L^2_{xy}}+\|u^0\|_{H^{m+1}_{xy}}\|\omega^0\|_{H^{m}_{xy}}+\|n^0\|_{H^{m}_{xy}}\\
            &\leq C_m.
        \end{align*}
        Combining (\ref{nabla n^0 Hmx}) and (\ref{n^0 t Hmx}), we deduce that
        \begin{align}\label{n^0 Hmx and Hmxt}
            &\frac{d}{dt}\left(\|\nabla\partial^m_x n^0\|^2_{L^2_{xy}}+\|\partial^m_x\partial_tn^0\|^2_{L^2_{xy}}\right)+\|\nabla\partial^{m}_x\partial_tn^0\|^2_{L^2_{xy}}\leq C_m\left(1+\|n^0\|^2_{{H}^{m+1}_{xy}}+\mathcal{E}(t)\right).
        \end{align}
        To derive the full norm of $\|n^0\|_{{H}^{m+2}_{xy}}$, we simplify the governing equation into an elliptic equation in the $y$-direction for analysis:
        \begin{align}\label{eq:y2n0}
            \partial^2_yn^0=\partial_tn^{0}+u^{0}\cdot\nabla n^{0}-\nabla\cdot(n^{0}v^{0})-\partial^2_xn^0.
        \end{align}
        To estimate $\partial^l_yn^0$, it suffices to bound the terms on the right-hand side. We observe the following facts:
        \begin{align}\label{n0 Hm+2}
            \|n^0\|_{H^{m+2}_{xy}}&\leq\sum_{a+l\leq m+2}\|\partial^a_x\partial^l_yn^0\|_{L^2_{xy}}.
        \end{align}
        Applying operator $\partial^a_x\partial^{l-2}_y$ to the equation $(\ref{eq:y2n0})$, for all integers $l\geq2$ and $a+l\leq m+2$, we have
        \begin{align}\label{n0 Hm+1}
            \partial^a_x\partial^l_yn^0=\partial^a_x\partial^{l-2}_y\partial_tn^{0}+\partial^a_x\partial^{l-2}_y(u^{0}\cdot\nabla n^{0})-\partial^a_x\partial^{l-2}_y\nabla\cdot(n^{0}v^{0})-\partial^{a+2}_x\partial^{l-2}_yn^0,
        \end{align}
        we notice that $a+l\leq m+2\Rightarrow a+l-2\leq m$ and $a+2+l-2=a+l\leq m+2$, then we obtain
        {\begin{align}\label{yn0 Hmx}
            \|\partial^a_x\partial^l_yn^0\|_{L^2_{xy}}&\leq \|\partial^a_x\partial^{l-2}_y\partial_tn^{0}\|_{L^2_{xy}}+\|\partial^a_x\partial^{l-2}_y(u^{0}\cdot\nabla n^{0})\|_{L^2_{xy}}\notag\\
            &\quad+\|\partial^a_x\partial^{l-2}_y\nabla\cdot(n^{0}v^{0})\|_{L^2_{xy}}+\|\partial^{a+2}_x\partial^{l-2}_yn^0\|_{L^2_{xy}}\\
            &\leq \|\partial_tn^{0}\|_{{H}^{m}_{xy}}+\|u^{0}\cdot\nabla n^{0}\|_{H^m_{xy}}+\|n^{0}v^{0}\|_{H^{m+1}_{xy}}+C\left(\|n^0\|_{H^{m+1}}+\|\nabla\partial_x^{m+1}n^0\|_{L^2_{xy}}\right)\notag\\
            &\leq C\left(\|\partial_tn^{0}\|_{{H}^{m}_{xy}}+\|\nabla\partial^{m+1}_x n^0\|_{L^2_{xy}}\right)+C\left(\|u^{0}\|_{{H}^{m}_{xy}}+\|v^0\|_{H^{m+1}_{xy}}+1\right)\|n^{0}\|_{{H}^{m+1}_{xy}}.\notag\notag
        \end{align}}

        \textbf{Step 2. Estimate $v^0$ and $u^0$.} Applying $\nabla^{m+1}$ to $(\ref{eq:n0v0u0})_2$ and multiplying $\nabla^{m+1} v^0$, we deduce from the curl condition $\nabla\times v^0=0$ that
        \begin{align}\label{v0 Hm+1}
            &\frac{1}{2}\frac{d}{dt}\|\nabla^{m+1}v^0\|^2_{L^2_{xy}}\notag\\
            &=-\int^\infty_0\int^\infty_{-\infty}\left(\nabla^{m+2} u^0\cdot v^0+\sum_{1\leq|\alpha|\leq m}\nabla^{m+2-\alpha} u^0\cdot\nabla^{\alpha}v^0+\nabla u^0\cdot\nabla^{m+1} v^0\right)\cdot\nabla^{m+1} v^0 dxdy\notag\\
            &\quad-\int^\infty_0\int^\infty_{-\infty}\left(\nabla^{m+1} u^0\cdot (\nabla v^0)+\sum_{1\leq|\alpha|\leq m}\nabla^{m+1-\alpha} u^0\cdot(\nabla^{\alpha+1}v^0)\right)\cdot\nabla^{m+1} v^0dxdy\\
            &\quad-\int^\infty_0\int^\infty_{-\infty}\nabla^{m+1} v^0 \cdot(\nabla\nabla^{m+1} v^0)\cdot u^0 dxdy+\int^\infty_0\int^\infty_{-\infty}\nabla^{m+1}\nabla n^0\cdot\nabla^{m+1} v^0dxdy\notag\\
            &\leq C\left(\|u^0\|_{H^{m+2}_{xy}}\|v^0\|_{H^{2}_{xy}}+\|u^0\|_{H^{m+2}_{xy}}\|v^0\|_{H^{m+1}_{xy}}+\|u^0\|_{H^{3}_{xy}}\|v^0\|_{H^{m+1}_{xy}}+\|n^0\|_{H^{m+2}_{xy}}\right)\|v^0\|_{H^{m+1}_{xy}}\notag\\
            &\leq \epsilon\|n^0\|^2_{H^{m+2}_{xy}}+C_\epsilon(1+\|u^0\|^2_{H^{m+2}_{xy}})\|v^0\|^2_{H^{m+1}_{xy}}\notag\\
            &\leq \epsilon\|n^0\|^2_{H^{m+2}_{xy}}+C_\epsilon(1+\mathcal E(t))^2,\notag\notag
        \end{align}
        where we use curl $\nabla\times v^0=0$ and $\nabla\cdot u^0=0$ to obtain
        $$\int^\infty_0\int^\infty_{-\infty}\nabla^{m+1} v^0 \cdot(\nabla\nabla^{m+1} v^0)\cdot u^0 dxdy=\int^\infty_0\int^\infty_{-\infty}u^0\cdot\nabla|\nabla^{m+1} v^0|^2 dxdy=0.$$
         Similarly, applying $\nabla^{m+1}$ to (\ref{eq:omega}) and multiplying $\nabla^{m+1} \omega^0$, we arrive at
        \begin{align}\label{omega Hm+1}
            &\frac{1}{2}\frac{d}{dt}\|\nabla^{m+1}\omega^0\|^2_{L^2_{xy}}\notag\\
            &=-\int^\infty_0\int^\infty_{-\infty}\left(\nabla^{m+1} u^0\cdot\nabla\omega^0+\sum_{1\leq|\alpha|\leq m}\nabla^{m+1-\alpha} u^0\cdot\nabla^{\alpha+1}\omega^0+u^0\cdot\nabla\nabla^{m+1}\omega^0\right)\cdot\nabla^{m+1}\omega^0 dxdy\notag\\
            &\quad+\int^\infty_0\int^\infty_{-\infty}\nabla^{m+1}\partial_x n^0\cdot\nabla^{m+1}\omega^0dxdy\\
            &\leq C\|u^0\|_{H^{m+2}_{xy}}\|\omega^0\|^2_{H^{m+1}_{xy}}+\epsilon\|n^0\|^2_{H^{m+2}_{xy}}+C\|\omega^0\|^2_{H^{m+1}_{xy}},\notag\\
            &\leq \epsilon\|n^0\|^2_{H^{m+2}_{xy}}+C_\epsilon(1+\mathcal E(t))^2.\notag\notag
        \end{align}
         \textbf{Step 3. Combining the above estimates yields energy closure for $\mathcal{E}(t)$}.
        Summing up $(\ref{n^0 Hmx and Hmxt})$, $(\ref{n0 Hm+2})$, $(\ref{yn0 Hmx})$, $(\ref{v0 Hm+1})$ and $(\ref{omega Hm+1})$ yields
        \begin{align}\label{energy closure}
            &\frac{d}{dt}\mathcal{E}(t)+C(1-4\epsilon)\|n^0\|^2_{{H}^{m+2}_{xy}}\leq \left(\mathcal{E}(t)+1\right)^2+C\|n^{0}\|^2_{{H}^{m+1}_{xy}}+C\left(\|\partial_tn^{0}\|_{{H}^{m}_{xy}}+\|\nabla\partial^{m+1}_x n^0\|_{L^2_{xy}}\right).
        \end{align}
        We now obtain the corresponding bounds for $\|\partial_tn^{0}\|_{{H}^{m}_{xy}}$ and $\|\nabla\partial^{m+1}_x n^0\|_{L^2_{xy}}$.
        It remains to recover the normal derivatives of \(\partial_tn^{0}\). From the equation,
        \begin{align*}
          \partial_y^2 \partial_tn^{0}
        =
        \partial_t^2n^0
        +
        \partial_t(u^0\cdot\nabla n^0)
        -
        \nabla\cdot\partial_t(n^{0}v^0)-\partial_x^2\partial_tn^{0},
        \end{align*}
        for \(l\ge2\), \(a+l\le m\) and $a+l-2\leq m-2$, we have
        \begin{align}\label{n0t Hm}
           \|\partial_x^a\partial_y^l \partial_tn^{0}\|_{L^2_{xy}}
        &=
        \|\partial_x^a\partial_y^{l-2}\partial_t^2n^0\|_{L^2_{xy}}
        +
        \|\partial_x^a\partial_y^{l-2}(\partial_tu^0\cdot\nabla n^0)\|_{L^2_{xy}}
        +
        \|\partial_x^a\partial_y^{l-2}(u^0\cdot\nabla \partial_tn^{0})\|_{L^2_{xy}}\notag\\
        &\quad+
        \|\nabla\cdot\partial_x^a\partial_y^{l-2}(\partial_tn^{0}v^0+n^0\partial_t v^0)\|_{L^2_{xy}}
        +
        \|\partial_x^{a+2}\partial_y^{l-2}\partial_tn^{0}\|_{L^2_{xy}}\notag\\
        &\leq \|\partial^2_tn^{0}\|_{{H}^{m-2}_{xy}}+\|\partial_tu^{0}\|_{{H}^{m-2}_{xy}}\|n^{0}\|_{{H}^{m-1}_{xy}}+\|u^{0}\|_{{H}^{m-2}_{xy}}\|\partial_tn^{0}\|_{{H}^{m-1}_{xy}}\\
        &\quad+\|\partial_tn^{0}\|_{{H}^{m-1}_{xy}}\|v^{0}\|_{{H}^{m-1}_{xy}}+\|n^{0}\|_{{H}^{m-1}_{xy}}\|\partial_tv^{0}\|_{{H}^{m-1}_{xy}}+\|\partial_tn^{0}\|_{{H}^{m-1}_{xy}}+\|\nabla\partial^{m-1}_x\partial_t n^0\|_{L^2_{xy}}\notag\\
        &\leq C_m(1+\mathcal E(t))^2+C\|n^{0}\|^2_{{H}^{m+1}_{xy}}.\notag\notag
        \end{align}
        Applying $\partial_x^{m+1}$ to the $n^0$ equation and taking the $L^2_{xy}$
        inner product with $\partial_x^{m+1}n^0$
        \begin{align*}
        &\frac{1}{2}\frac{d}{dt}\|\partial_x^{m+1}n^0\|_{L^2_{xy}}^2
        +\|\nabla\partial_x^{m+1}n^0\|_{L^2_{xy}}^2 \notag\\
        &
        = -\int^\infty_0\int^\infty_{-\infty}
        \partial_x^{m+1}(u^0\cdot\nabla n^0)
        \partial_x^{m+1}n^0dxdy
        -\int^\infty_0\int^\infty_{-\infty}
        \partial_x^{m+1}(n^0v^0)\cdot
        \nabla\partial_x^{m+1}n^0dxdy\\
        &\leq C\|u^0\|_{H^{m+1}_{xy}}\|n^0\|_{H^{m+1}_{xy}}\|\partial_x^{m+1}n^0\|_{L^2_{xy}}+\|n^0\|_{H^{m+1}_{xy}}\|v^0\|_{H^{m+1}_{xy}}\|\nabla\partial_x^{m+1}n^0\|_{L^2_{xy}}\\
        &\leq \frac{1}{2}\|\nabla\partial_x^{m+1}n^0\|_{L^2_{xy}}^2+C\|v^0\|^2_{H^{m+1}_{xy}}\|n^{0}\|^2_{{H}^{m+1}_{xy}}+C_m\|n^0\|_{H^{m+1}_{xy}}\|\partial_x^{m+1}n^0\|_{L^2_{xy}}
        \end{align*}
        which implies that
        \begin{align}\label{n0 Hm+1x}
        &\frac{d}{dt}\|\partial_x^{m+1}n^0\|_{L^2_{xy}}^2+\|\nabla\partial_x^{m+1}n^0\|_{L^2_{xy}}^2\leq C_m(1+\mathcal E(t))^2\|n^0\|_{H^{m+1}_{xy}}^2
        \end{align}
        Substituting (\ref{n0t Hm}) and (\ref{n0 Hm+1x}) into (\ref{energy closure}), by Gronwall's inequality, for sufficiently small $\epsilon$, we have
        \begin{align}\label{energy result}
            &\sup_{0\leq t\leq T}\mathcal{E}(t)+\int^T_0\|n^0\|^2_{{H}^{m+2}_{xy}}\leq C_{m+1}.
        \end{align}
        By arguments analogous to those for $(\ref{yn0 Hmx})$, we derive the following estimate for $\|n^0\|_{L^\infty_T{H}^{m+1}_{xy}}$:
        \begin{align*}
            \|n^0\|_{{H}^{m+1}_{xy}}\leq\|\partial_tn^{0}\|_{{H}^{m-1}_{xy}}+\|\nabla\partial_x^{m+1}n^0\|_{L^2_{xy}}+C_m\leq C_{m+1}.
        \end{align*}
        Collecting all previous estimates, we obtain the following result:
        \begin{align}\label{nvu space Hm+1}
            n^0\in {L^\infty_TH^{m+1}_{xy}}\cap {L^2_TH^{m+2}_{xy}},\notag\\
            v^0\in {L^\infty_TH^{m+1}_{xy}},\\
            u^0\in {L^\infty_TH^{m+2}_{xy}}.\notag\notag
        \end{align}
        \textbf{Step 4. Induction on time derivatives.} At this step, we arrive at the following conclusion:
        \begin{align}\label{nvu Hm+1tj}
            \partial^{j}_tn^0\in {L^\infty_TH^{m+1-2j}_{xy}},\;
            \partial^{j}_tv^0\in {L^\infty_TH^{m+1-2j}_{xy}},\;
            \partial^{j}_tu^0\in {L^\infty_TH^{m+2-2j}_{xy}}.
        \end{align}
        From equation $(\ref{eq:n0v0u0})_1$, for all $0\leq 2j\leq m+1$, we have
        \begin{align*}
            &\|\partial^{j+1}_tn^0\|_{H^{m+1-2(j+1)}_{xy}}\\
            &\leq\|\partial^{j}_t(u^0\cdot\nabla n^0)\|_{H^{m+1-2(j+1)}_{xy}}+\|\partial^{j}_t\nabla\cdot(n^0v^0)\|_{H^{m+1-2(j+1)}_{xy}}+\|\partial^{j}_t\Delta n^0\|_{H^{m+1-2(j+1)}_{xy}}\notag\\
            &\leq \sum_{0\leq s\leq j}\|\partial^{s}_tu^0\|_{H^{m+1-2s}_{xy}}\|\partial^{j-s}_tn^0\|_{H^{m-2(j-s)}_{xy}}+\sum_{0\leq s\leq j}\|\partial^{s}_tn^0\|_{H^{m-2s}_{xy}}\|\partial^{j-s}_tv^0\|_{H^{m-2(j-s)}_{xy}}+\|\partial^{j}_tn^0\|_{H^{m+1-2j}_{xy}}.\notag\notag
        \end{align*}
        Via mathematical induction, we arrive at
         \begin{align}\label{n Hmtj}
            \|\partial^{j+1}_tn^0\|_{H^{m-2j-1}_{xy}}\leq C_{m+1}.
        \end{align}
        Similarly, from equation $(\ref{eq:n0v0u0})_2$, the estimate of $\partial^j_tv$ is derived as follows:
        \begin{align}\label{v Hmtj}
            \|\partial^{j+1}_tv^0\|_{H^{m+1-2(j+1)}_{xy}}
            &\leq\|\partial^{j}_t\nabla(u^0\cdot v^0)\|_{H^{m+1-2(j+1)}_{xy}}+\|\partial^{j}_t\nabla n^0\|_{H^{m+1-2(j+1)}_{xy}}\notag\\
            &\leq \sum_{0\leq s\leq j}\|\partial^{s}_tu^0\|_{H^{m+1-2s}_{xy}}\|\partial^{j-s}_tv^0\|_{H^{m-2(j-s)}_{xy}}+\|\partial^{j}_tn^0\|_{H^{m-2j}_{xy}}\\
            &\leq C_{m+1}.\notag\notag
        \end{align}
        Moreover, the div?curl estimate yields that
        \begin{align*}
           \|\partial^{j+1}_tu^0\|_{H^{m+2-2(j+1)}_{xy}}= \|\partial^{j+1}_tu^0\|_{H^{m-2j}_{xy}}=\|\partial^{j+1}_tu^0\|_{L^2_{xy}}+\|\partial^{j+1}_t\omega^0\|_{H^{m-2j-1}_{xy}}.
        \end{align*}
       We estimate the two terms on the?right?hand?side separately as follows. Taking the time derivative of equation $(\ref{eq:n0v0u0})_4$, we assert that $\|\partial^{j+1}_tu^0\|_{L^2_{xy}}$ holds as the following
         \begin{align}\label{u0 Ht}
             &\frac{1}{2}\frac{d}{dt}\|\partial^{j+1}_tu^0\|^2_{L^2_{xy}}\notag\\
             &=\int^\infty_0\int^\infty_{-\infty}\left(-\partial^{j+1}_t(u^0\cdot\nabla u^0)+\partial^{j+1}_tn^0e_2\right)\cdot\partial^{j+1}_tu^0dxdy\\
            &\leq C\sum_{0\leq s\leq j+1}\|\partial^s_tu^0\|^2_{H^{m+2-2s}_{xy}}\|\partial^{j+1-s}_t\omega^0\|^2_{H^{m+1-2(j+1-s)}_{xy}}+\|\partial^{j+1}_tn^0\|^2_{L^2_{xy}}+\|\partial^{j+1}_tu^0\|^2_{L^2_{xy}}.\notag\notag
         \end{align}
         The estimate for the vorticity term is given as follows:
         \begin{align}\label{omega Hmtj}
            &\|\partial^{j+1}_t\omega^0\|_{H^{m+1-2(j+1)}_{xy}}\notag\\
            &\leq\|\partial^{j}_t(u^0\cdot\nabla \omega^0)\|_{H^{m+1-2(j+1)}_{xy}}+\|\partial^{j}_t\partial_x n^0\|_{H^{m+1-2(j+1)}_{xy}}\notag\\
            &\leq \sum_{0\leq s\leq j}\|\partial^{s}_tu^0\|_{H^{m+1-2s}_{xy}}\|\partial^{j-s}_t\omega^0\|_{H^{m-2(j-s)}_{xy}}+\|\partial^{j}_tn^0\|_{H^{m-2j}_{xy}}\\
            &\leq C_{m+1}.\notag\notag
        \end{align}
         Therefore, for all integers $j$ satisfying $0\leq2j\leq m+2$, combine (\ref{u0 Ht}) and (\ref{omega Hmtj}), we shows that
         \begin{align}
            \|\partial^{j+1}_tu^0\|^2_{H^{m+2-2(j+1)}_{xy}}\leq C_{m+1}.
         \end{align}
         Then the conclusion for (\ref{nvu Hm+1tj}) is verified.\\
         \textbf{Step 5. Inductive conclusion.} Combining Step 3 and Step 4, we deduce that if $\mathcal{P}_m$ holds, then
         \begin{align*}
            \sup_{0\leq t\leq T}\left(\|n^0\|^2_{{H}^{m+1}}+\|v^0\|^2_{H^{m+1}_{xy}}+\|u^0\|^2_{H^{m+2}_{xy}}\right)+\int^T_0\|n^0\|^2_{H^{m+2}_{xy}}dt&\leq C_{m+1},\\
           \sup_{0\leq t\leq T}\sum_{0<2j\leq m+1}\|\partial^j_tn^0\|^2_{H^{m+1-2j}_{xy}}+\|\partial^j_tv^0\|^2_{H^{m+1-2j}_{xy}}&\leq C_{m+1},\notag\\
           \sup_{0\leq t\leq T}\sum_{0<2j\leq m+2}\|\partial^j_tu^0\|^2_{H^{m+2-2j}_{xy}}&\leq C_{m+1}.
        \end{align*}
        Consequently, $\mathcal{P}_{m+1}$ holds true. As Lemma \ref{nabla3 n0} confirms the case $\mathcal{P}_{3}$, it follows from the induction method that (\ref{nvu Hm}) and (\ref{nvu tj}) are valid on $[0,T]$ for every integer $m\geq4$.\\

        Finally, we give the proof of Theorem \ref{th2.1}.\\

        \noindent \textbf{Proof of Theorem \ref{th2.1}.}
         Using the foregoing conclusions from Lemma $\ref{n0 v0 u0 L2}$, Lemma $\ref{nabla3 n0}$ and Corollary \ref{cor5.1}, together with the contraction mapping argument and difference energy estimates, we establish Theorem \ref{th2.1}.\\

       \section{inner layer estimation}\label{5}
       In this part, we mainly explain the well-posedness of the solution of $(\ref{eq:vB02})$-$(\ref{eq:pB2})$. For the regularity of $(\ref{eq:vB02})$ and $(\ref{eq:vB12})$, we introduce the following auxiliary functions related to $f^B(t,x,z)$
        \begin{align}\label{f}
        \begin{cases}        \partial_tf^B+\overline{\partial_yu^{0}_2}f^B+\overline{u^{0}_1}\partial_xf^B+z\overline{\partial_yu^{0}_2}\partial_zf^B+\overline{n^0}f^B=\partial^2_zf^B+\rho\\
        f^B|_{z=0}=0,\\
        f^B|_{t=0}=0.
        \end{cases}
	    \end{align}
        Then the regularity result of the solution to the equations (\ref{f}) below holds. $(\mathcal{H}^{k,m,l},\|\cdot\|_{k,m,l})$ be the anisotropic Sobolev spaces defined as  (\ref{H space}).
       \begin{proposition}\label{regularity of f^B}
		{ Suppose $(n^0,v^0,u^0)$ be the solutions that satisfy the conditions in Theorem \ref{th2.1} on $[0,T]$, such that for $m_0\geq4$
        \begin{align}\label{rho}
        \partial^j_t{u^{0}_1},\partial^j_t{u^{0}_2}\in L^\infty(0,T;H^{m_0-j+3}_{xy}),\;\partial^j_t{n^{0}}\in L^\infty(0,T;H^{m_0-j+2}_{xy}),
        \end{align}
        and
        \begin{align}\label{rhotj}
        \partial^j_t\rho&\in L^2(0,T;\mathcal{H}^{k+m_0-j,m_0-j,0}),\;\partial_t\rho\in L^\infty(0,T;\mathcal{H}^{k,m_0-3,0}),\;\rho\in L^\infty(0,T;\mathcal{H}^{k,m_0-l,l-2}),\;l=2,3,4
        \end{align}
            for any $k\in N$ and $j=0,1,2$. Then the solution $f^B$ to the equations $(\ref{f})$ satisfies the following properties:
			\begin{equation}\label{fBt}
				\begin{split}{}
                  \partial^j_tf^B&\in L^\infty(0,T;\mathcal{H}^{k,m_0-j,0})\cap L^2(0,T; \mathcal{H}^{k,m_0-j,1}),\;j=0,1,2,\\
                  \partial_tf^B&\in L^\infty(0,T;\mathcal{H}^{k,m_0-2,1}),\partial_tf^B\in L^\infty(0,T;\mathcal{H}^{k,m_0-3,2}),
				\end{split}
			\end{equation}
            and
            \begin{align}\label{fBz}
                f^B&\in L^\infty(0,T;\mathcal{H}^{k,m_0-i,i}),\;i=1,2,3,4.
            \end{align}
		}
	\end{proposition}
        \noindent{\bf{Proof.}} Before proceeding with the weighted estimates, we clarify how the factor \(z\) in the normal transport term is treated. A direct rough estimate would apparently increase the polynomial weight and lead to an unclosed estimate involving a higher weighted norm of the same unknown. To avoid this difficulty, we adopt appropriate weight by retaining the distribution of the tangential derivatives in the Leibniz expansion.
        More precisely, for \(j=0,1,2\), we set
        \begin{align}\label{weighted-index-f}
            \kappa_{\alpha,j}
            :=
            k+m_0-j-\alpha,
            \qquad
            0\leq\alpha\leq m_0-j,
        \end{align}
        and use the stronger weighted quantity
        \begin{align}\label{graded-weighted-energy-f}
            \mathcal{E}_{k,j}(f^B):=
            \sum_{\alpha=0}^{m_0-j}
            \int_0^{+\infty}\int^\infty_{-\infty}
            \left(1+z^{2\kappa_{\alpha,j}}\right)
            \left|
            \partial_x^\alpha\partial_t^j f^B
            \right|^2
            \,dx\,dz.
        \end{align}
        Since \(\kappa_{\alpha,j}\geq k\), an estimate of
        \(\mathcal{E}_{k,j}(f^B)\) immediately implies the corresponding
        \(\mathcal{H}^{k,m_0-j,0}\)-estimate stated in the Proposition. Note  that
        \begin{align}\label{weighted-index-relation-f}
            \kappa_{\alpha-\beta,j}
            =
            \kappa_{\alpha,j}+\beta,
            \qquad
            1\leq\beta\leq\alpha,
        \end{align}
        and hence
        \begin{align}\label{weighted-shift-f}
            z^2\left(1+z^{2\kappa_{\alpha,j}}\right)
            \leq
            C\left(1+z^{2\kappa_{\alpha-\beta,j}}\right).
        \end{align}
        Thus, the factor \(z\) is controlled by the larger polynomial weight
        assigned to the lower-order tangential derivative, since  this only produces
        finite classes of weights.\\
        \textbf{Step 1. For $0\leq\alpha\leq m_0$.}\\
        Applying $\partial_x^\alpha$ to the equations of $(\ref{f})_1$, multiplying $(\ref{f})_1$ by $\left(1+z^{2\kappa_{\alpha,0}}\right)\partial_x^\alpha f^B$, we obtain
        \begin{align*}
			&\frac{1}{2}\frac{d}{dt}\mathcal{E}_{k,0}(f^B)+\sum_{\alpha=0}^{m_0}\|\left(1+z^{2\kappa_{\alpha,0}}\right)^\frac{1}{2}\partial_x^\alpha\partial_zf^B\|^2_{L^2_{xz}}\\
			&=-\sum_{\alpha=0}^{m_0}\int^{+\infty}_{0}\int^{+\infty}_{-\infty}2\kappa_{\alpha,0}z^{2\kappa_{\alpha,0}-1}\partial_x^\alpha\partial_z f^B\partial_x^\alpha f^Bdxdz\\
            &\quad-\sum_{\alpha=0}^{m_0}\int^{+\infty}_{0}\int^{+\infty}_{-\infty}\left(1+z^{2\kappa_{\alpha,0}}\right)\partial_x^\alpha(\overline{\partial_yu^0_2}f^B)\partial_x^\alpha f^Bdxdz\\
            &\quad-\sum_{\alpha=0}^{m_0}\int^{+\infty}_{0}\int^{+\infty}_{-\infty}\left(1+z^{2\kappa_{\alpha,0}}\right)\partial_x^\alpha(\overline{u^0_1}\partial_xf^B)\partial_x^\alpha f^Bdxdz\\
            &\quad-\sum_{\alpha=0}^{m_0}\int^{+\infty}_{0}\int^{+\infty}_{-\infty}\left(1+z^{2\kappa_{\alpha,0}}\right)\partial_x^\alpha(z\overline{\partial_yu^0_2}\partial_zf^B)\partial_x^\alpha f^Bdxdz\\
            &\quad-\sum_{\alpha=0}^{m_0}\int^{+\infty}_{0}\int^{+\infty}_{-\infty}\left(1+z^{2\kappa_{\alpha,0}}\right)\partial_x^\alpha(\overline{n^0}f^B)\partial_x^\alpha f^Bdxdz+\sum_{\alpha=0}^{m_0}\int^{+\infty}_{0}\int^{+\infty}_{-\infty}\left(1+z^{2\kappa_{\alpha,0}}\right)\partial_x^\alpha\rho\partial_x^\alpha f^Bdxdz\\
            &=:J_1+...+J_5+\sum_{\alpha=0}^{m_0}\int^{+\infty}_{0}\int^{+\infty}_{-\infty}\left(1+z^{2\kappa_{\alpha,0}}\right)\partial_x^\alpha\rho\partial_x^\alpha f^Bdxdz.
            \end{align*}
            Using (\ref{rho}), we assemble all estimates of \(J_1\) through \(J_5\) in the following form:
            \begin{align*}
                J_1&\leq C\sum_{\alpha=0}^{m_0}\|\left(1+z^{2\kappa_{\alpha,0}}\right)^\frac{1}{2}\partial_x^\alpha\partial_z f^B\|_{L^2_{xz}}\|\left(1+z^{2\kappa_{\alpha,0}}\right)^\frac{1}{2}\partial_x^\alpha f^B\|_{L^2_{xz}}\\
                &\leq\delta_1\sum_{\alpha=0}^{m_0}\|\left(1+z^{2\kappa_{\alpha,0}}\right)^\frac{1}{2}\partial_x^\alpha\partial_z f^B\|_{L^2_{xz}}^2+C\mathcal{E}_{k,0}(f^B).\\
                J_2+J_5&\leq \sum_{\alpha=0}^{m_0}\sum_{0\leq\beta\leq \alpha}C_{\alpha}^{\beta}\|\overline{\partial_x^\beta\partial_yu^0_2}\|_{L^\infty_{x}}\|\left(1+z^{2\kappa_{\alpha,0}}\right)^\frac{1}{2}\partial_x^{\alpha-\beta} f^B\|_{L^2_{xz}}\|\left(1+z^{2\kappa_{\alpha,0}}\right)^\frac{1}{2}\partial_x^{\alpha} f^B\|_{L^2_{xz}}\\
                &\quad+\sum_{\alpha=0}^{m_0}\sum_{0\leq\beta\leq\alpha}C_{\alpha}^{\beta}\|\overline{\partial_x^\beta n^0}\|_{L^\infty_{x}}\|\left(1+z^{2\kappa_{\alpha,0}}\right)^\frac{1}{2}\partial_x^{\alpha-\beta} f^B\|_{L^2_{xz}}\|\left(1+z^{2\kappa_{\alpha,0}}\right)^\frac{1}{2}\partial_x^\alpha f^B\|_{L^2_{xz}}\\
                &\leq C\mathcal{E}_{k,0}(f^B).\\                J_3&\leq\sum_{\alpha=0}^{m_0}\|\overline{\partial_{x}u^0_1}\|_{L^\infty_{x}}\|\left(1+z^{2\kappa_{\alpha,0}}\right)^\frac{1}{2}\partial_x^{\alpha}f^B\|^2_{L^2_{xz}}\\            &\quad+\sum_{\alpha=0}^{m_0}\sum_{0<\beta\leq\alpha}C_{\alpha}^{\beta}\|\overline{\partial_x^{\beta}u^0_1}\|_{L^\infty_{x}}\|\left(1+z^{2\kappa_{\alpha,0}}\right)^\frac{1}{2}\partial_x^{\alpha-\beta+1}f^B\|_{L^2_{xz}}\|\left(1+z^{2\kappa_{\alpha,0}}\right)^\frac{1}{2}\partial_x^{\alpha}f^B\|_{L^2_{xz}}\\
                &\leq C\mathcal{E}_{k,0}(f^B).\\
                J_4&\leq\frac{1}{2}\sum_{\alpha=0}^{m_0}\int_0^{+\infty}\int^\infty_{-\infty}\overline{\partial_yu_2^0}\,
                \partial_z\left[z\left(1+z^{2\kappa_{\alpha,0}}\right)\right]\left|\partial_x^\alpha f^B\right|^2\,dx\,dz\\
                &\quad+\sum_{\alpha=0}^{m_0}\sum_{0<\beta\leq\alpha}C_{\alpha}^{\beta}\left|\int_0^{+\infty}\int^\infty_{-\infty}\left(1+z^{2\kappa_{\alpha,0}}\right)
                z\overline{\partial_x^\beta\partial_yu_2^0}\,\partial_z\partial_x^{\alpha-\beta}f^B\,\partial_x^\alpha f^B\,dx\,dz
                \right|\\
                &\leq C\sum_{\alpha=0}^{m_0}\|\overline{\partial_{y}u^0_2}\|_{L^\infty_{x}}\|\left(1+z^{2\kappa_{\alpha,0}}\right)\partial_x^{\alpha}f^B\|^2_{L^2_{xz}}\\            &\quad+\sum_{\alpha=0}^{m_0}\sum_{0<\beta\leq\alpha}C_{\alpha}^{\beta}\|\overline{\partial_x^\beta\partial_yu^0_2}\|_{L^\infty_{x}}\|((1+z^{2\kappa_{\alpha-\beta,0}}))^\frac{1}{2}\partial_x^{\alpha-\beta}\partial_z f^B\|_{L^2_{xz}}\|\left(1+z^{2\kappa_{\alpha,0}}\right)^\frac{1}{2}\partial_x^\alpha f^B\|_{L^2_{xz}}\\
                &\leq\delta_1\sum_{\alpha=0}^{m_0}\sum_{0<\beta\leq\alpha}\|\left(1+z^{2\kappa_{\alpha-\beta,0}})\right)^\frac{1}{2}\partial_x^{\alpha-\beta}\partial_z f^B\|_{L^2_{xz}}^2+C\mathcal{E}_{k,0}(f^B).
            \end{align*}
        It then follows immediately that
            \begin{align*}
    			&\frac{d}{dt}\mathcal{E}_{k,0}(f^B)+2(1-C\delta_1)\sum_{\alpha=0}^{m_0}\|\left(1+z^{2\kappa_{\alpha,0}}\right)^\frac{1}{2}\partial_x^\alpha\partial_zf^B\|^2_{L^2_{xz}}\leq C\mathcal{E}_{k,0}(f^B)+C\sum_{\alpha=0}^{m_0}\|\left(1+z^{2\kappa_{\alpha,0}}\right)^\frac{1}{2}\partial_x^\alpha\rho\|^2_{L^2_{xz}}.
                \end{align*}
        Since the assumptions on \(\rho\) hold for every polynomial weight, for sufficiently small $\delta$, Gronwall's inequality yields
    	\begin{align}\label{mathcal{E}_{k,0}}
           \sup_{0\leq t\leq T}\mathcal{E}_{k,0}(f^B)+\sum_{\alpha=0}^{m_0}\int^t_0\|\left(1+z^{2\kappa_{\alpha,0}}\right)^\frac{1}{2}\partial^\alpha_x\partial_zf^B\|^2_{L^2_{xz}}\leq C.
        \end{align}
         Next we will also consider the derivative estimate of $f$ with respect to $z$.\\
        \textbf{Step 2. For $0\leq\alpha\leq m_0-1$.}\\
         Similar to the above procedure in deriving $J_1-J_3,J_5$, we apply $\partial_x^\alpha$ to {the equation $(\ref{f})_1$ and multiply by} $\left(1+z^{2\kappa_{\alpha,1}}\right)\partial_x^\alpha\partial_t f^B$ to have
		\begin{align*}
			&\frac{1}{2}\frac{d}{dt}\sum_{\alpha=0}^{m_0-1}\|\left(1+z^{2\kappa_{\alpha,1}}\right)^\frac{1}{2}\partial^\alpha_x\partial_zf^B\|^2_{L^2_{xz}}+\mathcal{E}_{k,1}(f^B)\\
            &=-\sum_{\alpha=0}^{m_0-1}\int^{+\infty}_{0}\int^{+\infty}_{-\infty}2\kappa_{\alpha,1}z^{2\kappa_{\alpha,1}-1}\partial_x^\alpha\partial_z f^B\partial_x^\alpha\partial_t f^Bdxdz\\
            &\quad-\sum_{\alpha=0}^{m_0-1}\int^{+\infty}_{0}\int^{+\infty}_{-\infty}\left(1+z^{2\kappa_{\alpha,1}}\right)\partial_x^\alpha(\overline{\partial_yu^0_2}f^B)\partial_x^\alpha\partial_t f^Bdxdz\\
            &\quad-\sum_{\alpha=0}^{m_0-1}\int^{+\infty}_{0}\int^{+\infty}_{-\infty}\left(1+z^{2\kappa_{\alpha,1}}\right)\partial_x^\alpha(\overline{u^0_1}\partial_xf^B)\partial_x^\alpha\partial_t f^Bdxdz\\
            &\quad-\sum_{\alpha=0}^{m_0-1}\int^{+\infty}_{0}\int^{+\infty}_{-\infty}\left(1+z^{2\kappa_{\alpha,1}}\right)\partial_x^\alpha(z\overline{\partial_yu^0_2}\partial_zf^B)\partial_x^\alpha\partial_t f^Bdxdz\\
            &\quad-\sum_{\alpha=0}^{m_0-1}\int^{+\infty}_{0}\int^{+\infty}_{-\infty}\left(1+z^{2\kappa_{\alpha,1}}\right)\partial_x^\alpha(\overline{n^0}f^B)\partial_x^\alpha\partial_t f^Bdxdz+\sum_{\alpha=0}^{m_0-1}\int^{+\infty}_{0}\int^{+\infty}_{-\infty}\left(1+z^{2\kappa_{\alpha,1}}\right)\partial_x^\alpha\rho\partial_x^\alpha\partial_t f^Bdxdz\\
            &\leq \frac{3}{4}\mathcal{E}_{k,1}(f^B)+C\sum_{\alpha=0}^{m_0-1}\|\left(1+z^{2\kappa_{\alpha,1}}\right)^\frac{1}{2}\partial_x^\alpha\partial_zf^B\|^2_{L^2_{xz}}+C\mathcal{E}_{k,0}(f^B)+C\sum_{\alpha=0}^{m_0-1}\|\left(1+z^{2\kappa_{\alpha,1}}\right)^\frac{1}{2}\partial_x^\alpha\rho\|^2_{L^2_{xz}}\\
            &\quad+\left|\sum_{\alpha=0}^{m_0-1}\int^{+\infty}_{0}\int^{+\infty}_{-\infty}\left(1+z^{2\kappa_{\alpha,1}}\right)\partial_x^\alpha(z\overline{\partial_yu^0_2}\partial_zf^B)\partial_x^\alpha\partial_t f^Bdxdz\right|.
	\end{align*}
    Indeed, since \(\kappa_{\alpha,1}+1=\kappa_{\alpha,0}\), we have
    \begin{align*}
        &\left|\sum_{\alpha=0}^{m_0-1}
        \int_0^{+\infty}\int^{+\infty}_{-\infty}
        \left(1+z^{2\kappa_{\alpha,1}}\right)
        z\partial_x^\alpha(\overline{\partial_yu_2^0}\partial_zf^B)\partial_x^\alpha\partial_t f^B
        \,dx\,dz
        \right|\leq
        \frac{1}{4}\mathcal{E}_{k,1}(f^B)+
        C
        \sum_{\alpha=0}^{m_0-1}\left\|
        \left(1+z^{2\kappa_{\alpha,0}}\right)^{\frac12}
        \partial_z\partial_x^\alpha f^B
        \right\|_{L_{xz}^2}^2.
    \end{align*}
    Combining the above estimates and using (\ref{mathcal{E}_{k,0}}), an application of Gronwall¡¯s inequality yields that
    \begin{align}\label{fz1k}
			&\sup_{0\leq t\leq T}\sum_{\alpha=0}^{m_0-1}\|\left(1+z^{2\kappa_{\alpha,1}}\right)^\frac{1}{2}\partial^\alpha_x\partial_zf^B\|^2_{L^2_{xz}}\leq C,
	\end{align}
        For the estimates of $\|\partial_t f^B\|_{k,m_0-1,0}$, applying $\partial_x^\alpha\partial_t$ to the equation $(\ref{f})_1$
        \begin{align}\label{eq:fBtt}           &\partial^\alpha_x\partial^2_tf^B+\partial^\alpha_x\partial_t(\overline{\partial_yu^{0}_2}f^B)+\partial^\alpha_x\partial_t(\overline{u^{0}_1}\partial_xf^B)+\partial^\alpha_x\partial_t(z\overline{\partial_yu^{0}_2}\partial_zf^B)+\partial^\alpha_x\partial_t(\overline{n^{0}}f^B)\notag\\
        &=\partial^\alpha_x\partial_t\partial^2_zf^B+\partial^\alpha_x\partial_t\rho,
        \end{align} and multiplying by $\left(1+z^{2\kappa_{\alpha,1}}\right)\partial_x^\alpha\partial_tf^B$, by substituting (\ref{weighted-shift-f}), H{\"o}lder¡¯s inequality and Young¡¯s inequality, we deduce that
        \begin{align*}
			&\frac{1}{2}\frac{d}{dt}\mathcal{E}_{k,1}(f^B)+\sum_{\alpha=0}^{m_0-1}\|\left(1+z^{2\kappa_{\alpha,1}}\right)^\frac{1}{2}\partial^\alpha_x\partial_z\partial_tf^B\|^2_{L^2_{xz}}\\
			&\leq \delta_2\sum_{\alpha=0}^{m_0-1}\|\left(1+z^{2\kappa_{\alpha,1}}\right)^\frac{1}{2}\partial^\alpha_x\partial_z\partial_tf^B\|^2_{L^2_{xz}}+C\left(\mathcal{E}_{k,1}(f^B)+\mathcal{E}_{k,0}(f^B)\right)\\
            &\quad+\delta_2\sum_{\alpha=0}^{m_0-1}\sum_{0<\beta\leq\alpha}\|\left(1+z^{2\kappa_{\alpha-\beta,1}}\right)^\frac{1}{2}\partial^{\alpha-\beta}_x\partial_z\partial_tf^B\|^2_{L^2_{xz}}\\
            &\quad+C\left(\sum_{\alpha=0}^{m_0-1}\|\left(1+z^{2\kappa_{\alpha,0}}\right)^\frac{1}{2}\partial^\alpha_x\partial_zf^B\|^2_{L^2_{xz}}+\sum_{\alpha=0}^{m_0-1}\|\left(1+z^{2\kappa_{\alpha,1}}\right)^\frac{1}{2}\partial_x^\alpha\partial_t\rho\|^2_{L^2_{xz}}\right).
	\end{align*}
        Substituting (\ref{rho})-(\ref{rhotj}) and (\ref{mathcal{E}_{k,0}}), for sufficiently small $\delta_2$, we apply Gronwall¡¯s inequality to obtain
        \begin{align}\label{mathcal{E}_{k,1}}
			&\sup_{0\leq t\leq T}\mathcal{E}_{k,1}(f^B)+\int^T_0\sum_{\alpha=0}^{m_0-1}\|\left(1+z^{2\kappa_{\alpha,1}}\right)^\frac{1}{2}\partial^\alpha_x\partial_z\partial_tf^B\|^2_{L^2_{xz}}\leq C.
	\end{align}
    \textbf{Step 3. For $0\leq\alpha\leq m_0-2$.}\\
    We skip the Leibniz expansion and embedding procedures at this step, as they are essentially the same as in Step 1 and Step 2, and present the estimates directly.
    Taking the inner product of equation (\ref{eq:fBtt}) with $\left(1+z^{2\kappa_{\alpha,2}}\right)\partial_x^\alpha\partial^2_t f^B$ and integrating by parts, we obtain
         \begin{align*}
			&\frac{1}{2}\frac{d}{dt}\sum_{\alpha=0}^{m_0-2}\|\left(1+z^{2\kappa_{\alpha,2}}\right)^\frac{1}{2}\partial^\alpha_x\partial_z\partial_tf^B\|^2_{L^2_{xz}}+\mathcal{E}_{k,2}(f^B)\\
			&\leq \frac{1}{2}\mathcal{E}_{k,2}(f^B)+C\left(\sum_{\alpha=0}^{m_0-2}\|\left(1+z^{2\kappa_{\alpha,2}}\right)^\frac{1}{2}\partial^\alpha_x\partial_z\partial_tf^B\|^2_{L^2_{xz}}+\sum_{\alpha=0}^{m_0-2}\|\left(1+z^{2\kappa_{\alpha,1}}\right)^\frac{1}{2}\partial^\alpha_x\partial_z\partial_tf^B\|^2_{L^2_{xz}}\right)\\
            &\quad+C\left(\sum_{\alpha=0}^{m_0-1}\|\left(1+z^{2\kappa_{\alpha,2}}\right)^\frac{1}{2}\partial^\alpha_xf^B\|^2_{L^2_{xz}}+\sum_{\alpha=0}^{m_0-2}\|\left(1+z^{2\kappa_{\alpha,1}}\right)^\frac{1}{2}\partial^\alpha_x\partial_zf^B\|^2_{L^2_{xz}}\right)\\
            &\quad+C\left(\mathcal{E}_{k,1}(f^B)+\sum_{\alpha=0}^{m_0-2}\|\left(1+z^{2\kappa_{\alpha,2}}\right)^\frac{1}{2}\partial_x^\alpha\partial_t\rho\|^2_{L^2_{xz}}\right),
	      \end{align*}
        due to \(\kappa_{\alpha,2}+1=\kappa_{\alpha,1}\). Combining (\ref{mathcal{E}_{k,0}}), (\ref{fz1k}) and (\ref{mathcal{E}_{k,1}}), an application of Gronwall¡¯s inequality yields that
        \begin{align}\label{fzt1}
			&\sup_{0\leq t\leq T}\sum_{\alpha=0}^{m_0-2}\|\left(1+z^{2\kappa_{\alpha,2}}\right)^\frac{1}{2}\partial^\alpha_x\partial_z\partial_tf^B\|^2_{L^2_{xz}}+\int^T_0\mathcal{E}_{k,2}(f^B)\leq C,
	      \end{align}
    To estimate $\|\partial^2_t f^B\|_{k,m_0-2,0}$, making further derivation by $\partial^\alpha_x\partial^2_t$ for the equation $(\ref{f})_1$, which then can be expressed as follows
    \begin{align}\label{ft3}     &\partial^\alpha_x\partial^3_tf^B+\partial^\alpha_x\partial^2_t(\overline{\partial_yu^{0}_2}f^B)+\partial^\alpha_x\partial^2_t(\overline{u^{0}_1}\partial_xf^B)+\partial^\alpha_x\partial^2_t(z\overline{\partial_yu^{0}_2}\partial_zf^B)+\partial^\alpha_x\partial^2_t(\overline{n^{0}}f^B)\\
    &=\partial^\alpha_x\partial^2_t\partial^2_zf^B+\partial^\alpha_x\partial^2_t\rho, \notag\notag
    \end{align}
    multiply (\ref{ft3}) by $\left(1+z^{2\kappa_{\alpha,2}}\right)\partial_x^\alpha\partial^2_t f^B$, it then yields
         \begin{align*}
			&\frac{1}{2}\frac{d}{dt}\mathcal{E}_{k,2}(f^B)+\sum_{\alpha=0}^{m_0-2}\|\left(1+z^{2\kappa_{\alpha,2}}\right)^\frac{1}{2}\partial^\alpha_x\partial_z\partial^2_tf^B\|^2_{L^2_{xz}}\\
			&\leq \delta_3\sum_{\alpha=0}^{m_0-2}\|\left(1+z^{2\kappa_{\alpha,2}}\right)^\frac{1}{2}\partial^\alpha_x\partial_z\partial^2_tf^B\|^2_{L^2_{xz}}+C\left(\mathcal{E}_{k,2}(f^B)+\mathcal{E}_{k,1}(f^B)+\sum_{\alpha=0}^{m_0-1}\|\left(1+z^{2\kappa_{\alpha,2}}\right)^\frac{1}{2}\partial^\alpha_xf^B\|^2_{L^2_{xz}}\right)\\
            &\quad+\delta_3\sum_{\alpha=0}^{m_0-2}\sum_{0<\beta\leq\alpha}\|\left(1+z^{2\kappa_{\alpha-\beta,2}}\right)^\frac{1}{2}\partial^{\alpha-\beta}_x\partial_z\partial^2_tf^B\|^2_{L^2_{xz}}\\
            &\quad+C\sum_{\alpha=0}^{m_0-2}\Big{(}\|\left(1+z^{2\kappa_{\alpha,1}}\right)^\frac{1}{2}\partial^\alpha_x\partial_z\partial_tf^B\|^2_{L^2_{xz}}+\|\left(1+z^{2\kappa_{\alpha,1}}\right)^\frac{1}{2}\partial^\alpha_x\partial_zf^B\|^2_{L^2_{xz}}\Big{)}\\
            &\quad+C\sum_{\alpha=0}^{m_0-2}\|\left(1+z^{2\kappa_{\alpha,2}}\right)^\frac{1}{2}\partial^\alpha_x\partial^2_t\rho\|^2_{L^2_{xz}}.
	      \end{align*}
        Substituting $(\ref{mathcal{E}_{k,0}})$, $(\ref{fz1k})$, $(\ref{mathcal{E}_{k,1}})$ and $(\ref{fzt1})$, for sufficiently small $\delta_3$, we apply Gronwall¡¯s inequality to deduce that
        \begin{align}\label{mathcal{E}_{k,2}}
			&\sup_{0\leq t\leq T}\mathcal{E}_{k,2}(f^B)+\int^T_0\sum_{\alpha=0}^{m_0-2}\|\left(1+z^{2\kappa_{\alpha,2}}\right)^\frac{1}{2}\partial^\alpha_x\partial_z\partial^2_tf^B\|^2_{L^2_{xz}}\leq C.
	      \end{align}
          Combining the previous high-order weighted estimates $(\ref{mathcal{E}_{k,0}})$, $(\ref{fz1k})$, $(\ref{mathcal{E}_{k,1}})$, $(\ref{fzt1})$ and $(\ref{mathcal{E}_{k,2}})$, we obtain the following results for $\kappa_{\alpha,j}\ge k$, $j=0,1,2$.
        \begin{align}\label{fz0}            \|f^B\|_{L^\infty(0,T;\mathcal{H}^{k,m_0,0})}^2+\|\partial_zf^B\|_{L^2(0,T;\mathcal{H}^{k,m_0,0})}^2&\leq C,\notag\\
       \|\partial_zf^B\|^2_{L^\infty(0,T;\mathcal{H}^{k,m_0-1,0})}&\leq C,\notag\\
			\|\partial_t f^B\|^2_{L^\infty(0,T;\mathcal{H}^{k,m_0-1,0})}+\|\partial_t \partial_zf^B\|^2_{L^2(0,T;\mathcal{H}^{k,m_0-1,0})}&\leq C,\\
			\|\partial_t \partial_zf^B\|^2_{L^\infty(0,T;\mathcal{H}^{k,m_0-2,0})}+\|\partial^2_t f^B\|^2_{L^2(0,T;\mathcal{H}^{k,m_0-2,0})}&\leq C,\notag\\
			\|\partial^2_tf^B\|^2_{L^\infty(0,T;\mathcal{H}^{k,m_0-2,0})}+\|\partial_z\partial^2_t f^B\|^2_{L^2(0,T;\mathcal{H}^{k,m_0-2,0})}&\leq C.\notag\notag
	      \end{align}
        By $(\ref{fz1k})$, $(\ref{fzt1})$ and (\ref{fz0}), further application for the equation $(\ref{f})$ and $(\ref{normal defined})$ shows that
    \begin{align}\label{fz2t}
        \|\partial_t f^B\|_{k,m_0-3,2}&\leq C\left(\|\partial^2_tf^B\|_{k,m_0-3,0}+\|\partial_tf^B\|_{k,m_0-2,0}+\|\partial_t f^B\|_{k+1,m_0-3,1}\right)\notag\\
        &\quad+C\left(\|f^B\|_{k,m_0-2,0}+\|f^B\|_{k+1,m_0-3,1}+\|\partial_t \rho\|_{k,m_0-3,0}\right)\\
        &\leq C.
        \notag\notag
    \end{align}
    From the equation $(\ref{f})$, together with $(\ref{mathcal{E}_{k,0}})$, $(\ref{fz1k})$, $(\ref{mathcal{E}_{k,1}})$, $(\ref{fzt1})$, $(\ref{mathcal{E}_{k,2}})$ and the assumption on $\rho$ in $(\ref{rhotj})$, we further derive for $k\in N$ that
    \begin{equation}\label{fBz234}
		\begin{split}{}
			\|f^B\|_{k,m_0-i,i}&\leq C\left(\|\partial_t f^B\|_{k,m_0-i,i-2}+\|f^B\|_{k,m_0-i+1,i-2}+\|f^B\|_{k+1,m_0-i,i-1}+\|\rho\|_{k,m_0-i,i-2}\right)\leq C,
		\end{split}
	\end{equation}
    where $i=2,3,4$.
       Combining $(\ref{fz0})$-$(\ref{fz2t})$, we get the desired estimates and complete the proof.\\

       For the regularity of $(\ref{eq:vB11})$, $(\ref{eq:uB11})$, $(\ref{eq:vB21})$ and $(\ref{eq:uB21})$, we introduce the following system related to $g^B(t,x,z)$
        \begin{align}\label{g}
        \begin{cases}        \partial_tg^B+\partial_x(\overline{u^{0}_1}g^B)+z\overline{\partial_yu^{0}_2}\partial_zg^B=\partial^2_zg^B+\varrho\\
        \partial_zg^B|_{z=0}=\mu(t,x),\\
        g^B|_{t=0}=0.
        \end{cases}
	    \end{align}
        Then the regularity result of the solution to the equations (\ref{g}) below holds.
       \begin{proposition}\label{regularity of g}
		{ Suppose $(n^0,v^0,u^0)$ be the solutions that satisfy the conditions in Theorem \ref{th2.1} on $[0,T]$ such that for $m_0\geq4$
        \begin{align}\label{varrho}
        \partial^j_t{u^{0}_1}, \partial^j_t{u^{0}_2}\in L^\infty(0,T;H^{m_0-j+3}_{xy}),\partial^j_t\mu\in L^2(0,T;H^{m_0-j}_x),
        \end{align}
       for any $k\in N$ and $\varrho$ satisfies
        \begin{align}\label{varrho t}
            \partial^j_t\varrho&\in L^2(0,T;\mathcal{H}^{k+m_0-j,m_0-j,0}),\;\partial_t\varrho\in L^\infty(0,T;\mathcal{H}^{k,m_0-3,0}),\;\varrho\in L^\infty(0,T;\mathcal{H}^{k,m_0-l,l-2}),\;l=2,3,4
        \end{align}
           for $j=0,1,2$. Then the solution $g^B$ to the equations $(\ref{g})$ satisfies the following properties:
			\begin{equation}\label{gBt}
				\begin{split}{}
                  \partial^j_tg^B&\in L^\infty(0,T;\mathcal{H}^{k,m_0-j,0})\cap L^2(0,T;\mathcal{H}^{k,m_0-j,1}),\;j=0,1,2,\\
                  \partial_tg^B&\in L^\infty(0,T; \mathcal{H}^{k,m_0-2,1}),\;\partial_tg^B\in L^\infty(0,T; \mathcal{H}^{k,m_0-3,2}),
				\end{split}
			\end{equation}
            and
            \begin{align}\label{gBz}
               g^B\in L^\infty(0,T;\mathcal{H}^{k,m_0-i,i}),\;i=1,2,3,4.
            \end{align}

		}
	\end{proposition}
    \noindent{\bf{Proof.}} In the present proof, we employ the same weighted convention as in the proof of Proposition \ref{regularity of f^B}. In particular, the terms containing a factor \(z\) are controlled by assigning a larger polynomial weight to a lower-order tangential derivative. Therefore, no uncontrolled higher weighted norm of \(g^B\) appears.\\
    \textbf{Step 1. }For $0\leq\alpha\leq m_0$.\\
    {Applying $\partial_x^\alpha$ to the equations  $(\ref{g})_1$, multiplying by} $(1+z^{2\kappa_{\alpha,0}})\partial_x^\alpha g^B$ and using the Leibniz formula, we obtain
	\begin{align*}
			&\frac{1}{2}\frac{d}{dt}\mathcal{E}_{k,0}(g^B)+\sum_{\alpha=0}^{m_0}\|\left(1+z^{2\kappa_{\alpha,0}}\right)^\frac{1}{2}\partial_x^\alpha\partial_zg^B\|^2_{L^2_{xz}}\notag\\
			&=-\sum_{\alpha=0}^{m_0}\int^{+\infty}_{0}\int^{+\infty}_{-\infty}2\kappa_{\alpha,0}z^{2\kappa_{\alpha,0}-1}\partial_x^\alpha\partial_z g^B\partial_x^\alpha g^Bdxdz-\sum_{\alpha=0}^{m_0}\int^{+\infty}_{0}\int^{+\infty}_{-\infty}(1+z^{2\kappa_{\alpha,0}})\partial_x^\alpha(\partial_x\overline{u^0_1}g^B)\partial_x^\alpha g^Bdxdz\notag\\
            &\quad-\sum_{\alpha=0}^{m_0}\int^{+\infty}_{0}\int^{+\infty}_{-\infty}(1+z^{2\kappa_{\alpha,0}})\partial_x^\alpha(\overline{u^0_1}\partial_xg^{B})\partial_x^\alpha g^Bdxdz\notag\\
            &\quad-\sum_{\alpha=0}^{m_0}\int^{+\infty}_{0}\int^{+\infty}_{-\infty}(1+z^{2\kappa_{\alpha,0}})\partial_x^\alpha(z\overline{\partial_yu^0_2}\partial_zg^B)\partial_x^\alpha g^Bdxdz\\
			&\quad+\sum_{\alpha=0}^{m_0}\int^{+\infty}_{0}\int^{+\infty}_{-\infty}(1+z^{2\kappa_{\alpha,0}})\partial_x^\alpha\varrho\partial_x^\alpha g^Bdxdz+\sum_{\alpha=0}^{m_0}\int^{+\infty}_{-\infty}\partial_x^\alpha\partial_zg^B(t,x,0)\partial_x^\alpha g^B(t,x,0)dx\notag\\
            &=K_1+...+K_5+\sum_{\alpha=0}^{m_0}\int^{+\infty}_{-\infty}\partial_x^\alpha\partial_zg^B(t,x,0)\partial_x^\alpha g^B(t,x,0)dx.\notag\notag
	\end{align*}
    We collect the estimates for \(K_1\)-\(K_4\) as follows:
    \begin{align*}
                K_1&\leq\epsilon_1\sum_{\alpha=0}^{m_0}\|\left(1+z^{2\kappa_{\alpha,0}}\right)^\frac{1}{2}\partial_x^\alpha\partial_z g^B\|_{L^2_{xz}}^2+C\mathcal{E}_{k,0}(g^B),\\
                K_2&\leq \sum_{\alpha=0}^{m_0}\sum_{0\leq\beta\leq\alpha}C_{\alpha}^{\beta}\|\overline{\partial_x^{\beta+1} u^0_1}\|_{L^\infty_{x}}\|\left(1+z^{2\kappa_{\alpha,0}}\right)^\frac{1}{2}\partial_x^{\alpha-\beta} g^B\|_{L^2_{xz}}\|\left(1+z^{2\kappa_{\alpha,0}}\right)^\frac{1}{2}\partial_x^\alpha g^B\|_{L^2_{xz}}\leq C\mathcal{E}_{k,0}(g^B),\\                K_3&\leq\sum_{\alpha=0}^{m_0}\|\overline{\partial_{x}u^0_1}\|_{L^\infty_{x}}\|\left(1+z^{2\kappa_{\alpha,0}}\right)^\frac{1}{2}\partial_x^{\alpha}g^B\|^2_{L^2_{xz}}\\            &\quad+\sum_{\alpha=0}^{m_0}\sum_{0<\beta\leq\alpha}C_{\alpha}^{\beta}\|\overline{\partial_x^{\beta}u^0_1}\|_{L^\infty_{x}}\|\left(1+z^{2\kappa_{\alpha,0}}\right)^\frac{1}{2}\partial_x^{\alpha-\beta+1}g^B\|_{L^2_{xz}}\|\left(1+z^{2\kappa_{\alpha,0}}\right)^\frac{1}{2}\partial_x^{\alpha}g^B\|_{L^2_{xz}}\\
                &\leq C\mathcal{E}_{k,0}(g^B),\\
                K_4&\leq\frac{1}{2}\sum_{\alpha=0}^{m_0}\int_0^{+\infty}\int^\infty_{-\infty}\overline{\partial_yu_2^0}\,
                \partial_z\left[z\left(1+z^{2\kappa_{\alpha,0}}\right)\right]\left|\partial_x^\alpha g^B\right|^2\,dx\,dz\\
                &\quad+\sum_{\alpha=0}^{m_0}\sum_{0<\beta\leq\alpha}C_{\alpha}^{\beta}\left|\int_0^{+\infty}\int^\infty_{-\infty}\left(1+z^{2\kappa_{\alpha,0}}\right)
                z\overline{\partial_x^\beta\partial_yu_2^0}\,\partial_z\partial_x^{\alpha-\beta}g^B\,\partial_x^\alpha g^B\,dx\,dz
                \right|\\
                &\leq C\sum_{\alpha=0}^{m_0}\|\overline{\partial_{y}u^0_2}\|_{L^\infty_{x}}\|\left(1+z^{2\kappa_{\alpha,0}}\right)\partial_x^{\alpha}g^B\|^2_{L^2_{xz}}\\            &\quad+\sum_{\alpha=0}^{m_0}\sum_{0<\beta\leq\alpha}C_{\alpha}^{\beta}\|\overline{\partial_x^\beta\partial_yu^0_2}\|_{L^\infty_{x}}\|((1+z^{2\kappa_{\alpha-\beta,0}}))^\frac{1}{2}\partial_x^{\alpha-\beta}\partial_z g^B\|_{L^2_{xz}}\|\left(1+z^{2\kappa_{\alpha,0}}\right)^\frac{1}{2}\partial_x^\alpha g^B\|_{L^2_{xz}}\\
                &\leq\epsilon_1\sum_{\alpha=0}^{m_0}\sum_{0<\beta\leq\alpha}\|\left(1+z^{2\kappa_{\alpha-\beta,0}})\right)^\frac{1}{2}\partial_x^{\alpha-\beta}\partial_z g^B\|_{L^2_{xz}}^2+C\mathcal{E}_{k,0}(g^B).
            \end{align*}
            Since the assumptions on \(\varrho\) hold for every polynomial weight, we have
            \begin{align*}
                K_5\leq C\sum_{\alpha=0}^{m_0}\|\left(1+z^{2\kappa_{\alpha,0}}\right)^\frac{1}{2}\partial_x^\alpha\varrho\|_{L^2_{xz}}^2+C\mathcal{E}_{k,0}(g^B).
            \end{align*}
        Applying the boundary condition $\partial_x^\alpha\partial_zg^B(t,x,0)=\partial_x^\alpha\mu$ and the trace theorem, we obtain the following estimate:
        \begin{equation}\label{m0 boundary}
				\begin{split}{}
				&\sum_{\alpha=0}^{m_0}\int^{+\infty}_{-\infty}\partial_x^\alpha\partial_zg^B(t,x,0)\partial_x^\alpha g^B(t,x,0)dx\\
                &\leq\sum_{\alpha=0}^{m_0}\int^{+\infty}_{0}\int^{+\infty}_{-\infty}|\partial_x^\alpha\mu|\|\left(1+z^{2\kappa_{\alpha,0}}\right)^\frac{1}{2}\partial_x^\alpha g^B(t,x,z)\|^\frac{1}{2}_{L^2_{xz}}\|\left(1+z^{2\kappa_{\alpha,0}}\right)^\frac{1}{2}\partial_x^\alpha\partial_zg^B(t,x,z)\|^\frac{1}{2}_{L^2_{xz}}dxdz\\
                &\leq \epsilon_1\sum_{\alpha=0}^{m_0}\|\left(1+z^{2\kappa_{\alpha,0}}\right)^\frac{1}{2}\partial_x^\alpha\partial_zg^B\|^2_{L^2_{xz}}+C\mathcal{E}_{k,0}(g^B)+C\|\mu\|^2_{H^{m_0}_{x}}.
                \end{split}
			\end{equation}
        Combining $K_1$-$K_5$ and $(\ref{m0 boundary})$, for sufficiently small $\epsilon_1$ and applying Gronwall's inequality, we deduce that
        \begin{align}\label{mathcal{E}_{k,0}(gB)}
			\sup_{0\leq t\leq T}\mathcal{E}_{k,0}(g^B)+\int^T_0\sum_{\alpha=0}^{m_0}\|\left(1+z^{2\kappa_{\alpha,0}}\right)^\frac{1}{2}\partial_x^\alpha\partial_zg^B\|^2_{L^2_{xz}}\leq C.
	\end{align}
    \textbf{Step 2.} For $0\leq\alpha\leq m_0-1$.\\
         Similar to the above procedure, apply $\partial_x^\alpha$ to {the equation $(\ref{g})_1$ and multiply by} $\left(1+z^{2\kappa_{\alpha,1}}\right)\partial_x^\alpha\partial_t g^B$, we have
         \begin{align}\label{gz1}
			&\frac{1}{2}\frac{d}{dt}\sum_{\alpha=0}^{m_0-1}\|\left(1+z^{2\kappa_{\alpha,1}}\right)^\frac{1}{2}\partial_x^\alpha\partial_zg^B\|^2_{L^2_{xz}}+\mathcal{E}_{k,1}(g^B)\notag\\
			&\leq \frac{1}{2}\mathcal{E}_{k,1}(g^B)+C\sum_{\alpha=0}^{m_0-1}\|\left(1+z^{2\kappa_{\alpha,1}}\right)^\frac{1}{2}\partial_x^\alpha\partial_zg^B\|^2_{L^2_{xz}}+C\mathcal{E}_{k,0}(g^B)\notag\\
            &\quad+C\left(\sum_{\alpha=0}^{m_0-1}\|\left(1+z^{2\kappa_{\alpha,0}}\right)^\frac{1}{2}\partial_x^\alpha\partial_zg^B\|^2_{L^2_{xz}}+\sum_{\alpha=0}^{m_0-1}\|\left(1+z^{2\kappa_{\alpha,1}}\right)^\frac{1}{2}\partial_x^\alpha\varrho\|_{L^2_{xz}}^2\right)\\
            &\quad+\sum_{\alpha=0}^{m_0-1}\int^{+\infty}_{-\infty}\partial_x^\alpha\partial_zg^B(t,x,0)\partial_x^\alpha\partial_t g^B(t,x,0)dx,\notag\notag
	\end{align}
        with
        \begin{equation*}
				\begin{split}{}
				&\sum_{\alpha=0}^{m_0-1}\int^{+\infty}_{-\infty}\partial_x^\alpha\partial_zg^B(t,x,0)\partial_x^\alpha \partial_tg^B(t,x,0)dx\\
                &\leq \epsilon_2\sum_{\alpha=0}^{m_0-1}\|\left(1+z^{2\kappa_{\alpha,1}}\right)^\frac{1}{2}\partial_x^\alpha\partial_z\partial_tg^B\|^2_{L^2_{xz}}+C\left(\mathcal{E}_{k,1}(g^B)+\|\mu\|^2_{H^{m_0-1}_{x}}\right).
                \end{split}
			\end{equation*}
        On the other hand, applying $\partial_x^\alpha\partial_t$ to the equation $(\ref{g})_1$, we obtain
        \begin{align}\label{eq:gBtt}            \partial^\alpha_x\partial^2_tg^B+\partial^\alpha_x\partial_t(\overline{\partial_xu^{0}_1}g^B+\overline{u^{0}_1}\partial_xg^B)+\partial^\alpha_x\partial_t(z\overline{\partial_yu^{0}_2}\partial_zg^B)
            =\partial^\alpha_x\partial_t\partial^2_zg^B+\partial^\alpha_x\partial_t\rho,
        \end{align}
        and multiplying by $\left(1+z^{2\kappa_{\alpha,1}}\right)\partial_x^\alpha\partial_t g^B$, we get
        \begin{align}\label{gt0}
			&\frac{1}{2}\frac{d}{dt}\mathcal{E}_{k,1}(g^B)+\sum_{\alpha=0}^{m_0-1}\|\left(1+z^{2\kappa_{\alpha,1}}\right)^\frac{1}{2}\partial_x^\alpha\partial_z\partial_tg^B\|^2_{L^2_{xz}}\notag\\
			&\leq \epsilon_2\sum_{\alpha=0}^{m_0-1}\|\left(1+z^{2\kappa_{\alpha,1}}\right)^\frac{1}{2}\partial_x^\alpha\partial_z\partial_tg^B\|^2_{L^2_{xz}}+C\left(\mathcal{E}_{k,1}(g^B)+\mathcal{E}_{k,0}(g^B)\right)\\
            &\quad+\epsilon_2\sum_{\alpha=0}^{m_0-1}\sum_{0<\beta\leq\alpha}\|\left(1+z^{2\kappa_{\alpha-\beta,1}}\right)^\frac{1}{2}\partial_x^{\alpha-\beta}\partial_z\partial_tg^B\|^2_{L^2_{xz}}+C\sum_{\alpha=0}^{m_0-1}\|\left(1+z^{2\kappa_{\alpha,0}}\right)^\frac{1}{2}\partial_x^\alpha\partial_zg^B\|^2_{L^2_{xz}}\notag\\
            &\quad+C\left(\sum_{\alpha=0}^{m_0-1}\|\left(1+z^{2\kappa_{\alpha,1}}\right)^\frac{1}{2}\partial_x^\alpha\partial_t\varrho\|^2_{L^2_{xz}}+\|\partial_t\mu\|^2_{H^{m_0-1}_{x}}\right),\notag\notag
	\end{align}
    due to $\kappa_{\alpha,1}+1=\kappa_{\alpha,0}$. Combining $(\ref{gz1})$ and $(\ref{gt0})$, substituting the conclusion of $(\ref{mathcal{E}_{k,0}(gB)})$, employing the assumption on (\ref{varrho}) and Gronwall's inequality, for sufficiently small $\epsilon_2$, we deduce that
        \begin{equation}\label{gz1 infty}
		\begin{split}{}
			&\sup_{0\leq t\leq T}\left(\sum_{\alpha=0}^{m_0-1}\|\left(1+z^{2\kappa_{\alpha,1}}\right)^\frac{1}{2}\partial_x^\alpha\partial_zg^B\|^2_{L^2_{xz}}+\mathcal{E}_{k,1}(g^B)\right)+\int^T_0\sum_{\alpha=0}^{m_0-1}\|\left(1+z^{2\kappa_{\alpha,1}}\right)^\frac{1}{2}\partial_x^\alpha\partial_z\partial_tg^B\|^2_{L^2_{xz}}\leq C.
		\end{split}
	\end{equation}

    \textbf{Step 3.}For $0\leq\alpha\leq m_0-2$.\\
    Multiplying equation (\ref{eq:gBtt}) by $(1+z^{2\kappa_{\alpha,2}})\partial_x^\alpha\partial^2_t g^B$ and integrating by parts, we obtain that
         \begin{align}\label{gzt}
			&\frac{1}{2}\frac{d}{dt}\sum_{\alpha=0}^{m_0-2}\|\left(1+z^{2\kappa_{\alpha,2}}\right)^\frac{1}{2}\partial_x^\alpha\partial_z\partial_tg^B\|^2_{L^2_{xz}}+\mathcal{E}_{k,2}(g^B)\notag\\
			&\leq \frac{1}{2}\mathcal{E}_{k,2}(g^B)+\epsilon_3\sum_{\alpha=0}^{m_0-2}\|\left(1+z^{2\kappa_{\alpha,2}}\right)^\frac{1}{2}\partial_x^\alpha\partial_z\partial^2_tg^B\|^2_{L^2_{xz}}+C\mathcal{E}_{k,1}(g^B)\\
            &\quad+C\left(\sum_{\alpha=0}^{m_0-1}\|\left(1+z^{2\kappa_{\alpha,2}}\right)^\frac{1}{2}\partial_x^\alpha g^B\|^2_{L^2_{xz}}+\sum_{\alpha=0}^{m_0-2}\|\left(1+z^{2\kappa_{\alpha,1}}\right)^\frac{1}{2}\partial_x^\alpha\partial_z\partial_tg^B\|^2_{L^2_{xz}}\right)\notag\\
            &\quad+C\left(\sum_{\alpha=0}^{m_0-2}\|\left(1+z^{2\kappa_{\alpha,1}}\right)^\frac{1}{2}\partial_x^\alpha\partial_zg^B\|^2_{L^2_{xz}}+\sum_{\alpha=0}^{m_0-2}\|\left(1+z^{2\kappa_{\alpha,2}}\right)^\frac{1}{2}\partial_x^\alpha\partial_t\varrho\|^2_{L^2_{xz}}+\|\partial_t\mu\|^2_{H^{m_0-2}_{x}}\right).\notag\notag
	\end{align}
        On the other hand, applying $\partial_x^\alpha\partial^2_t$ to the equation $(\ref{g})_1$
        $$\partial^\alpha_x\partial^3_tg^B+\partial^\alpha_x\partial^2_t(\overline{\partial_xu^{0}_1}g^B+\overline{u^{0}_1}\partial_xg^B)+\partial^\alpha_x\partial^2_t(z\overline{\partial_yu^{0}_2}\partial_zg^B)
    =\partial^\alpha_x\partial^2_t\partial^2_zg^B+\partial^\alpha_x\partial^2_t\rho,$$
        and multiplying by $(1+z^{2\kappa_{\alpha,2}})\partial_x^\alpha\partial^2_t g^B$, we get
        \begin{align}\label{gtt}
			&\frac{1}{2}\frac{d}{dt}\mathcal{E}_{k,2}(g^B)+\sum_{\alpha=0}^{m_0-2}\|\left(1+z^{2\kappa_{\alpha,2}}\right)^\frac{1}{2}\partial_x^\alpha\partial_z\partial^2_tg^B\|^2_{L^2_{xz}}\notag\\
			&\leq \epsilon_3\sum_{\alpha=0}^{m_0-2}\|\left(1+z^{2\kappa_{\alpha,2}}\right)^\frac{1}{2}\partial_x^\alpha\partial_z\partial^2_tg^B\|^2_{L^2_{xz}}+C\left(\mathcal{E}_{k,2}(g^B)+\mathcal{E}_{k,1}(g^B)+\sum_{\alpha=0}^{m_0-1}\|\left(1+z^{2\kappa_{\alpha,2}}\right)^\frac{1}{2}\partial_x^\alpha g^B\|^2_{L^2_{xz}}\right)\notag\\
            &\quad+\epsilon_3\sum_{\alpha=0}^{m_0-2}\sum_{0<\beta\leq\alpha}\|\left(1+z^{2\kappa_{\alpha-\beta,2}}\right)^\frac{1}{2}\partial_x^{\alpha-\beta}\partial_z\partial^2_tg^B\|^2_{L^2_{xz}}\\
            &\quad+C\sum_{\alpha=0}^{m_0-2}\|\left(1+z^{2\kappa_{\alpha,1}}\right)^\frac{1}{2}\partial_x^\alpha\partial_z\partial_tg^B\|^2_{L^2_{xz}}+C\sum_{\alpha=0}^{m_0-2}\|\left(1+z^{2\kappa_{\alpha,1}}\right)^\frac{1}{2}\partial_x^\alpha\partial_zg^B\|^2_{L^2_{xz}}\notag\\
            &\quad+\sum_{\alpha=0}^{m_0-2}\|\left(1+z^{2\kappa_{\alpha,2}}\right)^\frac{1}{2}\partial_x^\alpha\partial^2_t\varrho\|^2_{L^2_{xz}}+C\|\partial^2_t\mu\|^2_{H^{m_0-2}_{x}}.\notag\notag
	\end{align}
    Combining $(\ref{gzt})$ and $(\ref{gtt})$, substituting the conclusion of $(\ref{mathcal{E}_{k,0}(gB)})$ and (\ref{gz1 infty}), employing the assumption on (\ref{varrho})-(\ref{varrho t}) and Gronwall's inequality, for sufficiently small $\epsilon_3$, we have
        \begin{align}\label{gzt1 and mathcal{E}_{k,2}(g^B)}
			&\sup_{0\leq t\leq T}\left(\sum_{\alpha=0}^{m_0-2}\|\left(1+z^{2\kappa_{\alpha,2}}\right)^\frac{1}{2}\partial_x^\alpha\partial_z\partial_tg^B\|^2_{L^2_{xz}}+\mathcal{E}_{k,2}(g^B)\right)+\int^T_0\sum_{\alpha=0}^{m_0-2}\|\left(1+z^{2\kappa_{\alpha,2}}\right)^\frac{1}{2}\partial_x^\alpha\partial_z\partial^2_tg^B\|^2_{L^2_{xz}}\leq C.
	\end{align}

    Combining the high-order weighted estimates $(\ref{mathcal{E}_{k,0}(gB)})$, $(\ref{gz1 infty})$ and $(\ref{gzt1 and mathcal{E}_{k,2}(g^B)})$, we obtain that for any $\kappa_{\alpha,j}\geq k$, $j=0,1,2$,
        \begin{align}\label{gz0}	\|g^B\|^2_{L^\infty(0,T;\mathcal{H}^{k,m_0,0})}+\|\partial_zg^B\|^2_{L^2(0,T;\mathcal{H}^{k,m_0,0})}&\leq C,\notag\\
			\|\partial_zg^B\|^2_{L^\infty(0,T;\mathcal{H}^{k,m_0-1,0})}+\|\partial_t g^B\|^2_{L^\infty(0,T;\mathcal{H}^{k,m_0-1,0})}+\|\partial_t \partial_zg^B\|^2_{L^2(0,T;\mathcal{H}^{k,m_0-1,0})}&\leq C,\\
			\|\partial_z\partial_tg^B\|^2_{L^\infty(0,T;\mathcal{H}^{k,m_0-2,0})}+\|\partial^2_t g^B\|^2_{L^\infty(0,T;\mathcal{H}^{k,m_0-2,0})}+\|\partial_z\partial^2_t g^B\|^2_{L^2(0,T;\mathcal{H}^{k,m_0-2,0})}&\leq C.\notag\notag
	\end{align}
    Further application of equation $(\ref{g})_1$, $(\ref{gz1 infty})$, $(\ref{gzt1 and mathcal{E}_{k,2}(g^B)})$ and $(\ref{gz0})$ shows that
    \begin{align}\label{gz2t}
        &\|\partial_t g^B\|_{{k,m_0-3,2}}\notag\\
        &\leq C\left(\|\partial^2_tg^B\|_{{k,m_0-3,0}}+\|\partial_tg^B\|_{{k,m_0-3,0}}+\|\partial_tg^B\|_{{k,m_0-2,0}}+\|\partial_t g^B\|_{{k+1,m_0-3,1}}\right)\notag\\
        &\quad+C\left(\|g^B\|_{{k,m_0-2,0}}+\|g^B\|_{{k+1,m_0-3,1}}+\|\partial_t \varrho\|_{{k,m_0-3,0}}\right)\notag\\
        &\leq C.\notag\notag
    \end{align}
    From equation $(\ref{g})$, together with $(\ref{gz1 infty})$, $(\ref{gzt1 and mathcal{E}_{k,2}(g^B)})$, $(\ref{gz0})$ and the assumption on $\varrho$ in $(\ref{varrho t})$, we further derive for $k\in N$ that
    \begin{equation}\label{gBz234}
		\begin{split}{}
			\|g^B\|_{k,m_0-i,i}&\leq C\left(\|\partial_t g^B\|_{k,m_0-i,i-2}+\|g^B\|_{k,m_0-i+1,i-2}+\|g^B\|_{k+1,m_0-i,i-1}+\|\varrho\|_{k,m_0-i,i-2}\right)\leq C,
		\end{split}
	\end{equation}
    where $i=2,3,4$.
    Therefore, we get the desired estimates and complete the proof.\\

      Clearly, among the systems $(\ref{eq:vB02})$-$(\ref{eq:pB2})$, the function $u^{B,1}_1$ defined in (\ref{eq:uB11}) appears repeatedly in the other equations. To simplify the well-posedness analysis for solutions of the system, we rearrange the order of derivation and first establish the well-posedness of $u^{B,1}_1$.
	\begin{lemma}\label{lem5.1}
		{Let $(n^0,v^0,u^0)$ denote the solution on $[0,T]$ constructed in Theorem \ref{th2.1}. Then equation $(\ref{eq:uB11})$ possesses a unique solution $u^{B,1}_1$, and for any $7\leq m_1 \leq m-4$, the following holds:
			\begin{equation}\label{uB11t2}
                \begin{split}{}
				\partial^j_tu^{B,1}_1&\in L^\infty(0,T; \mathcal{H}^{k,m_1-j,0})\cap L^2(0,T; \mathcal{H}^{k,m_1-j,1}),j=0,1,2,\\
                \end{split}
			\end{equation}
            and
    \begin{equation}\label{uB11z2t}
           \begin{split}{}
            \partial_tu^{B,1}_1\in L^\infty(0,T; \mathcal{H}^{k,m_1-2,1}),\;\partial_tu^{B,1}_1\in L^\infty(0,T; \mathcal{H}^{k,m_1-3,2}).\\
                \end{split}
			\end{equation}
            for $k\in N$. Furthermore, we have the following results
            \begin{equation}\label{uB14z}
           \begin{split}{}
              u^{B,1}_1\in L^\infty(0,T; \mathcal{H}^{k,m_1-i,i}),i=1,2,3,4
            \end{split}
			\end{equation}
           and that \begin{equation}\label{uB22}
                \begin{split}{}
				\partial^j_tu^{B,2}_2&\in L^\infty(0,T; \mathcal{H}^{k,m_1-1-j,1})\cap L^2(0,T; \mathcal{H}^{k,m_1-1-j,2}),j=0,1,2,\\
                \partial_tu^{B,2}_2&\in L^\infty(0,T; \mathcal{H}^{k,m_1-3,2}),\;\partial_tu^{B,2}_2\in L^\infty(0,T; \mathcal{H}^{k,m_1-4,3}),\\
                u^{B,2}_2&\in L^\infty(0,T; \mathcal{H}^{k,m_1-1-i,i+1}),\;i=1,2,3,4.
                \end{split}
			\end{equation}
		}
	\end{lemma}
	\noindent{\bf{Proof.}} Let $\varrho=0$ and $\mu=-\overline{\partial_yu^0_1}$.
    Applying Proposition \ref{regularity of g} to $(\ref{eq:uB11})$, we need prove $\mu$ satisfy the assumption of Proposition \ref{regularity of g} with $m_0=m_1$. By Theorem \ref{th2.1}, we have
        \begin{equation*}
				\begin{split}{}
                \|\mu\|_{L^2_TH^{m_1}_x}+\|\partial_t\mu\|_{L^2_TH^{m_1-1}_x}+\|\partial^2_t\mu\|_{L^2_TH^{m_1-2}_x}\leq C\left(\|u^0_1\|_{L^2_TH^{m_1+2}_{xy}}+\|\partial_tu^0_1\|_{L^2_TH^{m_1+1}_{xy}}+\|\partial^2_tu^0_1\|_{L^2_TH^{m_1}_{xy}}\right)\leq C,
                \end{split}
			\end{equation*}
          where we have used the Sobolev embedding inequality and the trace theorem for any $f(t,x,y)$, for $\iota\in N$ that is
         \begin{align}\label{boundary regularity}
          \|f(t,x,0)\|_{L^2_{T}H^\iota_{x}}\leq C\|f\|_{L^2_{T}H^{\iota+1}_{xy}},
         \end{align}
         which implies
          \begin{align*}
          \|\partial^j_t\overline{\partial_yu^{0}_1}\|_{L^2_TH^{m_1-j}_x}\leq C\|\partial^j_t{u^{0}_1}\|_{L^2_TH^{m_1+2-j}_{xy}},\;j=0,1,2.
        \end{align*}
        By $(\ref{gBt})$, we get the desired estimates for $(\ref{uB11t2})$ and $(\ref{uB11z2t})$. And (\ref{uB14z}) is valid by (\ref{gBz}). Then $(\ref{uB22})$ follows directly from $(\ref{eq:uB22})$, $(\ref{uB11t2})$-$(\ref{uB14z})$. The proof is finished.

	\begin{lemma}\label{lem5.2}
		{Suppose $(n^0,v^0,u^0)$ and $u^{B,1}_1$ be as obtained in Theorem \ref{th2.1} and Lemma $\ref{lem5.1}$ respectively. Then there exists a unique solution $v^{B,0}_2$ to the equations $(\ref{eq:vB02})$ such that
			\begin{equation}\label{vB02t}
				\begin{split}{}
				\partial^j_tv^{B,0}_2&\in L^\infty(0,T; \mathcal{H}^{k,m_1-j,0})\cap L^2(0,T; \mathcal{H}^{k,m_1-j,1}),\;j=0,1,2,\\
                \end{split}
			\end{equation}
            and
        \begin{equation}\label{vB0t2z}
				\begin{split}{}
                \partial_tv^{B,0}_2\in L^\infty(0,T; \mathcal{H}^{k,m_1-2,1}),\;\partial_tv^{B,0}_2\in L^\infty(0,T; \mathcal{H}^{k,m_1-3,2}).\\
                \end{split}
			\end{equation}
           Furthermore, we have the following results
            \begin{equation}\label{vB02z4}
				\begin{split}{}
				v^{B,0}_2\in L^\infty(0,T; \mathcal{H}^{k,m_1-i,i}),\;\;i=1,2,3,4,
                 \end{split}
			\end{equation}
            and that for $(\ref{eq:nB1})$ and $(\ref{eq:pB2})$, we have the following
            \begin{equation*}
				\begin{split}{}
             &\partial^j_tn^{B,1}\in L^\infty(0,T; \mathcal{H}^{k,m_1-j,1})\cap L^2(0,T; \mathcal{H}^{k,m_1-j,2}),\;j=0,1,2,\\
            &\partial_tn^{B,1}\in L^\infty(0,T; \mathcal{H}^{k,m_1-2,2}),\;\partial_tn^{B,1}\in L^\infty(0,T; \mathcal{H}^{k,m_1-3,3}),
                n^{B,1}\in L^\infty(0,T; \mathcal{H}^{k,m_1-i,i+1}),\\
            &\partial^j_tp^{B,2}\in L^\infty(0,T; \mathcal{H}^{k,m_1-j,2})\cap L^2(0,T; \mathcal{H}^{k,m_1-j,3}),\;j=0,1,2,\\
            &\partial_tp^{B,2}\in L^\infty(0,T; \mathcal{H}^{k,m_1-2,3}),\partial_tp^{B,2}\in L^\infty(0,T; \mathcal{H}^{k,m_1-3,4}),p^{B,2}\in L^\infty(0,T; \mathcal{H}^{k,m_1-i,i+2}),\;\;i=1,2,3,4.
                \end{split}
			\end{equation*}
		}
	\end{lemma}
	\noindent{\bf{Proof.}} We proceed to prove $(\ref{vB02t})$. Let $\varphi(z)$ be a smooth function defined on $[0,\infty)$ satisfying
    \begin{equation}\label{varphi}
				\begin{split}{}
				\varphi(0)=1,\;\;
                \varphi(z)=0\;\;for\;\;z>1.
                \end{split}
			\end{equation}
            Denote $\widetilde{v}^{B,0}_2(t,x,z)=v^{B,0}_2(t,x,z)+\varphi(z)\overline{v^0_2}$ and use the curl condition $\partial_z{v}^{B,1}_1=\partial_x{v}^{B,0}_2$, then one deduces from the equations $(\ref{eq:vB02})$ that
            \begin{eqnarray}\label{widetilde{v}^{B,0}2}
		\left\{
		\begin{split}{}			&\partial_t\widetilde{v}^{B,0}_2+\overline{\partial_yu^0_2}\widetilde{v}^{B,0}_2+\partial_x\widetilde{v}^{B,0}_2\overline{u^0_1}+\partial_z\widetilde{v}^{B,0}_2z\overline{\partial_yu^0_2}+\overline{n^0}\widetilde{v}^{B,0}_2=\partial^2_z\widetilde{v}^{B,0}_2+\rho\\
        &{\widetilde{v}^{B,0}_2|_{z=0}=0},\\
        &\widetilde{v}^{B,0}_2(0,x,z)=0.
		\end{split}
		\right.
	\end{eqnarray}
    where $$\rho(t,x,z)=\partial_t\overline{v^0_2}\varphi(z)+\overline{\partial_yu^0_2}\overline{v^0_2}\varphi(z)+\partial_x\overline{v^0_2}\overline{u^0_1}\varphi(z)+\overline{v^0_2}\partial_z\varphi(z)z\overline{\partial_yu^0_2}+\overline{n^0}\overline{v^0_2}\varphi(z)+\overline{v^0_2}\partial^2_z\varphi(z)-\overline{v^0_1}\partial_zu^{B,1}_1.$$ The compatibility condition $v_{02}(x,0)=0$ from conditions $(M_2)$ has been used to determine the initial data of $\widetilde{v}^{B,0}_2$ in $(\ref{widetilde{v}^{B,0}2})$.
         Next, we prove that $\rho$ satisfies the assumption in Proposition $\ref{regularity of f^B}$ with $m_0=m_1$. Applying Theorem \ref{th2.1} and Lemma $\ref{lem5.1}$, for $j=0,1,2$, we deduce that
         \begin{align*}
			&\|(1+z^{2(k+m_1-j)})^\frac{1}{2}\partial_t^j\rho\|_{L^2_TH^{m_1-j}_xL^2_z}\notag\\
            &\leq C\|\partial^{j+1}_t{v^0_2}\|^2_{L^2_TH^{m_1-j}_{xy}}\|(1+z^{2(k+m_1-j)})^\frac{1}{2}\varphi(z)\|_{L^2_z}\\
            &\quad+C\sum^j_{\tau=0}\left(\|\partial_t^\tau u^0_2\|_{L^\infty_TH^{m_1+3-\tau}_{xy}}\|\partial_t^{j-\tau}{v^0_2}\|_{L^2_TH^{m_1+1-(j-\tau)}_{xy}}\right)\|(1+z^{2(k+m_1-j)})^\frac{1}{2}\varphi(z)\|_{L^2_z}\notag\\
            &\quad+C\sum^j_{\tau=0}\left(\|\partial_t^\tau{u^0_1}\|_{L^\infty_TH^{m_1+2-\tau}_{xy}}\|\partial_t^{j-\tau}{v^0_2}\|_{L^2_TH^{m_1+2-(j-\tau)}_{xy}}\right)\|(1+z^{2(k+m_1-j)})^\frac{1}{2}\varphi(z)\|_{L^2_z}\\
            &\quad+C\sum^j_{\tau=0}\left(\|\partial_t^\tau u^0_2\|_{L^\infty_TH^{m_1+3-\tau}_{xy}}\|\partial_t^{j-\tau}{v^0_2}\|_{L^2_TH^{m_1+1-(j-\tau)}_{xy}}\right)\|(z^2+z^{2(k+m_1-j+1)})^\frac{1}{2}\partial_z\varphi(z)\|_{L^2_z}\notag\\
            &\quad+C\sum^j_{\tau=0}\left(\|\partial_t^\tau{n^0}\|_{L^\infty_TH^{m_1+2-\tau}_{xy}}\|\partial_t^{j-\tau}{v^0_2}\|_{L^2_TH^{m_1+1-(j-\tau)}_{xy}}\right)\|(1+z^{2(k+m_1-j)})^\frac{1}{2}\varphi(z)\|_{L^2_z}\\
            &\quad+C\|\partial_t^j{v^0_2}\|_{L^2_TH^{m_1+1-j}_{xy}}\|(1+z^{2(k+m_1-j)})^\frac{1}{2}\partial^2_z\varphi(z)\|_{L^2_z}\\
            &\quad+C\sum^j_{\tau=0}\left(\|\partial_t^\tau{v^0_1}\|_{L^\infty_TH^{m_1+2-\tau}_{xy}}\|\partial_t^{j-\tau} u^{B,1}_1\|_{L^2_TH^{(k+m_1-(j-\tau)),m_1-({j-\tau}),1}}\right)\\
            &\leq C.
		\end{align*}
        By (\ref{fBt}), it then allows us to deduce that
         \begin{equation*}
				\begin{split}{}
				\partial^j_t\widetilde{v}^{B,0}_2\in L^\infty(0,T; \mathcal{H}^{k,m_1-j,0})\cap L^2(0,T; \mathcal{H}^{k,m_1-j,1}),\;j=0,1,2,\\
                \end{split}
			\end{equation*}
        which implies (\ref{vB02t}) is valid by the definition of $\widetilde{v}^{B,0}_2$ and Theorem \ref{th2.1}.
         Moreover, we have the following estimates for $\|(1+z^{2k})^\frac{1}{2}\partial_t\rho\|_{L^\infty_TH^{m_1-3}_xL^2_z}$,
         \begin{align*}
          &\|(1+z^{2k})^\frac{1}{2}\partial_t\rho\|_{L^\infty_TH^{m_1-3}_xL^2_z}\notag\\
            &\leq C\|\partial^{2}_t{v^0_2}\|_{L^\infty_TH^{m_1-2}_{xy}}\|(1+z^{2k})^\frac{1}{2}\varphi(z)\|_{L^2_z}\\
            &\quad+C\left(\|\partial_t u^0_2\|_{L^\infty_TH^{m_1-1}_{xy}}\|{v^0_2}\|_{L^\infty_TH^{m_1-1}_{xy}}+\| u^0_2\|_{L^\infty_TH^{m_1}_{xy}}\|\partial_t {v^0_2}\|_{L^\infty_TH^{m_1-2}_{xy}}\right)\|(1+z^{2k})^\frac{1}{2}\varphi(z)\|_{L^2_z}\\
            &\quad+C\left(\|\partial_t{u^0_1}\|_{L^\infty_TH^{m_1-2}_{xy}}\| v^0_2\|_{L^\infty_TH^{m_1}_{xy}}+\|{u^0_1}\|_{L^\infty_TH^{m_1-1}_{xy}}\|\partial_t v^0_2\|_{L^\infty_TH^{m_1-1}_{xy}}\right)\|(1+z^{2k})^\frac{1}{2}\varphi(z)\|_{L^2_z}\notag\\
            &\quad+C\left(\|\partial_tu^0_2\|_{L^\infty_TH^{m_1-1}_{xy}}\|{v^0_2}\|_{L^\infty_TH^{m_1-1}_{xy}}+\|u^0_2\|_{L^\infty_TH^{m_1}_{xy}}\|\partial_t{v^0_2}\|_{L^\infty_TH^{m_1-2}_{xy}}\right)\|(z^2+z^{2(k+1)})^\frac{1}{2}\partial_z\varphi(z)\|_{L^2_z}\notag\\
            &\quad+C\left(\|\partial_t{n^0}\|_{L^\infty_TH^{m_1-2}_{xy}}\|{v^0_2}\|_{L^\infty_TH^{m_1-1}_{xy}}+\|{n^0}\|_{L^\infty_TH^{m_1-1}_{xy}}\|\partial_t{v^0_2}\|_{L^\infty_TH^{m_1-2}_{xy}}\right)\|(1+z^{2k})^\frac{1}{2}\varphi(z)\|_{L^2_z}\\
            &\quad+C\|\partial_t{v^0_2}\|_{L^\infty_TH^{m_1-2}_{xy}}\|(1+z^{2k})^\frac{1}{2}\partial^2_z\varphi(z)\|_{L^2_z}+C\|\partial_t{v^0_1}\|_{L^\infty_TH^{m_1-1}_{xy}}\|u^{B,1}_1\|_{L^\infty_T\mathcal{H}^{k,m_1-3,1}}\\
            &\quad+C\|{v^0_1}\|_{L^\infty_TH^{m_1-1}_{xy}}\|\partial_tu^{B,1}_1\|_{L^\infty_T\mathcal{H}^{k,m_1-3,1}}\\
            &\leq C.
         \end{align*}
         due to $(\ref{boundary regularity})$. Then similar to (\ref{fz2t}), we can get
        \begin{align}\label{widetilde{v}^{B,0}_2}
			&\|\partial_t \widetilde{v}^{B,0}_2\|_{k,m_1-3,2}\leq C,
	\end{align}
    which implies (\ref{vB0t2z}) is valid.
    Moreover, from (\ref{vB0t2z}), we have
    \begin{align}\label{rho Linfty}
        &\|\rho\|_{L^\infty_T\mathcal{H}^{k,m_1-2,0}}\notag\\
        &\leq C\left(\|{\partial_tv^0_2}\|_{L^\infty_TH^{m_1-1}_{xy}}+\|u^0_2\|_{L^\infty_TH^{m_1+1}_{xy}}\|{v^0_2}\|_{L^\infty_TH^{m_1-1}_{xy}}+\| v^0_2\|_{L^\infty_TH^{m_1+1}_{xy}}\|{u^0_1}\|_{L^\infty_TH^{m_1-1}_{xy}}\right)\|(1+z^{2k})^\frac{1}{2}\varphi(z)\|_{L^2_z}\notag\\
            &\quad+C\|{v^0_2}\|_{L^\infty_TH^{m_1-1}_{xy}}\|u^0_2\|_{L^\infty_TH^{m_1+1}_{xy}}\|(z^2+z^{2(k+1)})^\frac{1}{2}\partial_z\varphi(z)\|_{L^2_z}\\
            &\quad+C\|{n^0}\|_{L^\infty_TH^{m_1}_{xy}}\|{v^0_2}\|_{L^\infty_TH^{m_1-1}_{xy}}\|(1+z^{2k})^\frac{1}{2}\varphi(z)\|_{L^2_z}\notag\\
            &\quad+C\|{v^0_2}\|_{L^\infty_TH^{m_1-1}_{xy}}\|(1+z^{2k})^\frac{1}{2}\partial^2_z\varphi(z)\|_{L^2_z}+C\|{v^0_1}\|_{L^\infty_TH^{m_1+1}_{xy}}\|u^{B,1}_1\|_{L^\infty_T\mathcal{H}^{k,m_1-2,1}}\notag\\
            &\leq C.\notag\notag
    \end{align}
    A similar argument gives $\|\rho\|_{L^\infty_T\mathcal{H}^{k,m_1-3,1}}\leq C$ and $\|\rho\|_{L^\infty_T\mathcal{H}^{k,m_1-4,2}}\leq C$. Therefore, according to (\ref{fBz}), for the equation $(\ref{eq:vB02})_1$, we deduce for any $k\in N$ that (\ref{vB02z4}) is proved. The estimates for
       $n^{B,1}$ and $p^{B,2}$ follows directly from $(\ref{eq:nB1})$, $(\ref{eq:pB2})$ and the conclusion on $\overline{n^0}$ in Theorem \ref{th2.1}. The proof of Lemma $\ref{lem5.2}$ is complete.

      \begin{lemma}\label{lem5.3}
	 	{Let $(v^0,u^0)$, $u^{B,1}_1$ and $v^{B,0}_2$ be as obtained in Theorem \ref{th2.1}, Lemma $\ref{lem5.1}$ and Lemma $\ref{lem5.2}$
        respectively. Then the system $(\ref{eq:vB11})$ admits a unique solution $v^{B,1}_1$ satisfying the following regularity properties for
	 	\begin{equation}\label{vB11t}
                \begin{split}{}
				\partial^j_tv^{B,1}_1&\in L^\infty(0,T; \mathcal{H}^{k,m_1-1-j,0})\cap L^2(0,T; \mathcal{H}^{k,m_1-1-j,1}),\;j=0,1,2.\\
                \end{split}
		  \end{equation}
            Furthermore, it follows from $(\ref{eq:vB11})$ that
           \begin{equation}\label{vB11z1t z2t}
                \begin{split}{}
                \partial_tv^{B,1}_1\in L^\infty(0,T; \mathcal{H}^{k,m_1-3,1}),\;\partial_tv^{B,1}_1\in L^\infty(0,T; \mathcal{H}^{k,m_1-4,2}),\\
                \end{split}
			\end{equation}
            and
             \begin{equation}\label{vB11z4}
                \begin{split}{}
                v^{B,1}_1\in L^\infty(0,T; \mathcal{H}^{k,m_1-1-i,i}),\;i=1,2,3,4.\\
                \end{split}
			\end{equation}
	 	}
	 \end{lemma}
     \noindent{\bf{Proof.}}
     For the system (\ref{eq:vB11}), we suppose
     \begin{align*}
			\varrho&=-z\overline{\partial_x\partial_yu^0_2}v^{B,0}_2 -{\partial_xu^{B,1}_1\overline{v^0_1}}-\overline{\partial_xv^0_1}u^{B,1}_1+\partial_xn^{B,1},
		\end{align*}
          and $\mu=-\overline{\partial_yv^0_1}$. Using Theorem \ref{th2.1}, Lemma $\ref{lem5.1}$-Lemma $\ref{lem5.2}$ and Sobolev embedding inequality, we prove that $\varrho$ and $\mu$ satisfy assumptions in Proposition $\ref{regularity of g}$ with $m_0=m_1-1$. One deduces as follows
        \begin{align*}
			&\|\partial^j_t\varrho\|^2_{L^2_T\mathcal{H}^{k+m_1-1-j,m_1-1-j,0}}\\
            &\leq C\sum^j_{\tau=0}\Big(\|\partial^{j-\tau}_tu^0_2\|^2_{L^\infty_TH^{m_1+3-(j-\tau)}_{xy}}\|\partial^\tau_tv^{B,0}_2\|^2_{L^2_T\mathcal{H}^{k+m_1-\tau,m_1-1-\tau,0}}\Big)\\
            &\quad+C\sum^j_{\tau=0}\Big(\|\partial^{j-\tau}_tv^{0}_1\|^2_{L^\infty_TH^{m_1+1-(j-\tau)}_{xy}}\|\partial^\tau_tu^{B,1}_1\|^2_{L^2_T\mathcal{H}^{k+m_1-1-\tau,m_1-\tau,0}}\Big)\notag\\
            &\quad+C\sum^j_{\tau=0}\Big(\|\partial^{j-\tau}_tv^{0}_1\|^2_{L^\infty_TH^{m_1+2-(j-\tau)}_{xy}}\|\partial^\tau_tu^{B,1}_1\|^2_{L^2_T\mathcal{H}^{k+m_1-1-\tau,m_1-1-\tau,0}}+\|\partial^{j}_tn^{B,1}\|^2_{L^2_T\mathcal{H}^{k+m_1-1-j,m_1-j,0}}\Big)\\
            &\leq C
		\end{align*}
        and
        \begin{align*}
            \|\mu\|_{L^2_TH^{m_1-1}_x}+\|\partial_t\mu\|_{L^2_TH^{m_1-2}_x}+\|\partial^2_t\mu\|_{L^2_TH^{m_1-3}_x}\leq C\left(\|v^0_1\|^2_{L^2_TH^{m_1+1}_{xy}}+\|\partial_tv^0_1\|^2_{L^2_TH^{m_1}_{xy}}+\|\partial^2_tv^0_1\|^2_{L^2_TH^{m_1-1}_{xy}}\right)\leq C.
        \end{align*}
        due to $(\ref{boundary regularity})$. Therefore $(\ref{vB11t})$ is immediately available. Moreover, we estimate $\|(1+z^{2k})^\frac{1}{2}\partial_t\varrho\|^2_{L^\infty_TH^{m_1-4}L^2_{z}}$ as follows
        \begin{align*}
			&\|\partial_t\varrho\|^2_{L^\infty_T\mathcal{H}^{k,m_1-4,0}}\\
            &\leq C\left(\|\partial_tu^0_2\|^2_{L^\infty_TH^{m_1}_{xy}}\|v^{B,0}_2\|^2_{L^\infty_T\mathcal{H}^{k+m_1-1,m_1-4,0}}+\|u^0_2\|^2_{L^\infty_TH^{m_1}_{xy}}\|\partial_tv^{B,0}_2\|^2_{L^\infty_T\mathcal{H}^{k+1,m_1-4,0}}\right)\\
            &\quad+C\left(\|\partial_tv^{0}_1\|^2_{L^\infty_TH^{m_1-1}_{xy}}\|u^{B,1}_1\|^2_{L^\infty_T\mathcal{H}^{k,m_1-3,0}}+\|v^{0}_1\|^2_{L^\infty_TH^{m_1-1}_{xy}}\|\partial_tu^{B,1}_1\|^2_{L^\infty_T\mathcal{H}^{k,m_1-4,0}}\right)\\
            &\quad+C\left(\|\partial_tv^{0}_1\|^2_{L^\infty_TH^{m_1-1}_{xy}}\|u^{B,1}_1\|^2_{L^\infty_T\mathcal{H}^{k,m_1-4,0}}+\|v^{0}_1\|^2_{L^\infty_TH^{m_1}_{xy}}\|\partial_tu^{B,1}_1\|^2_{L^\infty_T\mathcal{H}^{k,m_1-4,0}}+\|\partial_tn^{B,1}\|^2_{L^\infty_T\mathcal{H}^{k,m_1-3,0}}\right)\\
            &\leq C.
		\end{align*}
        Then by $(\ref{gBt})$, we obtain
        \begin{align*}
            \|\partial_t {v}^{B,1}_1\|_{k,m_1-4,2}\leq C.
        \end{align*}
         From the argument similar to the derivation $(\ref{rho Linfty})$, we obtain $\|\varrho\|_{L^\infty H^{k,m_1-1-l,l-2}}\leq C$ for $l=2,3,4$. Combining $(\ref{gBz})$, it can be seen that $(\ref{vB11z4})$ holds. The proof of Lemma $\ref{lem5.3}$ is complete.

        By Lemma $\ref{lem5.1}$-Lemma $\ref{lem5.3}$, next we can estimate the well-posedness of the solutions $v^{B,1}_2$, $v^{B,2}_1$ and $u^{B,2}_1$ for equations $(\ref{eq:vB12})$, $(\ref{eq:vB21})$ and $(\ref{eq:uB21})$, respectively.
        Before calculating the well-posedness of $v^{B,1}_2$ and $v^{B,2}_1$, we first need the regularity {estimates} of $u^{B,2}_1$.
         \begin{lemma}\label{lem5.4}
	    {Let $(n^0,v^0,u^0)$ be the solution obtained in Theorem \ref{th2.1},  $n^{B,1}, v^{B,0}_2$ and $u^{B,1}_1$ be as derived in Lemma $\ref{lem5.1}$. Then there exists a unique solution $u^{B,2}_1$ to the equations $(\ref{eq:uB21})$ on $[0,T]$ such that
	 		\begin{equation}\label{uB21t}
                \begin{split}{}
				\partial^j_tu^{B,2}_1&\in L^\infty(0,T; \mathcal{H}^{k,m_1-2-j,0})\cap L^2(0,T; \mathcal{H}^{k,m_1-2-j,1}),\;j=0,1,2.\\
                \end{split}
		  \end{equation}
            Furthermore, we also have that
           \begin{equation}\label{uB21z1t z2t}
                \begin{split}{}
                \partial_tu^{B,2}_1\in L^\infty(0,T; \mathcal{H}^{k,m_1-4,1}),\;\partial_tu^{B,2}_1\in L^\infty(0,T; \mathcal{H}^{k,m_1-5,2}),\\
                \end{split}
			\end{equation}
            and that
            \begin{equation}\label{uB21z4}
                \begin{split}{}
               u^{B,2}_1\in L^\infty(0,T; \mathcal{H}^{k,m_1-2-i,i}),\;i=1,2,3,4,\\
                \end{split}
			\end{equation}
            as well as
            \begin{equation}\label{uB32}
                     \begin{split}{}
				\partial^j_tu^{B,3}_2&\in L^\infty(0,T; \mathcal{H}^{k,m_1-3-j,1})\cap L^2(0,T; \mathcal{H}^{k,m_1-3-j,2}),j=0,1,2,\\
                \partial_tu^{B,3}_2&\in L^\infty(0,T; \mathcal{H}^{k,m_1-5,2}),\;\partial_tu^{B,3}_2\in L^\infty(0,T; \mathcal{H}^{k,m_1-6,3}),\\
                u^{B,3}_2&\in L^\infty(0,T; \mathcal{H}^{k,m_1-3-i,i+1}),\;i=1,2,3,4.
                    \end{split}
	 		\end{equation}
	 	}
	 \end{lemma}
	 \noindent{\bf{Proof.}} Let
             \begin{equation*}
		\begin{split}{}			{\varrho}&=\Big(u^{B,1}_1z\overline{\partial_x\partial_yu^0_1}+u^{B,2}_2\overline{\partial_yu^0_1}+z\overline{\partial_yu^0_1}\partial_xu^{B,1}_1+\frac{1}{2}z^2\overline{\partial^2_yu^0_2}\partial_zu^{B,1}_1\\
        &\quad+\Big({u^{B,1}_1\partial_xu^{B,1}_1+\overline{u^{I,2}_2}\partial_zu^{B,1}_1+u^{B,2}_2\partial_zu^{B,1}_1}\Big)
        =\varrho_1+\varrho_2
		\end{split}
	\end{equation*}
    and $\mu=0$.
      To apply Proposition \ref{regularity of g} to the equations $(\ref{eq:uB21})$, we prove that $\varrho$ satisfies the assumptions of Proposition \ref{regularity of g} at $m_0=m_1-2$.
    Before obtaining the estimate of $\varrho$, we first need the following conclusion of $\overline{u^{I,2}_2}$. By the boundary expansion $(\ref{layer boundary})$ in Section 3, it is evident that
            \begin{equation*}
                     \begin{split}{}
               \overline{u^{I,2}_2}=-u^{B,2}_2(t,x,0).
                    \end{split}
	 	\end{equation*}
          First note that for $f(t,x,z)\in H^{\iota+1}_xH^1_z$ with fixed $t>0$ the following holds
          \begin{align}\label{boundary L4 regularity}
          \|\partial^\iota_xf(t,x,0)\|_{L^4_x}&\leq C\|f(t,x,z)\|_{H^{\iota+1}_xH^1_z},
        \end{align}
        and
        if $f(t,x,z)\in H^{\iota+2}_{x}H^{2}_{z}$, we have
        \begin{align}\label{boundary L infty regularity}
          \|\partial^\iota_xf(t,x,0)\|_{L^\infty_x}&\leq C\|f(t,x,z)\|_{H^{\iota+2}_xH^2_z}.
        \end{align}
          Then by Lemma $\ref{lem5.1}$ and the Cauchy-Schwarz inequality, for $j=0,1,2$ we can obtain
           \begin{equation}\label{u2I2 regularity}
                     \begin{split}{}  \|\partial^{m_1-2}_x\overline{u^{I,2}_2}\|_{L^4_x}&\leq\|\partial^{m_1-2}_xu^{B,2}_2(t,x,0)\|_{H^1_x}\leq C\|u^{B,2}_2\|_{H^{m_1-1}_xH^1_z},\\
                     \|\partial^{m_1-3}_x\partial_t\overline{u^{I,2}_2}\|_{L^4_x}&\leq\|\partial^{m_1-3}_x\partial_tu^{B,2}_2(t,x,0)\|_{H^1_x}\leq C\|\partial_tu^{B,2}_2\|_{H^{m_1-2}_xH^1_z},\\
                     \|\partial^{m_1-4}_x\partial_t\overline{u^{I,2}_2}\|_{L^\infty_x}&\leq C\|\partial_tu^{B,2}_2\|_{H^{m_1-2}_xH^2_z},\quad
                     \|\partial^{m_1-4}_x\partial^2_t\overline{u^{I,2}_2}\|_{L^2_x}\leq C\|\partial^2_tu^{B,2}_2\|_{H^{m_1-4}_xH^1_z}.\\
                    \end{split}
	 	\end{equation}
        Next, we can infer that the estimate of $\partial^j_t\varrho_1$ is as follows
    \begin{align*}
			&\|\partial^j_t\varrho_1\|_{L^2_T\mathcal{H}^{k+m_1-2-j,m_1-2-j,0}}\\
            &\leq C\sum^j_{\tau=0}\|\partial^\tau_tu^0_1\|_{L^\infty_TH^{m_1+2-\tau}_{xy}}\|\partial^{j-\tau}_tu^{B,1}_1\|_{L^2_T\mathcal{H}^{k+m_1-1-(j-\tau),m_1-2-{(j-\tau)},0}}\\
            &\quad+C\sum^j_{\tau=0}\|\partial^\tau_tu^0_1\|_{L^\infty_TH^{m_1+1-\tau}_{xy}}\|\partial^{j-\tau}_tu^{B,2}_2\|_{L^2_T\mathcal{H}^{k+m_1-2-(j-\tau),m_1-2-(j-\tau),0}}\\
            &\quad+C\sum^j_{\tau=0}\|\partial^\tau_tu^0_1\|_{L^\infty_TH^{m_1+1-\tau}_{xy}}\|\partial^{j-\tau}_tu^{B,1}_1\|_{L^2_T\mathcal{H}^{k+m_1-1-(j-\tau),m_1-1-(j-\tau),0}}\\
            &\quad+C\sum^j_{\tau=0}\Big(\frac12\|\partial^\tau_tu^0_2\|_{L^\infty_TH^{m_1+2-\tau}_{xy}}\|\partial^{j-\tau}_tu^{B,1}_1\|_{L^2_T\mathcal{H}^{k+m_1-(j-\tau),m_1-2-(j-\tau),1}}+\|\partial^j_tp^{B,2}\|_{k+m_1-2-j,m_1-1-j,0}\Big)\\
            &\leq C,
	\end{align*}
    due to $(\ref{boundary regularity})$. Moreover, $\partial^j_t\varrho_2$ is estimated as follows
    \begin{align*}
			&\|\varrho_2\|_{L^2_T\mathcal{H}^{k+m_1-2,m_1-2,0}}\\
            &\leq C\left(\|u^{B,1}_1\|_{L^\infty_TH^{m_1-1}_xH^1_z}\|u^{B,1}_1\|_{L^2_T\mathcal{H}^{k+m_1-2,m_1,1}}+2\|u^{B,2}_2\|_{L^\infty_TH^{m_1-1}_xH^1_z}\|u^{B,1}_1\|_{L^2_T\mathcal{H}^{k+m_1-2,m_1-1,2}}\right)\\
            &\leq C,\\
			&\|\partial_t\varrho_2\|_{L^2_T\mathcal{H}^{k+m_1-3,m_1-3,0}}\\
            &\leq C\left(\|\partial_tu^{B,1}_1\|_{L^\infty_TH^{m_1-2}_xH^1_z}\|u^{B,1}_1\|_{L^2_T\mathcal{H}^{k+m_1-3,m_1-1,1}}+\|u^{B,1}_1\|_{L^\infty_TH^{m_1-2}_xH^1_z}\|\partial_tu^{B,1}_1\|_{L^2_T\mathcal{H}^{k+m_1-3,m_1-1,1}}\right)\\
           &\quad+C\left(\|\partial_tu^{B,2}_2\|_{L^\infty_TH^{m_1-2}_xH^1_z}\|u^{B,1}_1\|_{L^2_T\mathcal{H}^{k+m_1-3,m_1-2,2}}+2\|u^{B,2}_2\|_{L^\infty_TH^{m_1-2}_xH^1_z}\|\partial_tu^{B,1}_1\|_{L^2_T\mathcal{H}^{k+m_1-3,m_1-2,2}}\right)\\
            &\leq C,\\
			&\|\partial^2_t\varrho_2\|_{L^2_T\mathcal{H}^{k+m_1-4,m_1-4,0}}\\
            &\leq C\left(\|\partial^2_tu^{B,1}_1\|_{L^\infty_TH^{m_1-4}_xL^2_z}\|u^{B,1}_1\|_{L^2_T\mathcal{H}^{k+m_1-4,m_1-1,2}}+2\|\partial_tu^{B,1}_1\|_{L^\infty_TH^{m_1-3}_xH^1_z}\|\partial_tu^{B,1}_1\|_{L^2_T\mathcal{H}^{k+m_1-4,m_1-2,1}}\right)\\  &\quad+C\left(\|u^{B,1}_1\|_{L^\infty_TH^{m_1-2}_xH^2_z}\|\partial^2_tu^{B,1}_1\|_{L^2_T\mathcal{H}^{k+m_1-4,m_1-3,0}}+\|\partial^2_tu^{B,2}_2\|_{L^\infty_TH^{m_1-4}_{x}H^1_z}\|u^{B,1}_1\|_{L^2_T\mathcal{H}^{k+m_1-4,m_1-2,3}}\right)\\
            &\quad+C\left(\|\partial_tu^{B,2}_2\|_{L^2_TH^{m_1-2}_{x}H^2_z}\|\partial_tu^{B,1}_1\|_{L^\infty_T\mathcal{H}^{k+m_1-4,m_1-4,1}}+\|u^{B,2}_2\|_{L^\infty_TH^{m_1-2}_{x}H^2_z}\|\partial^2_tu^{B,1}_1\|_{L^2_T\mathcal{H}^{k+m_1-4,m_1-4,1}}\right)\\
            &\quad+C\|\partial^2_tu^{B,2}_2\|_{L^\infty_TH^{m_1-3}_xH^1_z}\|u^{B,1}_1\|_{L^2_T\mathcal{H}^{k+m_1-4,m_1-3,2}}\\           &\quad+C\left(\|\partial_tu^{B,2}_2\|_{L^\infty_TH^{m_1-3}_xH^1_z}\|\partial_tu^{B,1}_1\|_{L^2_T\mathcal{H}^{k+m_1-4,m_1-3,2}}+\|u^{B,2}_2\|_{L^\infty_TH^{m_1-2}_xH^2_z}\|\partial^2_tu^{B,1}_1\|_{L^2_T\mathcal{H}^{k+m_1-4,m_1-4,1}}\right)\\
            &\leq C,
	\end{align*}
    where we have used (\ref{u2I2 regularity}) and the following embedding inequalities
        \begin{align}\label{fgHm}
            &\|f(t,x,z)g(t,x,z)\|_{H^\iota_xL^2_z}\notag\\
            &\leq\sum_{\alpha\leq \iota}\Big(\|\partial_x^\alpha f(t,x,z)\|_{L^4_{xz}}\|g(t,x,z)\|_{L^4_{xz}}+\sum_{\beta\leq \alpha}C^\beta_\alpha\|\partial_x^\beta f(t,x,z)\|_{L^4_{xz}}\|\partial_x^{\alpha-\beta} g(t,x,z)\|_{L^4_{xz}}\notag\\
            &\quad+\|f(t,x,z)\|_{L^4_{xz}}\|\partial_x^\alpha g(t,x,z)\|_{L^4_{xz}}\Big)\\
            &\leq C\|f(t,x,z)\|_{H^{\iota+1}_xH^1_z}\|g(t,x,z)\|_{H^{\iota+1}_xH^1_z},\notag\notag
        \end{align}
        and
        \begin{align}\label{fgHm1}
            &\|f(t,x,z)g(t,x,z)\|_{H^\iota_xL^2_z}\notag\\
            &\leq\sum_{\alpha\leq \iota}\Big(\|\partial_x^\alpha f(t,x,z)\|_{L^2_{xz}}\|g(t,x,z)\|_{L^\infty_{xz}}+\sum_{\beta\leq \alpha}C^\beta_\alpha\|\partial_x^\beta f(t,x,z)\|_{L^2_{xz}}\|\partial_x^{\alpha-\beta} g(t,x,z)\|_{L^\infty_{xz}}\notag\\
            &\quad+\|f(t,x,z)\|_{L^2_{xz}}\|\partial_x^\alpha g(t,x,z)\|_{L^\infty_{xz}}\Big)\\
            &\leq C\|f(t,x,z)\|_{H^{\iota}_xL^2_z}\|g(t,x,z)\|_{H^{\iota+2}_xH^2_z}.\notag\notag
        \end{align}
       Collecting the above estimates for $\varrho_1$ and $\varrho_2$, we deduce that $\|\partial^j_t {\varrho}\|_{L^2_{T}\mathcal{H}^{k+m_1-2-j,m_1-2-j,0}}\leq C$ is valid for $j=0,1,2$.
       By Theorem \ref{th2.1}, Lemma $\ref{lem5.1}$ and (\ref{u2I2 regularity}), we obtain
         \begin{align*}
             \|\partial_t {\varrho}\|_{L^2_{T}\mathcal{H}^{k,m_1-5,0}}\leq C\left(\|\partial_t {\varrho}_1\|_{L^2_{T}\mathcal{H}^{k,m_1-5,0}}+\|\partial_t {\varrho}_2\|_{L^2_{T}\mathcal{H}^{k,m_1-5,0}}\right)\leq C.
         \end{align*}

         By a similar argument as deriving $(\ref{gBt})$-$(\ref{gBz})$, we deduce for all $k\in N$ that $(\ref{uB21t})$ valid. From the argument similar to the derivation $(\ref{rho Linfty})$, we obtain $\|\varrho\|_{L^\infty H^{k,m_1-2-l,l-2}}\leq C$ for $l=2,3,4$. Combining $(\ref{gBz})$, we know that $(\ref{uB21z4})$ holds. Moreover, $(\ref{uB32})$ follows directly from $(\ref{eq:uB32})$, $(\ref{uB21t})$ and $(\ref{uB21z4})$. Thus, the proof of Lemma $\ref{lem5.4}$ is complete.

       \begin{lemma}\label{lem5.5}
		{Let $(n^0,v^0,u^0)$ be the solution obtained in Theorem \ref{th2.1}, $(u^{B,1}_1, u^{B,2}_2)$, $(v^{B,0}_2, n^{B,1})$, $v^{B,1}_1$ and $(u^{B,2}_1,u^{B,3}_2)$ be as obtained in Lemma $\ref{lem5.1}$-Lemma $\ref{lem5.4}$, respectively. Then there exists a unique solution $v^{B,1}_2$ to the equations $(\ref{eq:vB12})$, such that
			\begin{equation}\label{vB12t}
				\begin{split}{}
				\partial^j_tv^{B,1}_2&\in L^\infty(0,T; \mathcal{H}^{k,m_1-2-j,0})\cap L^2(0,T; \mathcal{H}^{k,m_1-2-j,1}),\;j=0,1,2.\\
                \end{split}
			\end{equation}
            Furthermore, it follows from the equations $(\ref{eq:vB12})$
            \begin{equation}\label{vB12z1t z2t}
				\begin{split}{}
                \partial_tv^{B,1}_2\in L^\infty(0,T; \mathcal{H}^{k,m_1-4,1}),\;\partial_tv^{B,1}_2\in L^\infty(0,T; \mathcal{H}^{k,m_1-5,2}),\\
                \end{split}
			\end{equation}
            and that
            \begin{equation}\label{vB12z4}
				\begin{split}{}
               v^{B,1}_2\in L^\infty(0,T; \mathcal{H}^{k,m_1-2-i,i}),\;\;i=1,2,3,4,\\
                \end{split}
			\end{equation}
            For the equations $(\ref{eq:nB2})$, we have the following
            \begin{align}\label{n^B2tz}
                \partial^j_tn^{B,2}&\in L^\infty(0,T; \mathcal{H}^{k,m_1-2-j,1})\cap L^2(0,T; \mathcal{H}^{k,m_1-2-j,2}),\;j=0,1,2,\notag\\
                \partial_tn^{B,2}&\in L^\infty(0,T; \mathcal{H}^{k,m_1-4,2}),\;\partial_tn^{B,2}\in L^\infty(0,T; \mathcal{H}^{k,m_1-5,3}),\\
                n^{B,2}&\in L^\infty(0,T; \mathcal{H}^{k,m_1-2-i,i+1}),\;i=1,2,3,4.\notag\notag
			\end{align}
		}
	\end{lemma}
	\noindent{\bf{Proof.}} Let
    \begin{align*}
        {\rho}&=\Big(-z\overline{\partial^2_yu^{I,0}_2}v^{B,0}_2-\overline{\partial_yu^0_1}v^{B,1}_1-\partial_zu^{B,1}_1z\overline{\partial_yv^0_1}-{\partial_zu^{B,2}_1\overline{v^0_1}}-\partial_zu^{B,2}_2\overline{v^0_2}-\overline{\partial_yv^0_1}u^{B,1}_1\\
        &\quad-\frac{1}{2}\partial_zv^{B,0}_2z^2\overline{\partial_yu^0_2}-\partial_zv^{B,1}_1z\overline{\partial_yu^0_1}-2\partial_zv^{B,0}_2\overline{v^0_2}-\overline{v^0_2}n^{B,1}\Big)\\
        &\quad-\Big(\partial_zu^{B,1}_1v^{B,1}_1+\partial_zu^{B,2}_2v^{B,0}_2+\partial_zv^{B,0}_2\overline{u^{I,2}_2}+\partial_zv^{B,0}_2u^{B,2}_2+\partial_zv^{B,1}_1u^{B,1}_1+2\partial_zv^{B,0}_2v^{B,0}_2+v^{B,0}_2n^{B,1}\Big)\\
        &\quad+\Big(\int^\infty_z\overline{\partial_yn^0}(\eta\partial_\eta v^{B,0}_2+v^{B,0}_2)d\eta\Big)\\
        &=\rho_1+\rho_2+\rho_3.
		\end{align*}
        Applying Lemma \ref{lem5.1}-Lemma \ref{lem5.4} and Proposition \ref{regularity of f^B} for $m_0=m_1-2$, $\|\partial^j_t {\rho}\|^2_{L^2_{T}\mathcal{H}^{k+m_1-2-j,m_1-2-j,0}}$ was estimated as follows
		\begin{align*}
			&\|\partial^j_t {\rho_1}\|^2_{L^2_{T}\mathcal{H}^{k+m_1-2-j,m_1-2-j,0}}\notag\\
            &\leq C\sum^j_{\tau=0}\|\partial_t^\tau u^0_2\|_{L^\infty_TH^{m_1+2-\tau}_{xy}}\|\partial_t^{j-\tau}v^{B,0}_2\|^2_{L^2_{T}\mathcal{H}^{k+m_1-1-(j-\tau),m_1-2-(j-\tau),0}}\notag\\
            &\quad+C\sum^j_{\tau=0}\|\partial_t^\tau u^0_1\|_{L^\infty_TH^{m_1+1-\tau}_{xy}}\|\partial_t^{j-\tau}v^{B,1}_1\|_{L^2_{T}\mathcal{H}^{k+m_1-2-(j-\tau),m_1-2-(j-\tau),0}}\\
            &\quad+C\sum^j_{\tau=0}\|\partial_t^\tau v^0_1\|_{L^\infty_TH^{m_1+1-\tau}_{xy}}\|\partial_t^{j-\tau}u^{B,1}_1\|_{L^2_{T}\mathcal{H}^{k+m_1-1-(j-\tau),m_1-2-(j-\tau),1}}\\
            &\quad+C\sum^j_{\tau=0}\|\partial_t^\tau{v^{0}_1}\|_{L^\infty_TH^{m_1-\tau}_{xy}}\|\partial_t^{j-\tau}u^{B,2}_1\|^2_{L^2_{T}\mathcal{H}^{k+m_1-2-(j-\tau),m_1-2-(j-\tau),1}}\\
            &\quad+C\sum^j_{\tau=0}\|\partial_t^\tau{v^{0}_2}\|_{L^\infty_TH^{m_1-\tau}_{xy}}\|\partial_t^{j-\tau}u^{B,2}_2\|_{L^2_{T}\mathcal{H}^{k+m_1-2-(j-\tau),m_1-2-(j-\tau),1}}\\
            &\quad+C\sum^j_{\tau=0}\|\partial_t^\tau v^{0}_1\|_{L^\infty_TH^{m_1+1-\tau}_{xy}}\|\partial_t^{j-\tau}u^{B,1}_1\|_{L^2_{T}\mathcal{H}^{k+m_1-2-(j-\tau),m_1-2-(j-\tau),0}}\\
            &\quad
            +C\sum^j_{\tau=0}\|\partial_t^\tau u^0_2\|_{L^\infty_TH^{m_1+2-\tau}_{xy}}\|\partial_t^{j-\tau}v^{B,0}_2\|_{L^2_{T}\mathcal{H}^{k+m_1-(j-\tau),m_1-2-(j-\tau),1}}\\
            &\quad+C\sum^j_{\tau=0}\|\partial_t^\tau u^{0}_1\|_{L^\infty_TH^{m_1+1-\tau}_{xy}}\|\partial_t^{j-\tau}v^{B,1}_1\|_{L^2_{T}\mathcal{H}^{k+m_1-1-(j-\tau),m_1-2-(j-\tau),1}}\\
            &\quad+C\sum^j_{\tau=0}\|\partial_t^\tau{v^{0}_2}\|_{L^\infty_TH^{m_1-\tau}_{xy}}\|\partial_t^{j-\tau}v^{B,0}_2\|_{L^2_{T}\mathcal{H}^{k+m_1-2-(j-\tau),m_1-2-(j-\tau),1}}\\
            &\quad+C\sum^j_{\tau=0}\|\partial_t^\tau{v^{0}_2}\|_{L^\infty_TH^{m_1-\tau}_{xy}}\|\partial_t^{j-\tau}n^{B,1}\|_{L^2_{T}\mathcal{H}^{k+m_1-2-(j-\tau),m_1-2-(j-\tau),0}}\\
            &\leq C.
		\end{align*}
        due to $(\ref{boundary regularity})$. Moreover, for $\partial^j_t {\rho_2}$, we first estimate $\partial_t {\rho_2}$ as follows
        \begin{align*}
			&\|\partial_t {\rho_2}\|_{L^2_{T}\mathcal{H}^{k+m_1-3,m_1-3,0}}\notag\\
            &\leq C\left(\|\partial_t u^{B,1}_1\|_{L^\infty_{T}H^{m_1-2}_xH^2_z}\|v^{B,1}_1\|_{L^2_{T}\mathcal{H}^{k+m_1-3,m_1-2,1}}+\|u^{B,1}_1\|_{L^\infty_{T}H^{m_1-2}_xH^2_z}\|\partial_tv^{B,1}_1\|_{L^2_{T}\mathcal{H}^{k+m_1-3,m_1-2,1}}\right)\\
            &\quad+C\left(\|\partial_tu^{B,2}_2\|_{L^2_{T}H^{m_1-2}_xH^2_z}\|v^{B,0}_2\|_{L^\infty_{T}\mathcal{H}^{k+m_1-3,m_1-2,1}}+\|u^{B,2}_2\|_{L^\infty_{T}H^{m_1-2}_xH^2_z}\|\partial_tv^{B,0}_2\|_{L^2_{T}\mathcal{H}^{k+m_1-3,m_1-2,1}}\right)\\
            &\quad+C\left(\|\partial_t{u^{B,2}_2}\|_{L^\infty_TH^{m_1-2}_xH^1_z}\|v^{B,0}_2\|_{L^2_{T}\mathcal{H}^{k+m_1-3,m_1-2,2}}+\|{u^{B,2}_2}\|_{L^\infty_TH^{m_1-2}_xH^1_z}\|\partial_tv^{B,0}_2\|_{L^2_{T}\mathcal{H}^{k+m_1-3,m_1-2,2}}\right)\\
            &\quad+C\left(\|\partial_tv^{B,0}_2\|_{L^2_{T}\mathcal{H}^{k+m_1-3,m_1-2,2}}\|{u^{B,2}_2}\|_{L^\infty_TH^{m_1-2}_xH^1_z}+\|v^{B,0}_2\|_{L^\infty_{T}\mathcal{H}^{k+m_1-3,m_1-2,2}}\|\partial_t{u^{B,2}_2}\|_{L^2_TH^{m_1-2}_xH^1_z}\right)\\
            &\quad+C\left(\|\partial_tv^{B,1}_1\|_{L^2_{T}\mathcal{H}^{k+m_1-3,m_1-3,1}}\| u^{B,1}_1\|_{L^\infty_TH^{m_1-1}_xH^2_z}+\|v^{B,1}_1\|_{L^\infty_{T}\mathcal{H}^{k+m_1-3,m_1-2,2}}\|\partial_tu^{B,1}_1\|_{L^2_TH^{m_1-2}_xH^1_z}\right)\\
            &\quad+C\left(\|\partial_tv^{B,0}_2\|_{L^2_TH^{m_1-2}_xH^2_z}\|v^{B,0}_2\|_{L^\infty_{T}\mathcal{H}^{k+m_1-3,m_1-2,1}}+2\|v^{B,0}_2\|_{L^\infty_TH^{m_1-2}_xH^2_z}\|\partial_tv^{B,0}_2\|_{L^2_{T}\mathcal{H}^{k+m_1-3,m_1-2,1}}\right)\\
            &\quad+C\left(\|\partial_tn^{B,1}\|_{L^2_TH^{m_1-2}_xH^1_z}\|v^{B,0}_2\|_{L^\infty_{T}\mathcal{H}^{k+m_1-3,m_1-2,1}}+\|n^{B,1}\|_{L^\infty_TH^{m_1-2}_xH^1_z}\|\partial_tv^{B,0}_2\|_{L^2_{T}\mathcal{H}^{k+m_1-3,m_1-2,1}}\right)\\
            &\leq C
		\end{align*}
        due to (\ref{u2I2 regularity}), the inequalities (\ref{fgHm}) and (\ref{fgHm1}).
        Similarly, one derives $$\|\rho_2\|_{L^2_T\mathcal{H}^{k+m_1-2,m_1-2,0}}+\|\partial^2_t\rho_2\|_{L^2_T\mathcal{H}^{k+m_1-4,m_1-4,0}}\leq C.$$ Further, it is obvious that using the Cauchy-Schwarz inequality
        \begin{align*}
			\|\partial^j_t {\rho_3}\|_{L^2_{T}\mathcal{H}^{k+m_1-2-j,m_1-2-j,0}}&=\|(1+z^{2(k+m_1-2-j)})^\frac{1}{2}\partial^j_t\Big{(}\overline{\partial_yn^0}\int^\infty_z\partial_z(zv^{B,0}_2)d\eta\Big{)}\|_{L^2_{T}H^{m_1-2-j}_xL^2_{z}}\\
            &=\|(1+z^{2(k+m_1-2-j)})^\frac{1}{2}\partial^j_t(-\overline{\partial_yn^0}zv^{B,0}_2)\|_{L^2_{T}H^{m_1-2-j}_xL^2_{z}}\\
            &\leq C\sum^j_{\tau=0}\|\partial^\tau_t n^{0}\|^2_{L^\infty_{T}H^{m_1+1-\tau}_{xy}}\|(1+z^{2(k+m_1-1-(j-\tau))})\partial^{j-\tau}_tv^{B,0}_2\|^2_{L^2_{T}H^{m_1-2-(j-\tau)}_xL^2_z}\\
            &\leq C.
		\end{align*}
        Collecting the above estimates for $\rho_1$, $\rho_2$, $\rho_3$, we deduce that $\|\partial^j_t {\rho}\|_{L^2_{T}\mathcal{H}^{k+m_1-2-j,m_1-2-j,0}}\leq C$ for $j=0,1,2$. Then by (\ref{fBt}), we conclude that
        \begin{equation*}
				\begin{split}{}
				\partial^j_t{v}^{B,1}_2\in L^\infty(0,T; \mathcal{H}^{k,m_1-2-j,0})\cap L^2(0,T; \mathcal{H}^{k,m_1-2-j,1}),\;j=0,1,2.\\
                \end{split}
			\end{equation*}
         Similar estimates can be $\|\partial_t\rho\|_{L^\infty_T\mathcal{H}^{k,m_1-5,0}}\leq C$, it follows from $(\ref{eq:vB12})$ that $\|\partial_t{v}^{B,1}_2\|_{L^\infty_T\mathcal{H}^{k,m_1-5,2}}\leq C$. Combining the above estimates of ${\rho}$, it implies $(\ref{vB12t})$. From the argument similar to the derivation $(\ref{rho Linfty})$, we obtain $\|\rho\|_{L^\infty H^{k,m_1-2-l,l-2}}\leq C$ for $l=2,3,4$. By $(\ref{fBz})$, we get $(\ref{vB12z4})$ from $(\ref{eq:vB12})$ and $(\ref{vB12t})$. Finally, $(\ref{n^B2tz})$ follows directly from $(\ref{eq:nB2})$, $(\ref{vB12t})$ and $(\ref{vB12z4})$. The proof of Lemma $\ref{lem5.5}$ is complete.

     \begin{lemma}\label{lem5.6}
	 	{Let $(n^0,v^0,u^0)$, $(u^{B,1}_1, u^{B,2}_2)$, $(v^{B,0}_2, n^{B,1})$, $v^{B,1}_1$,  $(u^{B,2}_1,u^{B,3}_2)$ and $(v^{B,1}_2, n^{B,2})$ be as obtained in Theorem \ref{th2.1}, Lemma $\ref{lem5.1}$-Lemma $\ref{lem5.5}$, respectively. Then there exists a unique solution $v^{B,2}_1$ of $(\ref{eq:vB21})$ on $[0,T]$ such that
	 		\begin{equation}\label{vB21t}
                \begin{split}{}
				\partial^j_tv^{B,2}_1&\in L^\infty(0,T; \mathcal{H}^{k,m_1-3-j,0})\cap L^2(0,T; \mathcal{H}^{k,m_1-3-j,1}),\;j=0,1,2.\\
                \end{split}
		  \end{equation}
            Furthermore,it follows from $(\ref{eq:vB21})$ that
           \begin{equation}\label{vB21z4}
                \begin{split}{}
                \partial_tv^{B,2}_1\in L^\infty(0,T; \mathcal{H}^{k,m_1-5,1}),\;\partial_tv^{B,2}_1\in L^\infty(0,T; \mathcal{H}^{k,m_1-6,2}),\\
                \end{split}
			\end{equation}
            and
            \begin{equation}\label{vB21z4}
                \begin{split}{}
                v^{B,2}_1\in L^\infty(0,T; \mathcal{H}^{k,m_1-3-i,i}),\;i=1,2,3,4.\\
                \end{split}
			\end{equation}
	 	}
	 \end{lemma}
	 \noindent{\bf{Proof.}}
             Let
             \begin{align*}
			{\varrho}&=\Big(z\overline{\partial_x\partial_yu^{I,0}_1}v^{B,1}_1+z\overline{\partial_x\partial_{y}u^{I,0}_2}v^{B,1}_2+\partial_x(\frac{1}{2}z^2\overline{\partial^2_yu^{I,0}_2}+\overline{u^{I,2}_2})v^{B,0}_2+\partial_xu^{B,1}_1z\overline{\partial_yv^{I,0}_1}+\partial_xu^{B,2}_1\overline{v^{I,0}_1}\\
            &\quad+\partial_xu^{B,2}_2\overline{v^{I,0}_2}+\overline{\partial_xv^{I,0}_1}u^{B,2}_1+\overline{\partial_xv^{I,0}_2}u^{B,2}_2+z\overline{\partial_x\partial_yv^{I,0}_1}u^{B,1}_1+\partial_xv^{B,0}_2(\frac{1}{2}z^2\overline{\partial^2_yu^{I,0}_2}+\overline{u^{I,2}_2})\\    &\quad+\partial_xv^{B,1}_1z\overline{\partial_yu^{I,0}_1}+\overline{\partial_xv^{I,0}_2}v^{B,0}_2+2\partial_xv^{B,0}_2\overline{v^{I,0}_2}-\partial_xn^{B,2}\Big)\\
            &\quad+\Big(\partial_xu^{B,1}_1v^{B,1}_1+\partial_xu^{B,2}_2v^{B,0}_2+\partial_xv^{B,0}_2u^{B,2}_2+\partial_xv^{B,1}_1u^{B,1}_1+2\partial_xv^{B,0}_2v^{B,0}_2\Big)\\
            &=\varrho_1+\varrho_2,
		\end{align*}
        and $\mu=0$. We will prove that $\varrho$ satisfy the assumptions of Proposition \ref{regularity of g} to (\ref{eq:vB21}) with $m_0=m_1-3$.
        First we divide $\varrho$ into $\varrho_1$ and $\varrho_2$, and each of them be estimated as follows
            \begin{align*}
			&\|\partial_t^j \varrho_1\|^2_{L^2_{T}\mathcal{H}^{k+m_1-3-j,m_1-3-j,0}}\\
            &\leq C\sum^j_{\tau=0}\|\partial_t^\tau u^0_1\|_{L^\infty_TH^{m_1+1-\tau}_{xy}}\|\partial_t^{j-\tau}v^{B,1}_1\|^2_{L^2_{T}\mathcal{H}^{k+m_1-2-(j-\tau),m_1-3-(j-\tau),0}}\\
            &\quad+C\sum^j_{\tau=0}\|\partial_t^\tau u^0_2\|_{L^\infty_TH^{m_1+1-\tau}_{xy}}\|\partial_t^{j-\tau}v^{B,1}_2\|_{L^2_{T}\mathcal{H}^{k+m_1-2-(j-\tau),m_1-3-(j-\tau),0}}\\
            &\quad
            +C\sum^j_{\tau=0}\|\partial_t^\tau u^0_2\|_{L^\infty_TH^{m_1+2-\tau}_{xy}}\|\partial_t^{j-\tau}v^{B,0}_2\|_{L^2_{T}\mathcal{H}^{k+m_1-1-(j-\tau),m_1-3-(j-\tau),0}}\\
            &\quad+C\sum^j_{\tau=0}\|\partial_t^\tau{u^{B,2}_2}\|_{L^\infty_TH^{m_1-1-\tau}_xH^1_z}\|\partial_t^{j-\tau}v^{B,0}_2\|_{L^2_{T}\mathcal{H}^{k+m_1-3-(j-\tau),m_1-2-(j-\tau),1}}\\
            &\quad+C\sum^j_{\tau=0}\|\partial_t^\tau v^0_1\|_{L^\infty_TH^{m_1-\tau}_{xy}}\|\partial_t^{j-\tau}u^{B,1}_1\|^2_{L^2_{T}\mathcal{H}^{k+m_1-2-(j-\tau),m_1-2-(j-\tau),0}}\\
            &\quad+C\sum^j_{\tau=0}\|\partial_t^\tau{v^{0}_1}\|_{L^\infty_TH^{m_1-1-\tau}_{xy}}\|\partial_t^{j-\tau}u^{B,2}_1\|^2_{L^2_{T}\mathcal{H}^{k+m_1-3-(j-\tau),m_1-2-(j-\tau),0}}
            \\
            &\quad+C\sum^j_{\tau=0}\|\partial_t^\tau{v^{0}_2}\|_{L^\infty_TH^{m_1-1-\tau}_{xy}}\|\partial_t^{j-\tau}u^{B,2}_2\|_{L^2_{T}\mathcal{H}^{k+m_1-3-(j-\tau),m_1-2-(j-\tau),0}}\\
            &\quad+C\sum^j_{\tau=0}\|\partial_t^\tau{v^{0}_1}\|_{L^\infty_TH^{m_1-\tau}_{xy}}\|\partial_t^{j-\tau}u^{B,2}_1\|_{L^2_{T}\mathcal{H}^{k+m_1-3-(j-\tau),m_1-3-(j-\tau),0}}
            \\
            &\quad+C\sum^j_{\tau=0}\|\partial_t^\tau{{v^0_2}}\|_{L^\infty_TH^{m_1-\tau}_{xy}}\|\partial_t^{j-\tau}u^{B,2}_2\|_{L^2_{T}\mathcal{H}^{k+m_1-3-(j-\tau),m_1-3-(j-\tau),0}}\\
            &\quad
            +C\sum^j_{\tau=0}\|\partial_t^\tau v^0_1\|_{L^\infty_TH^{m_1+1-\tau}_{xy}}\|\partial_t^{j-\tau}u^{B,1}_1\|_{L^2_{T}\mathcal{H}^{k+m_1-2-(j-\tau),m_1-3-(j-\tau),0}}\\
            &\quad+C\sum^j_{\tau=0}\|\partial_t^\tau u^0_2\|_{L^\infty_TH^{m_1+1-\tau}_{xy}}\|\partial_t^{j-\tau}v^{B,0}_2\|_{L^2_{T}\mathcal{H}^{k+m_1-1-(j-\tau),m_1-2-(j-\tau),0}}\\
            &\quad+C\sum^j_{\tau=0}\|\partial_t^\tau{u^{B,2}_2}\|_{L^\infty_TH^{m_1-2-\tau}_xH^1_z}\|\partial_t^{j-\tau}v^{B,0}_2\|_{L^2_{T}\mathcal{H}^{k+m_1-3-(j-\tau),m_1-1-(j-\tau),1}}\\
            &\quad+C\sum^j_{\tau=0}\|\partial_t^\tau u^{0}_1\|_{L^\infty_TH^{m_1-\tau}_{xy}}\|\partial_t^{j-\tau}v^{B,1}_1\|_{L^2_{T}\mathcal{H}^{k+m_1-2-(j-\tau),m_1-2-(j-\tau),0}}\\
            &\quad+C\sum^j_{\tau=0}\Big(\|\partial_t^\tau{v^{0}_2}\|_{L^\infty_TH^{m_1-\tau}_{xy}}\|\partial_t^{j-\tau}v^{B,0}_2\|_{L^2_{T}\mathcal{H}^{k+m_1-3-(j-\tau),m_1-2-(j-\tau),0}}+\|\partial_t^j n^{B,2}\|_{L^\infty_T\mathcal{H}^{k+m_1-3-j,m_1-2-j,0}}\Big).\\
            &\leq C,
		\end{align*}
        due to $(\ref{boundary regularity})$ and (\ref{u2I2 regularity}). For $\partial_t^j \varrho_2(j=0,1,2)$, we estimate $\partial_t \varrho_2$ as follows
        \begin{align*}
			&\|\partial_t \varrho_2\|^2_{L^2_{T}\mathcal{H}^{k+m_1-4,m_1-4,0}}\\
            &\leq C\left(\|\partial_t u^{B,1}_1\|_{L^2_TH^{m_1-2}_xH^1_z}\|v^{B,1}_1\|_{L^\infty_{T}\mathcal{H}^{k+m_1-4,m_1-3,1}}+\|u^{B,1}_1\|_{L^\infty_TH^{m_1-2}_xH^1_z}\|\partial_tv^{B,1}_1\|_{L^2_{T}\mathcal{H}^{k+m_1-4,m_1-3,1}}\right)\\
            &\quad+C\left(\|\partial_t{u^{B,2}_2}\|_{L^2_TH^{m_1-3}_xL^2_z}\|v^{B,0}_2\|_{L^\infty_{T}\mathcal{H}^{k+m_1-4,m_1-2,2}}+\|{u^{B,2}_2}\|_{L^\infty_TH^{m_1-2}_xH^1_z}\|\partial_tv^{B,0}_2\|_{L^2_{T}\mathcal{H}^{k+m_1-4,m_1-3,1}}\right)\\
            &\quad+C\left(\|\partial_tv^{B,0}_2\|_{L^\infty_{T}\mathcal{H}^{k+m_1-4,m_1-3,0}}\|{u^{B,2}_2}\|_{L^2_TH^{m_1-2}_xH^2_z}+\|v^{B,0}_2\|_{L^\infty_{T}\mathcal{H}^{k+m_1-4,m_1-2,1}}\|\partial_t{u^{B,2}_2}\|_{L^2_TH^{m_1-3}_xH^1_z}\right)\\
            &\quad+C\left(\|\partial_tv^{B,1}_1\|_{L^2_{T}\mathcal{H}^{k+m_1-4,m_1-3,0}}\|u^{B,1}_1\|_{L^\infty_TH^{m_1-2}_xH^2_z}+\|v^{B,1}_1\|_{L^2_{T}\mathcal{H}^{k+m_1-4,m_1-2,1}}\|\partial_tu^{B,1}_1\|_{L^\infty_TH^{m_1-3}_xH^1_z}\right)\\
            &\quad+C\left(\|\partial_tv^{B,0}_2\|_{L^\infty_TH^{m_1-3}_xL^2_z}\|v^{B,0}_2\|_{L^2_{T}\mathcal{H}^{k+m_1-4,m_1-2,2}}+2\|v^{B,0}_2\|_{L^\infty_TH^{m_1-2}_xH^1_z}\|\partial_tv^{B,0}_2\|_{L^2_{T}\mathcal{H}^{k+m_1-4,m_1-3,1}}\right)\\
            &\leq C,
		\end{align*}
        due to $(\ref{fgHm})$ and $(\ref{fgHm1})$, a similar argument can be deduced from $$\|\varrho_2\|_{L^2_T\mathcal{H}^{k+m_1-3,m_1-3,0}}+\|\partial^2_t\varrho_2\|_{L^2_T\mathcal{H}^{k+m_1-5,m_1-5,0}}\leq C.$$
        Similarly, one derives $\|\partial_t\varrho\|_{L^\infty_T\mathcal{H}^{k,m_1-6,0}}\leq C$, it follows from $(\ref{eq:vB21})$ that $\|\partial_tv^{B,2}_1\|_{L^\infty_T\mathcal{H}^{k,m_1-6,2}}\leq C$.
        Moreover, from the argument similar to the derivation $(\ref{rho Linfty})$, we obtain $\|\varrho\|_{L^\infty H^{k,m_1-3-l,l-2}}\leq C$ for $l=2,3,4$.
         Therefore, by a similar argument as deriving $(\ref{gBt})$-$(\ref{gBz})$, we deduce for all $k\in N$ that $(\ref{vB21t})$-$(\ref{vB21z4})$ are valid via $(\ref{eq:vB21})$. Thus, the proof of Lemma $\ref{lem5.6}$ is complete.\\

         Up to now, we have proved the well-posedness of the boundary layer systems given in (\ref{eq:vB02})-(\ref{eq:pB2}). This provides the basis for the proof of Theorem \ref{th inner layer}.\\

        \noindent \textbf{Proof of Theorem \ref{th inner layer}.}
         Combining with the estimates derived in Lemma $\ref{lem5.1}$-Lemma $\ref{lem5.6}$, we immediately obtain Theorem \ref{th inner layer}.  \;\;$\Box$\\

\section*{Acknowledgements}
W. Wang was supported by National Key R\&D Program of China (No. 2023YFA1009200) and NSFC under grant 12471219.
L. Zhao was supported by the National Natural Science Foundation of China(NSFC)[grant number 12501305].

\medskip
\noindent\textbf{Data Availability Statement:}
Data sharing is not applicable to this article as no data sets were generated or analyzed during the current study.

\noindent\textbf{Conflict of Interest:}
The authors declare that they have no conflict of interest.
	\hspace*{\parindent}

\end{document}